\documentclass{amsart}
\usepackage{amsmath,amssymb,graphicx,setspace,verbatim,url,hyperref,listings}
\usepackage{adjustbox}
\usepackage{color}
\usepackage{hyperref}
\usepackage{multicol}
\usepackage{enumitem}
\usepackage{float}
\usepackage{tikz}
\usepackage{tikz-cd}
\usepackage{soul}
\usepackage{forest}
\usetikzlibrary{positioning}
\usepackage{setspace}

\input xy
\xyoption{all}
\newtheorem{prop}{Proposition}[section]
\newtheorem{coro}[prop]{Corollary}
\newtheorem{thm}[prop]{Theorem}
\newtheorem{hypothesis}[prop]{Hypothesis}
\newtheorem{lemma}[prop]{Lemma}

\begin{document}

\title[SGGIs for alternating groups]{The maximal rank of a string group generated by involutions  for alternating groups}

\author[J.~Anzanello]{Jessica Anzanello}
\address{Dipartimento di Matematica e Applicazioni, University of Milano-Bicocca, Via Cozzi~55, 20125 Milano, Italy
}
\email{j.anzanello@campus.unimib.it}

\author[M. E. Fernandes]{Maria Elisa Fernandes}
\address{Maria Elisa Fernandes,
Center for Research and Development in Mathematics and Applications, Department of Mathematics, University of Aveiro, Portugal
}
\email{maria.elisa@ua.pt}

\author[Pablo Spiga]{Pablo Spiga}
\address{Dipartimento di Matematica e Applicazioni, University of Milano-Bicocca, Via Cozzi~55, 20125 Milano, Italy
}
\email{pablo.spiga@unimib.it}

\begin{abstract}
A string group generated by involutions, or SGGI, is a pair $\Gamma=(G, S)$, where $G$ is a group and  $S=\{\rho_0,\ldots, \rho_{r-1}\}$ is an ordered set of involutions generating  $G$ and satisfying the commuting property: 
$$\forall i,j\in\{0,\ldots, r-1\}, \;|i-j|\ne 1\Rightarrow (\rho_i\rho_j)^2=1.$$
When $S$ is an independent set, the rank of $\Gamma$  is the cardinality of $S$. 

We determine an upper bound for the  rank of an SGGI over the alternating group  of degree $n$. Our bound is tight when $n\equiv 0,1,4\pmod 5$.
\end{abstract}


\maketitle

\section{Introduction}\label{sec:intro}

Let $G$ be a group and let $S$ be a subset of $G$. We say that $S$ is an \textbf{\textit{independent generating set}} for $G$  if $G=\langle S\rangle$ and  $s\notin \langle S\setminus\{s\}\rangle$, $\forall s\in S$.

A \textbf{\textit{string group generated by involutions}} or, for short, an \textit{\textbf{SGGI}} is a pair $\Gamma=(G, S)$, where $G$ is a group and  $S=\{\rho_0,\ldots, \rho_{r-1}\}$ is an ordered set of involutions generating  $G$, satisfying the following property\footnote{This is usually referred to as the commuting property.}: 
$$\forall i,j\in\{0,\ldots, r-1\}, \;|i-j|\ne 1\Rightarrow (\rho_i\rho_j)^2=1.$$ To avoid cumbersome notation,  when  $S$ is understood, we simply say that $G$ is an SGGI.
The \textit{\textbf{dual}} of an SGGI is obtained by reversing the order of the string of generators. A group $G$ is an SGGI of \textit{\textbf{rank}} $r$ if $(G,S)$ is an SGGI, for some independent generating set $S$ of cardinality $r$. 

SGGIs play a prominent role in the study of polytopes. In fact, there is a one-to-one correspondence between abstract regular polytopes and string C-groups, which are SGGIs whose generating set $S = \{\rho_0, \ldots, \rho_{r-1}\}$ satisfies the following condition, known as the \emph{intersection property:
\[
\left\langle \rho_i \mid i \in J \right\rangle \cap \left\langle \rho_i \mid i \in K \right\rangle = \left\langle \rho_i \mid i \in J \cap K \right\rangle,
\]
for all subsets $J, K \subseteq \{0, \ldots, r-1\}$. This correspondence is  described with detail in \cite{arp}.} 

 The maximal rank of the symmetric group of degree $n$ is $n-1$ and the bound is attained, for instance, by the Coxeter generators, see~\cite{2000W}. In this paper we show that the maximal rank of an SGGI for the alternating group $\mathrm{Alt}(n)$ of degree $n$ is quite a bit smaller than $n-1$.
\begin{thm}\label{main}
The following holds:
\begin{itemize}
\item the group $\mathrm{Alt}(5)$ admits SGGIs and they all have rank $3$;
\item the groups $\mathrm{Alt}(3)$, $\mathrm{Alt}(4)$, $\mathrm{Alt}(6)$, $\mathrm{Alt}(7)$ and $\mathrm{Alt}(8)$ do not admit SGGIs;
\item for $n\ge 9$,
the maximum rank of an SGGI for $\mathrm{Alt}(n)$ is at most   $\lfloor 3(n-1)/5\rfloor$. Moreover, this bound is attained when $n\equiv 0,1,4\pmod 5$.
\end{itemize}
\end{thm}

In Table~\ref{table:examples}, we present permutation representation graphs of SGGIs for $\mathrm{Alt}(n)$ (see Section~\ref{prem} for the definition of a permutation representation graph). These graphs have rank $\lfloor 3(n-1)/5\rfloor$ when $n\equiv 0,1,4\pmod{5}$ and rank $\lfloor (3n-8)/5\rfloor$ when $n\equiv 2,3\pmod{5}$. The constructions depend on the parity of $n$ modulo $5$.  Moreover, in Table~\ref{table:examples}, we require $n\ge 22$ when $n\equiv 2\pmod 5$ and $n\ge 18$ when $n\equiv 3\pmod 5$. When $n\in \{12,13,17\}$, SGGIs for $\mathrm{Alt}(n)$ of rank $\lfloor (3n-8)/5\rfloor$ have been found with the help of a computer.\footnote{We are using the computer algebra system magma~\cite{magma} in this paper. Some of our arguments require extensive computations. We do not include the code for these computations in this manuscript, but the interested reader can find it in our companion paper submitted to arXiv~\cite{arxiv}.
}

Computational evidence suggests that for $n\equiv 2,3\pmod{5}$, the maximum rank of an SGGI for $\mathrm{Alt}(n)$ is $\lfloor (3n-8)/5\rfloor$. Consequently, Theorem~\ref{main} might be improved for these two congruence classes modulo $5$.  

\begin{table}[htbp]
\begin{adjustbox}{angle=90}
\begin{tabular}{|l|l|}\hline
Conditions&Permutation representation graph\\\hline
&\\
$n\equiv 0\text{ mod } 5$,&
$$ \xymatrix@-0.7pc{  *+[o][F]{}  \ar@{-}[r]^0 & *+[o][F]{}  \ar@{-}[r]^1 & *+[o][F]{}  \ar@{-}[r]^0 & *+[o][F]{}  \ar@{-}[r]^1 & *+[o][F]{}  \ar@{-}[r]^2& *+[o][F]{}  \ar@{-}[r]^3 & *+[o][F]{}  \ar@{=}[r]^2_4 & *+[o][F]{}  \ar@{-}[r]^3 &*+[o][F]{}  \ar@{-}[r]^4 & *+[o][F]{}  \ar@{.}[r] &   *+[o][F]{}  \ar@{-}[r]^{r-3} & *+[o][F]{}  \ar@{-}[r]^{r-2}  & *+[o][F]{}  \ar@{=}[r]^{r-3}_{r-1}  & *+[o][F]{}  \ar@{-}[r]^{r-2}  & *+[o][F]{}  \ar@{-}[r]^{r-1} & *+[o][F]{}  }$$\\
$n\ge 10$&\\\hline
&\\
$ n\equiv 1\text{ mod } 5,$&
$$\xymatrix@-0.7pc{  *+[o][F]{}  \ar@{-}[r]^0 & *+[o][F]{}  \ar@{-}[r]^1 & *+[o][F]{}  \ar@{=}[r]^0_2 & *+[o][F]{}  \ar@{-}[r]^1 & *+[o][F]{}  \ar@{-}[r]^2& *+[o][F]{}  \ar@{-}[r]^3 & *+[o][F]{}  \ar@{-}[r]^4 & *+[o][F]{}  \ar@{=}[r]^3_5 & *+[o][F]{}  \ar@{-}[r]^4 & *+[o][F]{}  \ar@{-}[r]^5& *+[o][F]{}  \ar@{.}[r] &   *+[o][F]{}  \ar@{-}[r]^{r-3} & *+[o][F]{}  \ar@{-}[r]^{r-2}  & *+[o][F]{}  \ar@{=}[r]^{r-3}_{r-1}  & *+[o][F]{}  \ar@{-}[r]^{r-2}  & *+[o][F]{}  \ar@{-}[r]^{r-1} & *+[o][F]{}  }$$\\
$n\ge 11$&\\\hline
&\\
$ n\equiv 2\text{ mod } 5$,  &
$$\xymatrix@-0.7pc{  *+[o][F]{}  \ar@{-}[r]^0 & *+[o][F]{}  \ar@{-}[r]^1 & *+[o][F]{}  \ar@{-}[r]^0 &*+[o][F]{}  \ar@{-}[r]^1 & *+[o][F]{}  \ar@{.}[r] & *+[o][F]{}  \ar@{-}[r]^6 & *+[o][F]{}  \ar@{-}[r]^7 & *+[o][F]{}  \ar@{-}[r]^6 &*+[o][F]{}  \ar@{-}[r]^7 &  *+[o][F]{}  \ar@{-}[r]^8 & *+[o][F]{}  \ar@{-}[r]^9 & *+[o][F]{}  \ar@{=}[r]^8_{10} & *+[o][F]{}  \ar@{-}[r]^9 &*+[o][F]{}  \ar@{-}[r]^{10} & *+[o][F]{}  \ar@{.}[r] &   *+[o][F]{}  \ar@{-}[r]^{r-3} & *+[o][F]{}  \ar@{-}[r]^{r-2}  & *+[o][F]{}  \ar@{=}[r]^{r-3}_{r-1}  & *+[o][F]{}  \ar@{-}[r]^{r-2}  & *+[o][F]{}  \ar@{-}[r]^{r-1} & *+[o][F]{}  }$$\\
$n\geq 22$&\\\hline

&\\
$ n\equiv 3\text{ mod } 5$, &
$$\xymatrix@-0.7pc{  *+[o][F]{}  \ar@{-}[r]^0 & *+[o][F]{}  \ar@{-}[r]^1 & *+[o][F]{}  \ar@{-}[r]^0 &*+[o][F]{}  \ar@{-}[r]^1 & *+[o][F]{}  \ar@{.}[r] & *+[o][F]{}  \ar@{-}[r]^4 & *+[o][F]{}  \ar@{-}[r]^5 & *+[o][F]{}  \ar@{-}[r]^4 &*+[o][F]{}  \ar@{-}[r]^5 &  *+[o][F]{}  \ar@{-}[r]^6 & *+[o][F]{}  \ar@{-}[r]^7 & *+[o][F]{}  \ar@{=}[r]^6_8 & *+[o][F]{}  \ar@{-}[r]^7 &*+[o][F]{}  \ar@{-}[r]^8 & *+[o][F]{}  \ar@{.}[r] &   *+[o][F]{}  \ar@{-}[r]^{r-3} & *+[o][F]{}  \ar@{-}[r]^{r-2}  & *+[o][F]{}  \ar@{=}[r]^{r-3}_{r-1}  & *+[o][F]{}  \ar@{-}[r]^{r-2}  & *+[o][F]{}  \ar@{-}[r]^{r-1} & *+[o][F]{}  }$$\\
$n\geq 18$&\\\hline
 
&\\
$ n\equiv 4\text{ mod } 5$, &
$$\xymatrix@-0.7pc{  *+[o][F]{}  \ar@{-}[r]^0 & *+[o][F]{}  \ar@{-}[r]^1 & *+[o][F]{}  \ar@{-}[r]^0 &*+[o][F]{}  \ar@{-}[r]^1 & *+[o][F]{}  \ar@{-}[r]^2 & *+[o][F]{}  \ar@{-}[r]^3 & *+[o][F]{}  \ar@{-}[r]^2 & *+[o][F]{}  \ar@{-}[r]^3& *+[o][F]{}  \ar@{-}[r]^4 & *+[o][F]{}  \ar@{-}[r]^5 & *+[o][F]{}  \ar@{=}[r]^4_6 & *+[o][F]{}  \ar@{-}[r]^5 &*+[o][F]{}  \ar@{-}[r]^6 & *+[o][F]{}  \ar@{.}[r] &   *+[o][F]{}  \ar@{-}[r]^{r-3} & *+[o][F]{}  \ar@{-}[r]^{r-2}  & *+[o][F]{}  \ar@{=}[r]^{r-3}_{r-1}  & *+[o][F]{}  \ar@{-}[r]^{r-2}  & *+[o][F]{}  \ar@{-}[r]^{r-1} & *+[o][F]{}  } $$\\
$n\geq 14$&\\\hline
\end{tabular}
\end{adjustbox}
\caption{Permutation representation graphs for the upper bound in Theorem~\ref{main}.} \label{table:examples}
\end{table}

\section{Preliminaries, notation and basic results}\label{prem}

\subsection{String groups generated by involutions and their friends}
We collect here some basic information and notation.

In the case that the group $G$ is a permutation group, we introduce an auxiliary gadget for graphically  representing an SGGI. Let  $G$ be a permutation group of degree $n$, and let $\Gamma=(G, \{\rho_0,\ldots,\rho_{r-1}\})$ be an SGGI. The \textit{\textbf{permutation representation graph}} of $\Gamma$ is the $r$-edge-labeled multigraph with vertex set the domain of $G$, and with an edge
 $\{a,\,b\}$ having colour $i$, for each $a,b$  such that $a\ne b$ and $a\rho_i=b$, and for each $i\in\{0,\ldots,r-1\}$.\footnote{When we ignore the edge colours in a permutation representation graph, we obtain what is usually called the Schreier graph of the permutation group.} For instance, when $\Gamma=(\langle (1\,2)(3\,4),(1\,3)(2\,4)\rangle,\{(1\,2)(3\,4),(1\,3)(2\,4),(1\,4)(2\,3)\})$, the permutation representation graph is the complete graph on $4$ vertices, where to each perfect matching is  assigned a distinct color. The permutation representation graph has $n$ vertices and $\sum_{i=0}^{r-1}c_i$ edges, where $c_i$ is the number of cycles of length $2$ in $\rho_i$.

Since some of our arguments are inductive, we need some ad hoc notation for our work. Let $\Gamma=(G,\{\rho_0,\ldots,\rho_{r-1}\})$ be an SGGI and let $i,i_1,\ldots,i_k\in \{0,\ldots,r-1\}$. We let
\begin{align*}
G_{i_1,\ldots,i_k} &=\langle \rho_j \,|\, j \notin \{i_1,\ldots,i_k\}\rangle,&
\Gamma_{i_1,\ldots,i_k} &=(G_{i_1,\ldots,i_k}, \{\rho_j \,|\, j \notin \{i_1,\ldots,i_k\}\}),\\
G_{\{i_1,\ldots,i_k\}} &=\langle\rho_j \,|\, j \in \{i_1,\ldots,i_k\}\rangle,
&\Gamma_{\{i_1,\ldots,i_k\}} &=(G_{\{i_1,\ldots,i_k\}}, \{\rho_j \,|\, j \in \{i_1,\ldots,i_k\}\}),\\
G_{<i} &=\langle\rho_0,\ldots, \rho_{i-1}\rangle,
&\Gamma_{<i} &= (G_{<i},\{ \rho_0,\ldots, \rho_{i-1}\}),\,\,\hbox{ for } i\neq 0,\\
G_{>i} &=\langle \rho_{i+1},\ldots, \rho_{r-1}\rangle,
& \Gamma_{>i} &= (G_{>i},\{ \rho_{i+1},\ldots, \rho_{r-1}\}),\,\,\hbox{ for } i\neq r-1.
\end{align*}

Finally, since our work builds upon some of the results in~\cite{2017CFLM}, we recall the definition of \textit{\textbf{string C-group}}.  A string C-group is an SGGI $(G,\{\rho_0,\ldots,\rho_{r-1}\}) $ satisfying the \textit{\textbf{intersection property}}\footnote{This property is of paramount importance for studying polytopes.}:
$$G_I\cap G_J=G_{I\cap J}\text{ for each }I,J\subseteq \{0,\ldots,r-1\}.$$
When $\Gamma$ is a string C-group, $S$ is an independent generating set for $G$ and hence $\Gamma$ is an SGGI of rank $|S|$. For instance, when $G$ is the symmetric group $\mathrm{Sym}(n)$ and $S=\{(1\,2),(2\,3),\ldots,(n-1\,n)\}$, we see that $(G,S)$ is a string C-group and that the permutation representation graph is the usual Coxeter diagram.

\subsection{Fracture graphs and 2-fracture graphs}\label{sec:fracturegraphs}
Let $\Gamma=(G,\{\rho_0,\,\ldots,\,\rho_{r-1}\})$ be an SGGI, where $G$ is a permutation group with domain $\{1,\,\ldots,\,n\}$.

Here, we define fracture graphs when $G$ is transitive and $G_i$ is intransitive for every $i\in\{0,\ldots,r-1\}$.\footnote{In Section~\ref{sec:5}, we need an analogous definition when $G$ is intransitive. However, we postpone this technical difficulty until it is really necessary.} Therefore, we assume these conditions hold throughout this section and whenever fracture graphs are used.

For each $i$, since $G_i$ is intransitive, $\rho_i$ has at least one cycle (of length 2) containing points from
different $G_i$-orbits. Choosing one such cycle for each $i$, we construct a graph having vertex set $\{1,\ldots,n\}$,
 where each selected cycle defines an edge. The resulting graph, which has $r$ edges, is called a \textit{\textbf{fracture graph}} for $\Gamma$. In general, fracture graphs are  not unique, as they depend on the choice of cycles.

Additionally, under the assumption  that, for each $i$, $\rho_i$ has at least two cycles containing points from
different $G_i$-orbits, then, taking (for each $i$) two $i$-edges between each of these pairs of points, we obtain a graph on $n$ vertices with $2r$ edges that we call a
$2$-\textbf{\textit{fracture graph}} for $\Gamma$. Observe that, when $G$ is contained in the alternating group $\mathrm{Alt}(n)$, each $\rho_i$ does have at least two cycles (of length 2).

In~\cite{2017CFLM}, the authors proved the following two results. Although~\cite{2017CFLM} focuses on string C-groups, none of the proofs in \cite[Section~4]{2017CFLM} rely on the intersection property. Consequently, these results remain valid for SGGIs as well.

\begin{prop}[Proposition 4.12,~\cite{2017CFLM}]\label{dis2F}
Let \(\Gamma = (G, S)\) be an SGGI for a permutation group \( G \) of degree \( n \) admitting at least one $2$-fracture graph. If \( \Gamma \) has no connected $2$-fracture graph, then there exists a $2$-fracture graph having at least one component that is a tree, while all other components are either trees or have at most one cycle (which is an alternating square).
\end{prop}

\begin{coro}[Corollary 4.19, \cite{2017CFLM}]\label{c2F}
Assume $n\geq 9$. Let $\Gamma = (G, S)$ be an SGGI for a permutation group $G$ of degree $n$. If  $\Gamma$ has a connected $2$-fracture graph that is not a tree, then $n$ is even, $\Gamma$ has rank $n/2$ and, up to duality, $\Gamma$ falls into one of the following two cases:
\begin{enumerate}
    \item\label{c2F1} $G$ is permutation isomorphic to $\mathrm{Sym}(2)\times\mathrm{Sym}(n/2)$, and the permutation representation graph is
    $$
    
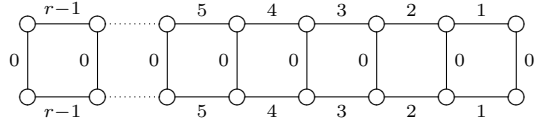
\xymatrix@-1.4pc{
    *+[o][F]{}\ar@{-}[dd]_0\ar@{-}[rr]^{r-1}&&*+[o][F]{}\ar@{-}[dd]_0\ar@{.}[rr]&&*+[o][F]{}\ar@{-}[rr]^5\ar@{-}[dd]_0&&*+[o][F]{}\ar@{-}[rr]^4\ar@{-}[dd]_0&&*+[o][F]{}\ar@{-}[rr]^3\ar@{-}[dd]_0&&*+[o][F]{}\ar@{-}[rr]^2 \ar@{-}[dd]_0 &&*+[o][F]{}   \ar@{-}[dd]^0\ar@{-}[rr]^1&& *+[o][F]{} \ar@{-}[dd]^0\\
    &&&&&&&&&&&&&&\\
    *+[o][F]{}\ar@{-}[rr]_{r-1}*+[o][F]{}&&*+[o][F]{}\ar@{.}[rr]&&*+[o][F]{}\ar@{-}[rr]_5&&*+[o][F]{}\ar@{-}[rr]_4&&*+[o][F]{}\ar@{-}[rr]_3&&*+[o][F]{}\ar@{-}[rr]_2&&*+[o][F]{} \ar@{-}[rr]_1&& *+[o][F]{} }
    $$
    
    \item\label{c2F2}  $G$ is permutation isomorphic to $\mathrm{Sym}(2) \wr \mathrm{Sym}(n/2)$ when $n/2$ is even, and   
    $G$ is permutation isomorphic to $\mathrm{Sym}(2)^{n/2-1} \rtimes \mathrm{Sym}(n/2)$ (which is a subgroup of index $2$ in $\mathrm{Sym}(2) \wr \mathrm{Sym}(n/2)$) when $n/2$ is odd. The permutation representation graph is 
    $$
    
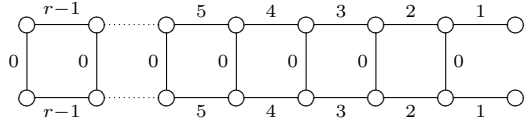
\xymatrix@-1.4pc{
    *+[o][F]{}\ar@{-}[dd]_0\ar@{-}[rr]^{r-1}&&*+[o][F]{}\ar@{-}[dd]_0\ar@{.}[rr]&&*+[o][F]{}\ar@{-}[rr]^5\ar@{-}[dd]_0&&*+[o][F]{}\ar@{-}[rr]^4\ar@{-}[dd]_0&&*+[o][F]{}\ar@{-}[rr]^3\ar@{-}[dd]_0&&*+[o][F]{}\ar@{-}[rr]^2 \ar@{-}[dd]_0 &&*+[o][F]{}   \ar@{-}[dd]^0\ar@{-}[rr]^1&& *+[o][F]{} \\
    &&&&&&&&&&&&&&\\
    *+[o][F]{}\ar@{-}[rr]_{r-1}*+[o][F]{}&&*+[o][F]{}\ar@{.}[rr]&&*+[o][F]{}\ar@{-}[rr]_5&&*+[o][F]{}\ar@{-}[rr]_4&&*+[o][F]{}\ar@{-}[rr]_3&&*+[o][F]{}\ar@{-}[rr]_2&&*+[o][F]{} \ar@{-}[rr]_1&& *+[o][F]{} }
    $$
\end{enumerate}
\end{coro}

\subsection{Splits and perfect splits}\label{sec:splits}

Let $G$ be a permutation group of degree $n$, and let $\Gamma=(G,S)$ be an SGGI with $S=\{\rho_0,\ldots,\rho_{r-1}\}$. 
Suppose $\Gamma$ has a fracture graph.

Suppose that, for some $i\in\{0, \ldots, r-1\}$, 
\begin{itemize}
\item $G_i$ has exactly two orbits; observe that under this condition there exists a partition of $\{1,\ldots,n\}$ into two sets $O_1$ and $O_2$, of sizes $n_1$ and $n_2=n - n_1$ such that $\rho_i$ is the unique permutation of $S$ swapping elements in different parts of this partition of  $\{1,\ldots,n\}$; and
\item there is exactly one pair of points $(a,b) \in O_1\times O_2$ such that $a\rho_i=b$.
\end{itemize}
Then, $\{a,b\}$ is called a \textbf{\textit{split}} or \textbf{\textit{$i$-split}} of $\Gamma$. We refer to $i$ as the \textbf{\textit{label}} of a split.

For $j\in \{0,\ldots,r-1\}\setminus\{i\}$, let $\alpha_j$ be the restriction of $\rho_j$ to $O_1$, and let $\beta_j$ be the restriction of $\rho_j$ to $O_2$.
If, up to duality, $\alpha_j = 1$ for each $j\in \{i+1,\ldots, r-1\}$ and
$\beta_j = 1$ for each $j\in \{0,\ldots, i-1\}$,
 then we say that the $i$-split $\{a,b\}$ is \textbf{\emph{perfect}}. The following illustrates the situation when we have a perfect split having label $i$.
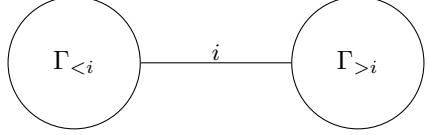
\begin{center}
\begin{tikzpicture}[
  vertex/.style={circle, draw, minimum size=50pt, inner sep=0pt}, 
  edge label/.style={above, midway, fill=white, inner sep=1pt, font=\small} 
]

\node[vertex] (GammaLeft) at (0,0) {$\Gamma_{<i}$};
\node[vertex] (GammaRight) [right=of GammaLeft, xshift=1cm] {$\Gamma_{>i}$}; 

\draw (GammaLeft) -- (GammaRight) node[edge label] {$i$};

\end{tikzpicture}
\end{center}

 If $j$ is the label of an edge adjacent to a $i$-split, then $j\in\{i-1, i+1\}$. Hence if  $\Gamma_{<i}$ (resp.  $\Gamma_{>i}$)  is nontrivial, then it must have a pendent edge with label $i-1$ (resp. $i+1$).

\section{SGGIs admitting a $2$-fracture graph}\label{2f}
In Proposition~4.9 of~\cite{2024CFL}, the authors classify string C-groups $\Gamma = (G, S)$ of rank at least $(n-1)/2$ that admit a $2$-fracture graph, where $G$ is a permutation group of degree $n$. Since their focus was on string C-groups, they excluded SGGIs that do not satisfy the intersection property from their classification. However, the proof of Proposition~4.9 in~\cite{2024CFL} inherently provides a classification of SGGIs of a permutation group of degree $n$ with rank at least $(n-1)/2$. Therefore, in the following, we revisit some of the steps in their proof to recover the permutation representation graphs of the SGGIs omitted from their classification.\footnote{That is, SGGIs that are not string C-groups.}

\begin{prop}\label{frac2}
Let $G$ be a transitive permutation group of degree $n$, and let $\Gamma= (G,\{\rho_0, \ldots, \rho_{r-1}\})$ be an  SGGI of rank $r$.
If $\Gamma$ has a $2$-fracture graph, then $r\le n/2$. Moreover,  if   $r\geq (n-1)/2$, then  up to duality $\Gamma$ has a permutation representation graph isomorphic to one of the graphs in Table~$\ref{T2F}$.
\end{prop}
\begin{proof}
Assume first that $n\geq 9$. Let $\mathcal{G}$ be the permutation representation graph of $\Gamma$. By assumption, $\Gamma$ has a 2-fracture graph. If $\Gamma$ admits a connected $2$-fracture graph that is not a tree, then Corollary~\ref{c2F} shows that $n$ is even, $r=n/2$ and $\mathcal{G}$ is one of the graphs described in parts~\eqref{c2F1} and~\eqref{c2F2}. We have reported these two graphs in the first two rows of Table~\ref{T2F}. Therefore, we may suppose that either $\Gamma$ admits no connected $2$-fracture graph, or that each connected $2$-fracture graph is a tree.

In this case, from Proposition~\ref{dis2F}, we deduce that $\Gamma$ admits a 2-fracture graph having at least one component
that is a tree, while all other components are either trees or have at most one cycle.\footnote{We are including here the possibility that a $2$-fracture graph is a tree.} Let $n_1,\ldots,n_\kappa$ be the cardinalities of the connected components of this $2$-fracture graph. From Euler's formula, the number of edges in each connected component is at most $n_j$, and there exists a $j$ such that the number of edges is exactly $n_j-1$. Since the total number of edges in the 2-fracture graph is $2r$, we deduce that
$$2r\le \sum_{i=1}^\kappa n_i-1=n-1.$$
This shows that $r\le n/2$ in all cases, and hence it remains to prove that, if $r\ge (n-1)/2$, then $\mathcal{G}$ is isomorphic to one of the graphs in Table~$\ref{T2F}$.
As $r\ge (n-1)/2$, we deduce that $r=(n-1)/2$.

When $n\leq 8$, the proof in~\cite[Proposition~4.9]{2024CFL} is via a computer search; a similar computer search shows that the possible graphs $\mathcal{G}$ are listed in Table~\ref{T2F}. Therefore, we may suppose $n\ge 9$.

At this point, in the proof of~\cite[Proposition~4.9]{2024CFL}, the authors consider a $2$-fracture graph $\mathcal{F}$ with a minimal number of squares, denoted by $s$. The proof is then divided into two cases: $s > 0$ and $s = 0$. When $s > 0$, following the proof of~\cite[Proposition~4.9]{2024CFL}, we see that the graph $\mathcal{G}$ corresponds to graph (3) in Table~\ref{T2F}.  

Now, suppose $s = 0$, meaning that $\mathcal{F}$ is a tree. If every pair of adjacent edges in $\mathcal{F}$ carries consecutive labels, then, by the proof of~\cite[Proposition~4.9]{2024CFL}, the graph obtained is a linear graph whose sequence of labels is given by:
$$
(0,1,0,1,2,3,2,3,\dots,r-2,r-1,r-2,r-1).$$
The resulting graph $\mathcal{G}$ corresponds to graph (4) in Table~\ref{T2F}.

Finally, suppose that $\mathcal{F}$ contains adjacent edges with nonconsecutive labels. Then, by the proof of~\cite[Proposition~4.9]{2024CFL}, the resulting graphs $\mathcal{G}$ correspond to graphs (5), (6), (7), (8), (9), (10), and (11) in Table~\ref{T2F}.
\end{proof}
\begin{table}[htbp]

\[\begin{array}{|l|}
\hline
(1) n\geq 4:\; \xymatrix@-1.5pc{
*+[o][F]{}\ar@{-}[dd]_0\ar@{-}[rr]^{r-1}&&*+[o][F]{}\ar@{-}[dd]_0\ar@{.}[rr]&&*+[o][F]{}\ar@{-}[rr]^4\ar@{-}[dd]_0&&*+[o][F]{}\ar@{-}[rr]^3\ar@{-}[dd]_0&&*+[o][F]{}\ar@{-}[rr]^2 \ar@{-}[dd]_0 &&*+[o][F]{}   \ar@{-}[dd]^0\ar@{-}[rr]^1&& *+[o][F]{} \ar@{-}[dd]^0\\
&&&&&&&&&&&&&&\\
*+[o][F]{}\ar@{-}[rr]_{r-1}*+[o][F]{}&&*+[o][F]{}\ar@{.}[rr]&&*+[o][F]{}\ar@{-}[rr]_4&&*+[o][F]{}\ar@{-}[rr]_3&&*+[o][F]{}\ar@{-}[rr]_2&&*+[o][F]{} \ar@{-}[rr]_1&& *+[o][F]{} }\\
(2) n\geq 6:\; \xymatrix@-1.5pc{
*+[o][F]{}\ar@{-}[dd]_0\ar@{-}[rr]^{r-1}&&*+[o][F]{}\ar@{-}[dd]_0\ar@{.}[rr]&&*+[o][F]{}\ar@{-}[rr]^5\ar@{-}[dd]_0&&*+[o][F]{}\ar@{-}[rr]^4\ar@{-}[dd]_0&&*+[o][F]{}\ar@{-}[rr]^3\ar@{-}[dd]_0&&*+[o][F]{}\ar@{-}[rr]^2 \ar@{-}[dd]_0 &&*+[o][F]{}   \ar@{-}[dd]^0\ar@{-}[rr]^1&& *+[o][F]{} \\
&&&&&&&&&&&&&&\\
*+[o][F]{}\ar@{-}[rr]_{r-1}*+[o][F]{}&&*+[o][F]{}\ar@{.}[rr]&&*+[o][F]{}\ar@{-}[rr]_5&&*+[o][F]{}\ar@{-}[rr]_4&&*+[o][F]{}\ar@{-}[rr]_3&&*+[o][F]{}\ar@{-}[rr]_2&&*+[o][F]{} \ar@{-}[rr]_1&& *+[o][F]{} }\\

(3) n\geq 7:\;  \xymatrix@-0.8pc{  *+[o][F]{} \ar@{-}[r]^1 & *+[o][F]{}  \ar@{=}[r]^2_0 & *+[o][F]{} \ar@{-}[r]^1 & *+[o][F]{} \ar@{-}[r]^2 & *+[o][F]{} \ar@{.}[r] & *+[o][F]{}  \ar@{-}[r]^{r-4}& *+[o][F]{}  \ar@{-}[r]^{r-3}  & *+[o][F]{}  \ar@{-}[r]^{r-2}  & *+[o][F]{}  \ar@{-}[r]^{r-1}  & *+[o][F]{}   &\\
 & & &*+[o][F]{} \ar@{-}[r]_2  \ar@{-}[u]_0&  *+[o][F]{} \ar@{-}[u]_0 \ar@{.}[r] & *+[o][F]{}  \ar@{-}[r]_{r-4} \ar@{-}[u]_0& *+[o][F]{}  \ar@{-}[r]_{r-3}\ar@{-}[u]_0 &  *+[o][F]{}  \ar@{-}[r]_{r-2}\ar@{-}[u]_0 &*+[o][F]{}  \ar@{-}[r]_{r-1}\ar@{-}[u]_0&  *+[o][F]{} \ar@{-}[u]_0  }  \\

(4) n\geq 5:\; \xymatrix@-0.8pc{  *+[o][F]{}  \ar@{-}[r]^0 & *+[o][F]{}  \ar@{-}[r]^1 & *+[o][F]{}  \ar@{-}[r]^0 & *+[o][F]{}  \ar@{-}[r]^1 &*+[o][F]{}  \ar@{.}[r] &  *+[o][F]{}  \ar@{-}[r]^{r-2} & *+[o][F]{}  \ar@{-}[r]^{r-1}  & *+[o][F]{}  \ar@{-}[r]^{r-2}  & *+[o][F]{}  \ar@{-}[r]^{r-1} & *+[o][F]{} } \\

(5) n\geq 7:\; \xymatrix@-0.45pc{  *+[o][F]{} \ar@{-}[r]^1 & *+[o][F]{}  \ar@{-}[r]^0 & *+[o][F]{} \ar@{-}[r]^1 & *+[o][F]{} \ar@{-}[r]^2 & *+[o][F]{} \ar@{.}[r]   & *+[o][F]{}  \ar@{-}[r]^{r-2}  & *+[o][F]{}  \ar@{-}[r]^{r-1}  & *+[o][F]{}   &\\
 & & &*+[o][F]{} \ar@{-}[r]_2  \ar@{-}[u]_0&  *+[o][F]{} \ar@{-}[u]_0 \ar@{.}[r]  &  *+[o][F]{}  \ar@{-}[r]_{r-2}\ar@{-}[u]_0 &*+[o][F]{}  \ar@{-}[r]_{r-1}\ar@{-}[u]_0&  *+[o][F]{} \ar@{-}[u]_0  } \\

(6) n\geq 7:\;  \xymatrix@-0.8pc{  *+[o][F]{} \ar@{-}[r]^1 & *+[o][F]{}  \ar@{-}[r]^2 & *+[o][F]{} \ar@{-}[r]^1 & *+[o][F]{} \ar@{-}[r]^2 & *+[o][F]{} \ar@{.}[r] & *+[o][F]{}  \ar@{-}[r]^{r-4}& *+[o][F]{}  \ar@{-}[r]^{r-3}  & *+[o][F]{}  \ar@{-}[r]^{r-2}  & *+[o][F]{}  \ar@{-}[r]^{r-1}  & *+[o][F]{}   &\\
 & & &*+[o][F]{} \ar@{-}[r]_2  \ar@{-}[u]_0&  *+[o][F]{} \ar@{-}[u]_0 \ar@{.}[r] & *+[o][F]{}  \ar@{-}[r]_{r-4} \ar@{-}[u]_0& *+[o][F]{}  \ar@{-}[r]_{r-3}\ar@{-}[u]_0 &  *+[o][F]{}  \ar@{-}[r]_{r-2}\ar@{-}[u]_0 &*+[o][F]{}  \ar@{-}[r]_{r-1}\ar@{-}[u]_0&  *+[o][F]{} \ar@{-}[u]_0  } \\

(7) n\geq 9:\; \xymatrix@-0.45pc{ *+[o][F]{}  \ar@{-}[r]^0 & *+[o][F]{} \ar@{-}[r]^1 & *+[o][F]{} \ar@{-}[r]^2 & *+[o][F]{} \ar@{.}[r] & *+[o][F]{}  \ar@{-}[r]^{i-2}   & *+[o][F]{}  \ar@{-}[r]^{i-1}  & *+[o][F]{}  \ar@{-}[r]^i& *+[o][F]{}  \ar@{-}[r]^{i+1}& *+[o][F]{}  \ar@{-}[r]^{i+2}&*+[o][F]{}  \ar@{.}[r]  & *+[o][F]{}  \ar@{-}[r]^{r-1}  & *+[o][F]{}   \\
*+[o][F]{}  \ar@{-}[r]_0 \ar@{-}[u]_i & *+[o][F]{} \ar@{-}[r]_1 \ar@{-}[u]_i & *+[o][F]{} \ar@{-}[r]_2  \ar@{-}[u]_i&  *+[o][F]{} \ar@{-}[u]_i \ar@{.}[r] & *+[o][F]{}  \ar@{-}[r]_{i-2} \ar@{-}[u]_i& *+[o][F]{}   \ar@{-}[u]_i&     &  *+[o][F]{}  \ar@{-}[r]_{i+1}\ar@{-}[u]_{i-1}& *+[o][F]{}  \ar@{-}[r]_{i+2}\ar@{-}[u]_{i-1} &*+[o][F]{} \ar@{.}[r]\ar@{-}[u]_{i-1}&*+[o][F]{}  \ar@{-}[r]_{r-1}\ar@{-}[u]_{i-1}&  *+[o][F]{} \ar@{-}[u]_{i-1} } \\

(8) n\geq 6:\; \xymatrix@-0.48pc{ *+[o][F]{}  \ar@{-}[r]^0& *+[o][F]{}  \ar@{-}[r]^1 &*+[o][F]{}  \ar@{-}[r]^0 & *+[o][F]{} \ar@{-}[r]^1 & *+[o][F]{} \ar@{.}[r]&*+[o][F]{}  \ar@{-}[r]^{i-3} &*+[o][F]{}  \ar@{-}[r]^{i-2} & *+[o][F]{} \ar@{-}[r]^{i-1} & *+[o][F]{}   \ar@{-}[r]^i & *+[o][F]{}  \ar@{-}[r]^{i+1}& *+[o][F]{}  \ar@{.}[r] &  *+[o][F]{}  \ar@{-}[r]^{r-1}  & *+[o][F]{}  \\
&& && && & &*+[o][F]{}  \ar@{-}[r]_i& *+[o][F]{}  \ar@{-}[r]_{i+1} \ar@{-}[u]_{i-1}& *+[o][F]{}  \ar@{.}[r]\ar@{-}[u]_{i-1} & *+[o][F]{}  \ar@{-}[r]_{r-1}\ar@{-}[u]_{i-1}&  *+[o][F]{} \ar@{-}[u]_{i-1}  }\\

(9) n\geq 6:\; \xymatrix@-0.48pc{ *+[o][F]{}  \ar@{-}[r]^0& *+[o][F]{}  \ar@{-}[r]^1 &*+[o][F]{}  \ar@{-}[r]^0 & *+[o][F]{} \ar@{-}[r]^1 & *+[o][F]{} \ar@{.}[r]&*+[o][F]{}  \ar@{-}[r]^{r-5} &*+[o][F]{}  \ar@{-}[r]^{r-4} & *+[o][F]{} \ar@{-}[r]^{r-3} & *+[o][F]{}   \ar@{-}[r]^{r-2} & *+[o][F]{}  \ar@{-}[r]^{r-1}& *+[o][F]{} &\\
&& && && & && *+[o][F]{}  \ar@{-}[r]_{r-1} \ar@{-}[u]_{r-3}& *+[o][F]{}  \ar@{-}[r]_{r-2}\ar@{-}[u]_{r-3}& *+[o][F]{}   } \\

(10) n\geq 6:\; \xymatrix@-0.48pc{ *+[o][F]{}  \ar@{-}[r]^0& *+[o][F]{}  \ar@{-}[r]^1 &*+[o][F]{}  \ar@{-}[r]^0 & *+[o][F]{} \ar@{-}[r]^1 & *+[o][F]{} \ar@{.}[r]&*+[o][F]{}  \ar@{-}[r]^{r-5} &*+[o][F]{}  \ar@{-}[r]^{r-4} & *+[o][F]{} \ar@{-}[r]^{r-3} & *+[o][F]{}   \ar@{-}[r]^{r-2} & *+[o][F]{}  \ar@{-}[r]^{r-1}& *+[o][F]{} \ar@{-}[r]^{r-2} &*+[o][F]{} \\
&& && && & && *+[o][F]{}  \ar@{-}[r]_{r-1} \ar@{-}[u]_{r-3}& *+[o][F]{} \ar@{-}[u]_{r-3}&   }\\

(11) n\geq 7:\;  \xymatrix@-0.45pc{ *+[o][F]{}  \ar@{-}[r]^0 & *+[o][F]{} \ar@{-}[r]^1 & *+[o][F]{} \ar@{-}[r]  \ar@{-}[r]^2 & *+[o][F]{}  \ar@{.}[r] & *+[o][F]{}  \ar@{-}[r]^{r-2}  & *+[o][F]{}  \ar@{-}[r]^{r-1}  & *+[o][F]{}  &\\
&*+[o][F]{}  \ar@{-}[r]_1 & *+[o][F]{} \ar@{-}[u]_0  \ar@{-}[r]_2  & *+[o][F]{}  \ar@{.}[r] \ar@{-}[u]_0 &  *+[o][F]{}  \ar@{-}[r]_{r-2}\ar@{-}[u]_0&*+[o][F]{}  \ar@{-}[r]_{r-1}\ar@{-}[u]_0&  *+[o][F]{} \ar@{-}[u]_0  } \\

(12)\,\xymatrix@-0.8pc{   *+[o][F]{}  \ar@{=}[r]^2_0 & *+[o][F]{} \ar@{-}[r]^1 & *+[o][F]{} \ar@{-}[r]^2 & *+[o][F]{} \\
  & *+[o][F]{} \ar@{-}[r]^1 &*+[o][F]{} \ar@{-}[r]_2  \ar@{-}[u]_0&  *+[o][F]{} \ar@{-}[u]_0 }\quad
   
 (13)\,\xymatrix@-0.8pc{   *+[o][F]{}  \ar@{=}[r]^2_0 & *+[o][F]{} \ar@{-}[r]^1 & *+[o][F]{} \ar@{-}[r]^2 & *+[o][F]{}& \\
  & &*+[o][F]{} \ar@{-}[r]_2  \ar@{-}[u]_0&  *+[o][F]{} \ar@{-}[u]_0\ar@{-}[r]^1 & *+[o][F]{}  }\quad

   (14)\, \xymatrix@-0.8pc{   *+[o][F]{}  \ar@{=}[r]^2_0 & *+[o][F]{} \ar@{-}[r]^1 & *+[o][F]{} \ar@{-}[r]^2 & *+[o][F]{}\ar@{-}[r]^1 & *+[o][F]{}  \\
  & &*+[o][F]{} \ar@{-}[r]_2  \ar@{-}[u]_0&  *+[o][F]{} \ar@{-}[u]_0 &}\\

(15)\, \xymatrix@-0.4pc{   *+[o][F]{}  \ar@{-}[r]^1 & *+[o][F]{}  \ar@{-}[r]^0 & *+[o][F]{}  \ar@{-}[r]^1 & *+[o][F]{}   &\\
   &  *+[o][F]{}  \ar@{-}[r]_0\ar@{-}[u]_2 &*+[o][F]{}  \ar@{-}[r]_1\ar@{-}[u]_2&  *+[o][F]{}   } \quad

  (16)\, \xymatrix@-0.8pc{  *+[o][F]{} \ar@{-}[r]^1 & *+[o][F]{}  \ar@{-}[r]^2 & *+[o][F]{} \ar@{-}[r]^1 & *+[o][F]{} \ar@{-}[r]^2 & *+[o][F]{} \\
 & & &*+[o][F]{} \ar@{-}[r]_2  \ar@{-}[u]_0&  *+[o][F]{} \ar@{-}[u]_0} \\

 (17)\, \xymatrix@-0.45pc{ *+[o][F]{}  \ar@{-}[r]^0 &   *+[o][F]{}  \ar@{-}[r]^1 & *+[o][F]{}  \ar@{-}[r]^2 & *+[o][F]{}   &\\
 &  &  *+[o][F]{}  \ar@{-}[r]_2\ar@{-}[u]_0 &*+[o][F]{}  \ar@{-}[r]_1\ar@{-}[u]_0&  *+[o][F]{} } \quad

 (18)\, \xymatrix@-0.45pc{ *+[o][F]{}  \ar@{-}[r]^0 &   *+[o][F]{}  \ar@{-}[r]^1 & *+[o][F]{}  \ar@{-}[r]^2 & *+[o][F]{}   \ar@{-}[r]^1 & *+[o][F]{} \\
 &  &  *+[o][F]{}  \ar@{-}[r]_2\ar@{-}[u]_0 &*+[o][F]{} \ar@{-}[u]_0& } \quad
 
 (19)\, \xymatrix@-0.45pc{ *+[o][F]{}  \ar@{-}[r]^1 \ar@{-}[d]^3 &   *+[o][F]{}  \ar@{-}[r]^0\ar@{-}[d]^3 & *+[o][F]{} \ar@{-}[d]^3 \\
 *+[o][F]{}  \ar@{-}[r]_1 &   *+[o][F]{} \ar@{-}[d]_2 \ar@{-}[r]^0 & *+[o][F]{}\ar@{-}[d]^2  \\
&    *+[o][F]{}  \ar@{-}[r]_0&*+[o][F]{}} \\
   
\hline
\end{array}\]
\caption{Permutation representation graphs relevant for Proposition~\ref{frac2}.} \label{T2F}
\end{table}

\section{SGGIs having splits: all of which are perfect}\label{sec:4}

\begin{hypothesis}\label{hyp}{\rm
Let $G$ be a permutation group of degree $n$, and let $\Gamma = (G, S)$ be an SGGI of rank $r$ with $S = \{\rho_0, \ldots, \rho_{r-1}\}$. Throughout this section\footnote{With the only exception of Proposition~\ref{split-prim}.}, we assume the following conditions.
\begin{itemize}
    \item $G$ is a transitive subgroup of $\mathrm{Alt}(n)$.
    \item For each $i \in \{0, \ldots, r-1\}$, the subgroup $G_i$ is intransitive. These two conditions are equivalent to the existence of a fracture graph for $\Gamma$.
    \item The SGGI $\Gamma$ does not admit a $2$-fracture graph. Consequently, by  Section~\ref{sec:splits}, $\Gamma$ admits at least one split. Let $s$ denote the number of splits in $\Gamma$, so in particular, $s \geq 1$.
    \item Every split of $\Gamma$ is a perfect split.
\end{itemize}
}
\end{hypothesis}

\begin{lemma}\label{uniquePS}
Assume Hypothesis~$\ref{hyp}$. If  $s=1$, then $r\leq(n-1)/2$. 
\end{lemma} 
\begin{proof}
Let $x$ be the label of the unique split of $\Gamma$ and let $r_1$ and $r_2$ be the ranks of the SGGIs $\Gamma_{<x}$ and $\Gamma_{>x}$, respectively. Let $X_1$ and $X_2$ be  the $G_x$-orbits that are fixed pointwise by $G_{>x}$ and $G_{<x}$, respectively, and let $n_1$ and $n_2$ denote their cardinalities.

Since $\Gamma$ has a unique split and this split is perfect, both $\Gamma_{<x}$ and $\Gamma_{>x}$ admit a 2-fracture graph. Furthermore, the groups $G_{<x}$ and $G_{>x}$ are  transitive of degree $n_1$ and $n_2$, respectively. Thus, by Proposition~\ref{frac2}, we have  
\[
r_1 \leq \frac{n_1}{2} \quad \text{and} \quad r_2 \leq \frac{n_2}{2}.
\]

Up to duality, we may suppose $x \ne r-1$.  Assume also $x\ne 0$. Since a split does not belong to a square, the permutation representation graph of $\Gamma_{<x}$ or of $\Gamma_{>x}$ has a pendent edge labeled $x-1$ or $x+1$, respectively. If $r_1= n_1/2$ or $r_2=n_2/2$, then Proposition~\ref{frac2} shows that the permutation representation graph of $\Gamma_{<x}$ or of $\Gamma_{>x}$ is one of the graphs in the first two lines of Table~\ref{T2F}, however neither of these graphs has a pendent edge of label $x-1$ or $x+1$. Therefore,
\[
r_1 \le\frac{n_1-1}{2} \quad \text{and} \quad r_2\le\frac{n_2-1}{2}.
\]
In particular, to conclude the proof, it suffices to exclude the possibility that $r_1=(n_1-1)/2$ and $r_2=(n_2-1)/2$.
 
 If $r_1= \frac{n_1 - 1}{2}$ and $r_2= \frac{n_2 - 1}{2}$, then, by Proposition~\ref{frac2}, $\Gamma_{<x}$ and $\Gamma_{>x}$ correspond to one of the graphs in Table~\ref{T2F}.
The graphs in Table~\ref{T2F} that contain a pendent edge having label $x-1$ or $x+1$  are (4), (8), (9), (10), (11), (17), and (18). We now consider each of these possibilities  in turn.  In this case-by-case analysis, we must consider each graph in Table~\ref{T2F} with a pendent edge, as well as its dual, by relabeling the set of edge labels using either the sequence \((0, \ldots, x-1)\) or \((x+1, \ldots, r-1)\).

As $G$ consists of even permutations, $\rho_x$ has at least two cycles of length $2$, and hence there is a pair of distinct vertices $(u, v)\in (X_1\times X_1)\cup (X_2\times X_2)$   such that $u\rho_x=v$. Without loss of generality, suppose that $(u, v)\in X_1\times X_1$. If $\Gamma_{<x}$ is represented by the graph (4) with $r_1=x$, then there is only one possibility for the $x$-edge $(u,v)$, which is as follows.
$$\xymatrix@-0.7pc{  *+[o][F]{}  \ar@{-}[r]^0 & *+[o][F]{}  \ar@{-}[r]^1 & *+[o][F]{}  \ar@{-}[r]^0 & *+[o][F]{}  \ar@{-}[r]^1 &*+[o][F]{}  \ar@{.}[r] &  *+[o][F]{}  \ar@{-}[r]^{x-2} & *+[o][F]{}  \ar@{-}[r]^{x-1}  & *+[o][F]{u}  \ar@{=}[r]^{x-2}_x  & *+[o][F]{v}  \ar@{-}[r]^{x-1} & *+[o][F]{} }$$
But then $\Gamma$ has a perfect split with label $x-2$, contradicting that $\Gamma$ has a unique split and that the label of this split is $x$.

Suppose that $\Gamma_{<x}$ is the graph (8) with the relabeling of the edges given by $(0,1,\ldots, r_1-1)\mapsto (x-1,x-2,\ldots,0)$. Then, the $x$-edge $(u,v)$ must be as follows. 
$$
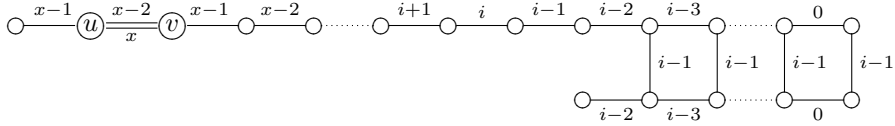
\xymatrix@-0.38pc{ *+[o][F]{}  \ar@{-}[r]^{x-1}& *+[o][F]{u}  \ar@{=}[r]^{x-2}_x &*+[o][F]{v}  \ar@{-}[r]^{x-1} & *+[o][F]{} \ar@{-}[r]^{x-2} & *+[o][F]{} \ar@{.}[r]&*+[o][F]{}  \ar@{-}[r]^{i+1} &*+[o][F]{}  \ar@{-}[r]^i & *+[o][F]{} \ar@{-}[r]^{i-1} & *+[o][F]{}   \ar@{-}[r]^{i-2} & *+[o][F]{}  \ar@{-}[r]^{i-3}& *+[o][F]{}  \ar@{.}[r] &  *+[o][F]{}  \ar@{-}[r]^0  & *+[o][F]{}  \\
&& && && & &*+[o][F]{}  \ar@{-}[r]_{i-2}& *+[o][F]{}  \ar@{-}[r]_{i-3} \ar@{-}[u]_{i-1}& *+[o][F]{}  \ar@{.}[r]\ar@{-}[u]_{i-1} & *+[o][F]{}  \ar@{-}[r]_0\ar@{-}[u]_{i-1}&  *+[o][F]{} \ar@{-}[u]_{i-1}  }$$
Then, we get the same contradiction as before. For the graphs (9) and (10), we get an entirely analogous contradiction.

Now, consider the graph (11) with the relabeling $(0,1,\ldots, r_1-1)\mapsto (x-1,x-2,\ldots,0)$. It is not possible to add the $x$-edge $\{u,v\}$ to  this graph without breaking 
the commuting property of the generators. The same happens when we consider graphs (17) and (18).

Consequently, $$r=r_1+r_2+1\leq\frac{n_1+n_2-3}{2}+1=(n-1)/2.$$

If $x=0$, then $\Gamma_{>0}$ has a 2-fracture graph. As before, a case-by-case analysis shows that $\Gamma_{>0}$ cannot be one of the graphs in Table~\ref{T2F}, leading to $r-1=r_2\leq \frac{n_2-2}{2}$.
Since in this case $r_1=0$ and $n_1=1$, it follows that $r_1+r_2\leq\frac{n_1+n_2-3}{2}$, as desired. 
\end{proof}

\begin{prop}\label{1notPS}
Assume Hypothesis~$\ref{hyp}$.
If  $i$ is the label of a perfect split of $\Gamma$, then $i\notin\{1,r-2\}$.
\end{prop}
\begin{proof}
Up to duality, we may suppose $i=1$. 
Let $X$ and $Y$ be the two $G_1$-orbits.
In this case, $\rho_0$ fixes  one of the two $G_1$-orbits pointwise, let it be $Y$, and it is the only permutation acting non-trivially on $X$. This implies that $\rho_0$ is a transposition, a contradiction. 
\end{proof}

We include here a result from~\cite{2024CFL} which does not require Hypothesis~\ref{hyp}.
\begin{prop}[Proposition 5.1, \cite{2024CFL}]\label{split-prim}
If $G$ is transitive and $\Gamma$ has a perfect $i$-split, then $G$ is primitive.
\end{prop}

\begin{prop}\label{3PS}
Assume Hypothesis~$\ref{hyp}$. If $i-1,\,i$ and $i+1$ are labels of perfect splits, then either $\Gamma_{<i-1}$ or $\Gamma_{>i+1}$ has a $2$-fracture graph.
\end{prop}
\begin{proof}
Let $X$ be the $G_{i-1}$-orbit that is fixed by $G_{>i-1}$ pointwise and let $Y$ be the $G_{i+1}$-orbit that is fixed by $G_{<i+1}$ pointwise.
As $G$ consists of even permutations, $\rho_i$ acts non-trivially in one of the two sets $X$ or $Y$. 

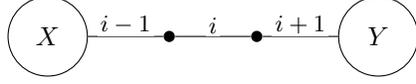
\begin{center}

\begin{tikzpicture}[
  vertex/.style={circle, draw, minimum size=30pt, inner sep=0pt}, 
  point/.style={circle, fill, inner sep=1.5pt}, 
  edge label/.style={above, midway, fill=white, inner sep=1pt, font=\small} 
]
\node[vertex] (X) at (0,0) {$X$};
\node[point] (P1) [right=of X] {};
\node[point] (P2) [right=of P1] {};
\node[vertex] (Y) [right=of P2] {$Y$};

\draw (X) -- (P1) node[edge label] {$i-1$};
\draw (P1) -- (P2) node[edge label] {$i$};
\draw (P2) -- (Y) node[edge label] {$i+1$};
\end{tikzpicture}
\end{center}

Up to duality, we may suppose that  $\rho_i$ acts non-trivially on $X$. Then, as it commutes with all the elements of $G_{<i-1}$, $G_{<i-1}$ is imprimitive. Hence, by Proposition~\ref{split-prim}, $\Gamma_{<i-1}$ has no perfect splits. As all splits of $\Gamma$ are perfect, $\Gamma_{<i-1}$ has a 2-fracture graph.
\end{proof}

\begin{coro}\label{4PS}
Assume Hypothesis~$\ref{hyp}$.
Then $\Gamma$ cannot have four consecutive perfect splits.
\end{coro}
\begin{proof}
If $i-1,\,i,\, i+1,\,i+2$ are labels of perfect splits, then neither $\Gamma_{>i+1}$ nor $\Gamma_{<i}$ has a $2$-fracture graph, contradicting Proposition~\ref{3PS}.
\end{proof}

\begin{prop}\label{maxS}
Assume Hypothesis~$\ref{hyp}$. Then $s\leq 2r/3$.
\end{prop}

\begin{proof}
Let $u=(u_i)_{i=0}^{r-1}$ be the sequence defined by  
\[
u_i = \begin{cases}  
1, & \text{if } i \text{ is the label of a perfect split;} \\  
0, & \text{otherwise.}  
\end{cases}
\]
Thus, $s$ is the number of ones in the sequence $u$.  

By Proposition~\ref{1notPS},  an edge with label $1$ or $r-2$ cannot be a perfect split, meaning that $u_1 = u_{r-2} = 0$). Furthermore, by Corollary~\ref{4PS}, the sequence $u$ cannot contain four consecutive ones.

Suppose first that $u$ does not contain three consecutive ones.  
By Proposition~\ref{1notPS}, the following sequence maximizes the number of ones:  
\[
(1,0, \underbrace{1,1,0,\ldots, 1,1,0}_{1,1,0 \text{ is repeated } k \text{ times}}, \underbrace{1,\ldots,1}_{x \text{ times}},0,1),
\]
where  
\[
x =  
\begin{cases}  
2, & \text{if } r \equiv 0 \pmod{3}, \\  
0, & \text{if } r \equiv 1 \pmod{3}, \\  
1, & \text{if } r \equiv 2 \pmod{3}.  
\end{cases}
\]
In all cases, we obtain $s \leq \frac{2r}{3}$.  

Now suppose that $u$ contains a block of three consecutive ones, say at positions $i-1, i, i+1$.  
By Proposition~\ref{3PS}, up to duality, we may assume that $u_k = 0$ for all $k < i-1$.  
If there are no other triples of consecutive ones, then the following sequence maximizes the number of ones:  
\[
(0,0,1,1,1,0, \underbrace{1,1,0,\ldots, 1,1,0}_{1,1,0 \text{ is repeated } k \text{ times}}, \underbrace{1,\ldots,1}_{x \text{ times}},0,1),
\]
where again $x$ depends on $r \pmod{3}$. As shown earlier, this gives $s < \frac{2r}{3}$.  

If there is a second triple of consecutive ones, then by Proposition~\ref{3PS} $u$ must be constantly zero after this triple.  
Thus, as before, we conclude that $s < \frac{2r}{3}$.
\end{proof}

We are finally ready to prove the main result of this section.
\begin{prop}\label{casePS}
Assume Hypothesis~$\ref{hyp}$. Then, $r\leq (2n+s-2)/4$ and $r\leq \lfloor 3(n-1)/5\rfloor$.
\end{prop}
\begin{proof}
Let $\{i_k \mid k = 1, \ldots, s\}$ be the set of labels of the perfect splits of $\Gamma$ with $i_1<\cdots<i_s$.
Now, let $X_0$ be the set of points fixed by $G_{>i_1}$, let $X_s$ be the set of points fixed by $G_{<i_s}$, and for $k \in \{1, \ldots, s-1\}$, let $X_k$ be the set of points fixed by both $G_{>i_{k+1}}$ and $G_{<i_k}$.  
Additionally, let $r_0 = i_1$, which is the rank of $G_{<i_1}$, $r_s = r - i_s - 1$, the rank of $G_{>i_s}$, and, for $k \in \{1, \ldots, s-1\}$, let $r_k = i_{k+1} - i_{k} - 1$, the rank of $G_{\{i_k+1, \ldots, i_{k+1}-1\}}$. Finally, let $n_k = |X_k|$ for $k \in \{0, \ldots, s\}$. The following diagram may help illustrate our argument.
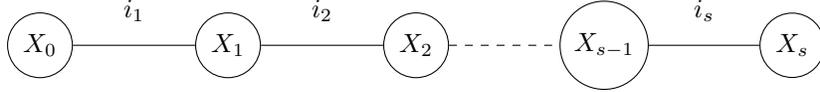
\begin{center}
\begin{tikzpicture}
    \tikzstyle{mynode} = [draw, circle, minimum size=12mm, inner sep=0pt]

    \node[mynode] (X1) at (0, 0) {$X_0$};
    \node[mynode] (X2) at (2.5, 0) {$X_1$};
    \node[mynode] (X3) at (5, 0) {$X_2$};
    \node[mynode] (Xs-1) at (7.5, 0) {$X_{s-1}$};
    \node[mynode] (Xs) at (10, 0) {$X_s$};

    \draw (X1) -- (X2) 
        node[midway, above=2mm] {\textit{$i_1$}};
    \draw (X2) -- (X3) 
        node[midway, above=2mm] {\textit{$i_2$}};
    \draw[dashed] (X3) -- (Xs-1); 
    \draw (Xs-1) -- (Xs) 
        node[midway, above=2mm] {\textit{$i_s$}};
\end{tikzpicture}
\end{center}

By construction, $G_{>i_s}$ and $G_{<i_1}$ are either trivial or transitive SGGIs having a 2-fracture graph.
Then, by Proposition~\ref{frac2},   $r_0\leq n_0/2$ and $r_s\leq n_s/2$. 
If $r_0= n_0/2$, then, from Proposition~\ref{frac2}, we deduce that  $\Gamma_{<i_1}$ has  permutation representation graph corresponding to graphs (1) or (2) in Table~\ref{T2F}. However, the split with label $i_1$ must be incident to a pendent edge with label $i_1-1$. Since neither (1) nor (2) has such an edge, we deduce  $r_0\leq  (n_0-1)/2$.  Similarly,  $r_s\leq (n_s-1)/2$.

Consider the SGGI arising from the permutation group $G^{(k)}=G_{\{i_k+1,\ldots,i_{k+1}-1\}}$ acting on $X_k$, for each $k\in \{1,\ldots s-1\}$. 
Assume that $G^{(k)}$ is intransitive. Observe that this happens when two components of the permutation representation graph of $$\Gamma_{i_k+1,\ldots,i_{k+1}-1}$$ are connected by an edge with label $i_k$ and an edge with label $i_{k+1}$. In this case, the group  $G_{\{i_k,\ldots,i_{k+1}-1\}}$\footnote{which has one more generator}  acts transitively on $n_k+1$ points and has a 2-fracture graph. Moreover, this graph has two pendent edges, one with label $i_{k}$ and another with label $i_{k+1}-1$. The latter is the edge incident to the split having label $i_{k+1}$. Hence, the permutation representation graph of $\Gamma_{\{i_k,\ldots,i_{k+1}-1\}}$ is neither graph (1) nor graph (2)
of Table~\ref{T2F}. We conclude  that  the rank of  $G_{\{i_k,\ldots,i_{k+1}-1\}}$, which is $r_k+1$, is less than $(n_k+1)/2$. Thus $r_k+1\leq \frac{(n_k+1)-1}{2}$, which gives $r_k\leq \frac{n_k-2}{2}$.

Let us also denote $G_{<i_1}$ and $G_{>i_s}$ by  $G^{(0)}$ and $G^{(s)}$. Let $k\in \{0,\ldots,s-1\}$.
Suppose that $G^{(k)}$ and  $G^{(k+1)}$ are both transitive and let $S_k$ and $S_{k+1}$ be the ordered generating sets of $G^{(k)}$ and  $G^{(k+1)}$, respectively.
Consider the SGGI $\Phi=(F, S_k\cup\{\rho_k\}\cup S_{k+1})$, where $F=\langle S_k,\rho_k,S_{k+1}\rangle$. Observe that $F$ is a permutation group of degree $n_k+n_{k+1}$.
By construction, this SGGI has exactly one split and this split is perfect.
Hence, by Lemma~\ref{uniquePS} applied to $\Phi$, $r_k+r_{k+1}\leq (n_k+n_{k+1}-3)/2$. 

The previous two paragraphs show that at most $\lceil (s+1)/2\rceil$ ranks $r_k$ of the set $\{r_0,\ldots, r_s\}$ attain the upper bound $(n_k-1)/2$. 
Consequently, 
\begin{align*}
r&=r_0+\cdots+r_s+s\leq \frac{(n_0+\cdots+n_s)-(s+1)-s/2}{2}+s\\
&= \frac{n-s-s/2-1}{2}+s. 
\end{align*}
Thus $r\leq \frac{2n+s-2}{4}$.

Now, the rest of the proof follows from Proposition~\ref{maxS}.
\end{proof}


\section{SGGIs having non-perfect splits}\label{sec:5}
This section is the core of our argument for proving Theorem~\ref{main} and, as in Section~\ref{sec:4}, we borrow some of the ideas from~\cite{2017CFLM,2024CFL}.

\subsection{General notation and basic results}\label{sec:generalnotation}
\begin{hypothesis}\label{hypNP}{\rm
Let $G$ be a permutation group of degree $n$, and let $\Gamma = (G, S)$ be an SGGI of rank $r$ with $S = \{\rho_0, \ldots, \rho_{r-1}\}$. Throughout this section, we assume the following conditions.
\begin{itemize}
    \item $G$ is a transitive subgroup of $\mathrm{Alt}(n)$.
    \item For each $i \in \{0, \ldots, r-1\}$, the subgroup $G_i$ is intransitive. These two conditions are equivalent to the existence of a fracture graph for $\Gamma$.
    \item The graph $\Gamma$ does not admit a $2$-fracture graph. Consequently, by  Section~\ref{sec:splits}, $\Gamma$ admits at least one split. 
    \item There exists $i\in \{0,\ldots,r-1\}$ such that  $\Gamma$ has a split with label $i$ that is not perfect.
\end{itemize}
}
\end{hypothesis}

In what follows, let $\mathcal{G}$ be the permutation representation graph of $\Gamma$. Let $\{a, b\}$ be the split with label $i$. Let $O_a$ and $O_b$ be the $G_i$-orbits containing $a$ and $b$, respectively, and let $n_A$ and $n_B$ denote the sizes of $O_a$ and $O_b$. From the definition of split in Section~\ref{sec:splits}, $O_a\cup O_b$ equals the domain of $G$.  Let $A$ and $B$ be the permutation groups induced by $G_i$  in its action on $O_a$ and $O_b$.\footnote{Observe that $A$ and $B$ are quotients of $G_i$, whereas when the split is perfect, $A$ and $B$ are also subgroups of $G_i$.} Notice that $A$ and $B$ are  SGGIs. For each $l\in \{0,\ldots,r-1\}\setminus\{i\}$, let $\rho_l=\alpha_l\beta_l$, where $\alpha_l$ and $\beta_l$ are the restrictions of $\rho_l$ to  $O_a$ and $O_b$, respectively. 
Then 
\begin{eqnarray*}
A& = \langle \alpha_l\, |\, l\in \{0,\ldots, r-1\}\setminus\{i\}\rangle,\\
B& = \langle \beta_l\, |\, l\in \{0,\ldots, r-1\}\setminus\{i\}\rangle.
\end{eqnarray*} 
Let 
\begin{eqnarray*}
J_A&= \{ l  \in \{0,\ldots, r-1\}\setminus\{i\}\mid  \alpha_l \textrm{ is not the identity}\},\\
J_B&= \{ l  \in \{0,\ldots, r-1\}\setminus\{i\}\mid  \beta_l \textrm{ is not the identity}\}.
\end{eqnarray*} 
 In Subsection 5.1 of \cite{2017CFLM}, the authors prove several results where the intersection property is not required.
 We recall here some of these results.
  
 \begin{prop}[Proposition 5.1, \cite{2017CFLM}] \label{CCD}
If  $A$ is primitive, then the set $J_A$ is an interval.  The same result holds for $B$.
\end{prop}

The main result of Subsection 5.1 of \cite{2017CFLM} gives an upper bound for the rank of $\Gamma$, when $A$ and $B$ are both imprimitive.
 
 \begin{prop}[Proposition 5.7, \cite{2017CFLM}] \label{bothimp}
If $A$ and $B$ are both imprimitive, then  $r \leq  (n-1)/2$.
\end{prop}

The following result is one of the most important tools in the proofs of this section.

\begin{prop}[Proposition 5.18, \cite{2017CFLM}]\label{path}
If $e$ is an $f$-edge of $\mathcal{G}$ not in an alternating square, then any path (not containing another $f$-edge) from $e$ to an edge with label $l$, with $l<f$ (resp. $l>f$), contains all labels between $l$ and $f$. 
Moreover, there exists a path from $e$ to an $l$-edge, that is fixed by $G_{<l}$  (resp. $G_{>l}$).
\end{prop}
Observe that Proposition~\ref{path} applies immediately to the edge $e=\{a,b\}$, because a split edge does not belong to any alternating square.

Let us begin with the specific case where $\{0, r-1\} \subseteq J_A$, which, by Proposition~\ref{path}, implies that $J_A = \{0, \ldots, r-1\} \setminus \{i\}$.

\begin{prop}\label{allinone} 
If $\{0,r-1\}\subseteq J_A$, then $r\leq  (n-1)/2$.
\end{prop}
\begin{proof}
Since $\{0,r-1\}\subseteq J_A$, we have $0\ne i\ne r-1$ and hence $J_A = \{0, \ldots, r-1\} \setminus \{i\}$ is not an interval. Thus,
by Proposition~\ref{CCD},  $A$ is imprimitive. If $B$ is imprimitive, then the result follows from Proposition~\ref{bothimp}. Hence, for the rest of the proof, we may suppose that $J_B$ is an interval and that $B$ is primitive. We consider two cases separately.

\smallskip

\noindent\textsc{Case $G_{<i}$ is transitive on $O_a$.}

\smallskip

\noindent  

Let $l > i$. As $\rho_l$ commutes with $G_{<i}$ and $G_{<i}$ is transitive on $O_a$, it follows that $\rho_l$ is fixed-point-free on $O_a$. Since $\{a, b\}$ is a split and splits cannot be contained in alternating squares, $\rho_{i+1}$ is the unique involution in $S$ with label $>i$ not fixing $a$. We conclude that $l = i+1$, $i+1 = r - 1$, and $G_i = G_{r-2}$.

Since $G_{i+1} = G_{r-1}$ is intransitive, $\rho_{i+1} = \rho_{r-1}$ acts non-trivially on $O_b$, that is, $r - 1 \in J_B$.  
As $J_B$ is an interval and $J_B \subseteq \{0, \ldots, r-1\} \setminus \{i\} = \{0, \ldots, r-3\} \cup \{r-1\}$, we deduce that $J_B = \{r-1\}$. Hence, $n_B = 2$ and $n_A = n - 2$.

From the fact that $\rho_{r-1}$ acts fixed-point-freely on $O_a$ and commutes with $G_{<r-2}$, we deduce that $G_{<r-2}$ preserves a system of imprimitivity with $n_A/2$ blocks of cardinality $2$. Hence, we obtain an embedding $G_{<r-2} \leq \mathrm{Sym}(2) \mathrm{wr}\, \mathrm{Sym}(n_A/2)$, and $\rho_{r-1}$ swaps all pairs of vertices within the blocks of size $2$.

We claim that, for every $j < i$, the group $\langle \rho_0, \ldots, \rho_{j-1}, \rho_{j+1}, \ldots, \rho_{i-1} \rangle$ acts intransitively on the system of imprimitivity preserved by $G_{<i}$. Assume the contrary, and let $j < i$ be such that $\langle \rho_0, \ldots, \rho_{j-1}, \rho_{j+1}, \ldots, \rho_{i-1} \rangle$ is transitive on the system of imprimitivity. Now, the group
\[
\langle \rho_0, \ldots, \rho_{j-1}, \rho_{j+1}, \rho_{j+2}, \ldots, \rho_{r-4}, \rho_{r-3}, \rho_{r-1} \rangle
\]
acts transitively on the system of imprimitivity of $G_{<i}$ and, via $\rho_{r-1}$, acts transitively on the two points within each block. Therefore, 
\[
\langle \rho_0, \ldots, \rho_{j-1}, \rho_{j+1}, \ldots, \rho_{r-3}, \rho_{r-1} \rangle
\]
has orbits $O_a$ and $O_b$ on the domain of $G$. Finally, since $\rho_{r-2}$ swaps $a$ and $b$, we deduce that 
\[
\langle \rho_0, \ldots, \rho_{j-1}, \rho_{j+1}, \ldots, \rho_{r-3}, \rho_{r-2}, \rho_{r-1} \rangle = G_j
\]
is transitive, which contradicts Hypothesis~\ref{hypNP}.

The previous claim shows that the set $\{\rho_0, \dots, \rho_{r-3}\}$ generates the block action independently. 
 The maximal size of an independent set of degree $n_A/2$ is at most $n_A/2 - 1$ \cite{2000W}. Therefore,
\[
\left|\left\{\rho_0, \dots, \rho_{r-3}\right\}\right| \leq \frac{n_A}{2} - 1 \leq \frac{n - 2}{2} - 1.
\]
Assume now that $r > \frac{n - 1}{2}$. Then it follows that
\[
\left|\left\{\rho_0, \dots, \rho_{r-3}\right\}\right| > \frac{n - 1}{2} - 2.
\]
Combining these inequalities forces $n_A = n - 2$ and
\[
\left|\left\{\rho_0, \dots, \rho_{r-3}\right\}\right| = \frac{n - 2}{2} - 1.
\]
The independence of the set $\{\rho_0, \dots, \rho_{r-3}\}$, together with the string condition, implies that 
the permutation representation graph of the group generated by $\rho_0, \dots, \rho_{r-3}$ in its action on the blocks is as in the following picture.
\[
\xymatrix@-1.3pc{
*+[o][F]{}\ar@{-}[rr]^0 
&&
*+[o][F]{} \ar@{-}[rr]^{1}
&&
*+[o][F]{} \ar@{.}[rr] 
&&
*+[o][F]{}\ar@{-}[rr]^{r-4}
&&
*+[o][F]{}\ar@{-}[rr]^{r-3} &&
*+[o][F]{}\ar@{}[rr]^{}&&}
\]
From~\cite{2000W}, this implies that $r-2 \leq n_A/2 - 1$, leading to $r \leq n/2$. If $r<n/2$, then the result follows. Assume then $r=n/2$, that is, $r=n_A/2+1$.

Observe that $\langle \rho_0, \ldots, \rho_{r-4} \rangle$ fixes one block, namely the block containing the point $a$, and acts as the full symmetric group of degree $r-3$ on the remaining blocks. As $G \le \mathrm{Alt}(n)$, $\rho_i = \rho_{r-2}$ cannot be a transposition, meaning that it acts non-trivially on $O_a$. Since $\rho_i = \rho_{r-2}$ commutes with $\langle \rho_0, \ldots, \rho_{r-4} \rangle$, we deduce that $\rho_i$ fixes setwise $r-3$ blocks. As there are $r-2$ blocks in total, we conclude that $\rho_i = \rho_{r-2}$ also fixes setwise the remaining block, that is, $\rho_i$ lies in the kernel of the action on the blocks. However, this leads to the contradiction that either $\rho_{r-2}$ or $\rho_{r-1}$ is an odd permutation, as illustrated in the following graph.
$$
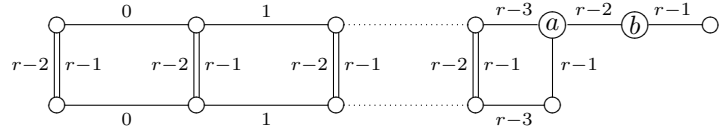
\xymatrix@-1.4pc{
*+[o][F]{}\ar@{=}[dd]_{r-2}^{r-1}\ar@{-}[rrrr]^0&&&&*+[o][F]{}\ar@{=}[dd]_{r-2}^{r-1}\ar@{-}[rrrr]^1&&&&*+[o][F]{}\ar@{=}[dd]_{r-2}^{r-1}\ar@{.}[rrrr]&&&&*+[o][F]{}\ar@{-}[rr]^{r-3}\ar@{=}[dd]_{r-2}^{r-1}&&*+[o][F]{a}\ar@{-}[rr]^{r-2}\ar@{-}[dd]^{r-1}&&*+[o][F]{b}\ar@{-}[rr]^{r-1}&&*+[o][F]{} \\
&&&&&&&&&&&&&&&&&&\\
*+[o][F]{}\ar@{-}[rrrr]_0&&&&*+[o][F]{}\ar@{-}[rrrr]_1&&&&*+[o][F]{}\ar@{.}[rrrr]&&&&*+[o][F]{}\ar@{-}[rr]_{r-3}&&*+[o][F]{}&&&& }$$

\smallskip

\noindent\textsc{ Case $G_{<i}$ is intransitive on $O_a$.}

\smallskip

\noindent Recall that $J_A=\{0,\ldots,i-1\}\cup\{i+1,\ldots,r-1\}$. Let $A_{<i}=\langle \alpha_l\mid l<i\rangle$ and $A_{>i}=\langle\alpha_l\mid l>i\rangle$. Observe that $A_{<i}$ and $G_{<i}$ induce the same action on $O_a$. We have $A=\langle A_{<i},A_{>i}\rangle$ and $A_{<i}$ commutes with $A_{>i}$.  Thus $A_{<i}\unlhd A$. Since $G_{<i}$ is intransitive on $O_a$ so is $A_{<i}$. Hence, the orbits of $A_{<i}$ on $O_a$ give rise to a system of imprimitivity for $A$. In particular, there exists an embedding $A\le \mathrm{Sym}(k)\mathrm{wr}\mathrm{Sym}(m)$, with 
  $G_{<i}$ fixing all the blocks setwise. 
  
  Let $\beta$ denote the block containing the vertex $a$. Since $\{a,b\}$ is an $i$-split, the only permutation in $S$ with label $>i$ not fixing $a$  is $\rho_{i+1}$.
Therefore, $\rho_{i+1}$ is the only permutation in $S$ with label $>i$ that does not fix the block $\beta$. Hence, $\beta\rho_{i+1}$ is a  $G_{<i}$-orbit disjoint from  $\beta$. By Proposition~\ref{path}, there exists a path $\mathcal{P}_1$, whose vertices are in $\beta$, from a $0$-edge to the vertex $a$, containing all labels from  $0$ to $i-1$. Similarly, there is a path $\mathcal{P}_2$ in the block $\beta\rho_{i+1}$ containing all labels from $0$ to $i-1$.

\smallskip

\noindent\textsc{Subcase: there is an edge with label $l>i$ inside $\beta\rho_{i+1}$ (or $\beta$).}

\smallskip 

\noindent Since $\beta\rho_{i+1}$ is a block, this implies that $\rho_l$ fixes setwise $\beta\rho_{i+1}$. Since $G_{<i}$ is transitive on $\beta \rho_{i+1}$ and since $G_{<i}$ centralizes $\rho_l$, we deduce that $\rho_l$ acts fixed-point-freely on $\beta\rho_{i+1}$. Moreover, since a split cannot belong to an alternating square, we deduce that $\rho_{i+2}$ is the only permutation that acts non-trivially on $a\rho_{i+1}\in\beta\rho_{i+1}$. This implies $l=i+2$ and $r-1=i+2$. From the definition of the system of imprimitivity, $G_{<i}=\langle\rho_0,\ldots,\rho_{r-3}\rangle$ fixes setwise $\beta\rho_{i+1}$. Since $\rho_l=\rho_{r-1}$ also fixes setwise $\beta\rho_{i+1}$, we deduce that $G_{i+1}$ fixes setwise $\beta\rho_{i+1}$.
From this, we deduce that 
$\beta$ and $\beta\rho_{i+1}$ are  the only two blocks.

From the previous paragraph, $\rho_{i+2}=\rho_{r-1}$ centralizes $G_{<i}$ and acts fixed-point-freely on $\beta\rho_{i+1}$. Therefore, we obtain a homomorphism of $G_{<i}$ into $\mathrm{Sym}(2)\mathrm{wr} \mathrm{Sym}(k/2)$.
Moreover, as $G_{i+2}$ is intransitive, $\rho_{i+2}=\rho_{r-1}$ acts non-trivially on $O_b$.
Since $J_B$ is an interval and $i=r-3$, we have $J_B=\{r-2,r-1\}$ and $n_B\geq 3$.
In addition, arguing as in the previous case, since $G_j$ is intransitive for every $j$, the set $\{\rho_0,\ldots,\rho_{i-1}\}$ generates the block action (for the block system with $k/2$ blocks of size $2$) independently.
Hence, by~\cite{2000W}, $i=r-3\leq k/2-1=n_A/4-1\leq (n-3)/4-1$. This gives $r\leq (n-1)/2$ (notice that $n$ is necessarily greater than or equal to $4+3=7$).

\smallskip

\noindent\textsc{Subcase: $\beta$ and $\beta\rho_{i+1}$ do not contain edges with label $l>i$.}

\smallskip 

\noindent  By Proposition~\ref{path},  there is a path $\mathcal{P}_3$ in $O_a$ containing edges with all labels from $r-1$ to $i+1$. Using the commuting property between $\rho_{i-1}$ and $\rho_l$ with $l>i$, applying $\rho_{i-1}$ to $\mathcal{P}_3$, we obtain a second path  $\mathcal{P}_4$ disjoint from $\mathcal{P}_3$ and containing edges with all labels from $r-1$ to $i+1$.
Therefore, $\mathcal{P}_1$, $\mathcal{P}_2$, $\mathcal{P}_3$ and $\mathcal{P}_4$ have no edge in common, as shown in the following figure.
$$
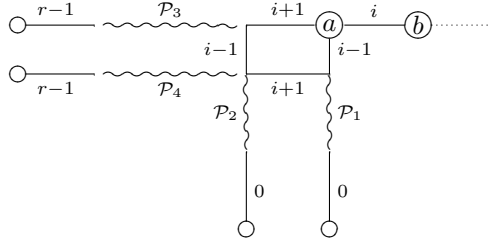
\xymatrix@-1.3pc{  *+[o][F]{} \ar@{-}[rr]^{r-1}&&*+[.][F]{} \ar@{~}[rrrr]^{\mathcal{P}_3}   &&&&*+[.][F]{}\ar@{-}[d]_{i-1}\ar@{-}[rr]^{i+1} && *+[o][F]{a}\ar@{-}[d]^{i-1} \ar@{-}[rr]^i&& *+[o][F]{b}\ar@{.}[rr]&&\\
   *+[o][F]{} \ar@{-}[rr]_{r-1}&&*+[.][F]{} \ar@{~}[rrrr]_{\mathcal{P}_4}   &&&&*+[.][F]{}\ar@{~}[dd]_{\mathcal{P}_2}\ar@{-}[rr]_{i+1}&&*+[.][F]{}\ar@{~}[dd]^{\mathcal{P}_1}&&&&\\
  &&&&&&&& &&&&\\
  &&&&&&*+[.][F]{} &&*+[.][F]{}  &&&&\\
   &&&&&&&& &&&&\\
  &&&&&&*+[o][F]{}\ar@{-}[uu]_{0}&& *+[o][F]{} \ar@{-}[uu]_{0}&&&&
  }
 $$
 The inequality $2(r-1)\leq n- n_B$ holds. Equality is achieved if and only if there are no points outside the paths $\mathcal{P}_1$, $\mathcal{P}_2$, $\mathcal{P}_3$, and $\mathcal{P}_4$. This implies that $i\in \{1, r-2\}$, and that the permutation representation graph of $A$ is one of the following.
 $$
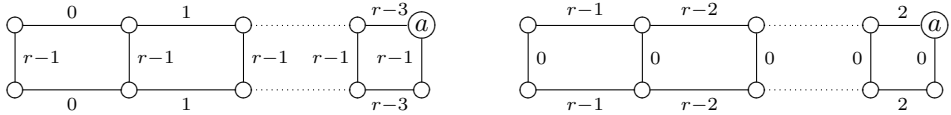
\xymatrix@-1.6pc{
*+[o][F]{}\ar@{-}[dd]^{r-1}\ar@{-}[rrrr]^0&&&&*+[o][F]{}\ar@{-}[dd]^{r-1}\ar@{-}[rrrr]^1&&&&*+[o][F]{}\ar@{-}[dd]^{r-1}\ar@{.}[rrrr]&&&&*+[o][F]{}\ar@{-}[rr]^{r-3}\ar@{-}[dd]_{r-1}&&*+[o][F]{a}\ar@{-}[dd]_{r-1}\\
&&&&&&&&&&&&&&&&\\
*+[o][F]{}\ar@{-}[rrrr]_0&&&&*+[o][F]{}\ar@{-}[rrrr]_1&&&&*+[o][F]{}\ar@{.}[rrrr]&&&&*+[o][F]{}\ar@{-}[rr]_{r-3}&&*+[o][F]{}&&}
 \quad \xymatrix@-1.6pc{
*+[o][F]{}\ar@{-}[dd]^0\ar@{-}[rrrr]^{r-1}&&&&*+[o][F]{}\ar@{-}[dd]^0\ar@{-}[rrrr]^{r-2}&&&&*+[o][F]{}\ar@{-}[dd]^0\ar@{.}[rrrr]&&&&*+[o][F]{}\ar@{-}[rr]^{2}\ar@{-}[dd]_0&&*+[o][F]{a}\ar@{-}[dd]_0\\
&&&&&&&&&&&&&&&&\\
*+[o][F]{}\ar@{-}[rrrr]_{r-1}&&&&*+[o][F]{}\ar@{-}[rrrr]_{r-2}&&&&*+[o][F]{}\ar@{.}[rrrr]&&&&*+[o][F]{}\ar@{-}[rr]_{2}&&*+[o][F]{}&&}
$$

Assume first that $n_B \geq 3$. Then $2(r-1)\leq n-3$, and hence $r\leq (n-1)/2$.

Now assume $n_B = 1$. Suppose first that the equality $2(r-1)= n - n_B = n - 1$ holds. Since all $G_j$s are intransitive, $\rho_i$ must be the transposition swapping $a$ and $b$, contradicting the assumption $\rho_i \in \mathrm{Alt}(n)$. Thus, $2(r-1) < n - 1$ and therefore $r \leq n/2$. If $n$ is odd, then $r \leq (n-1)/2$. If $n$ is even, then $A$ has odd degree and hence cannot have a block system with blocks of even size or with an even number of blocks. Since $n_B = 1$ and $G \leq \mathrm{Alt}(n)$, we deduce that $A$ consists of even permutations. We conclude that there are at least two vertices not in the paths $\mathcal{P}_1$, $\mathcal{P}_2$, $\mathcal{P}_3$, and $\mathcal{P}_4$. Hence, $2(r-1)+1 \leq n - 2$, and therefore $r \leq \frac{n - 1}{2}$.%
\footnote{This part of the argument follows the proof of Proposition 5.25 in \cite{2017CFLM}.}

Finally, assume $n_B = 2$. If $2(r-1) < n - n_B$, then $r \leq (n - 1)/2$. Suppose instead that $2(r - 1) = n - n_B = n - 2$.
Assume that $A$ has the permutation representation graph on the left above, that is, $i = r - 2$. Then the graph $\mathcal{G}$ is as follows:
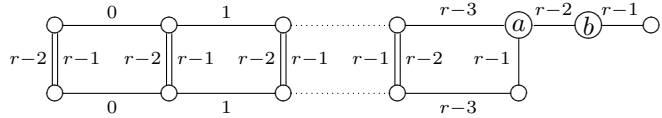
\[
\xymatrix@-1.6pc{
*+[o][F]{}\ar@{=}[dd]^{r-1}_{r-2}\ar@{-}[rrrr]^0&&&&*+[o][F]{}\ar@{=}[dd]^{r-1}_{r-2}\ar@{-}[rrrr]^1&&&&*+[o][F]{}\ar@{=}[dd]^{r-1}_{r-2}\ar@{.}[rrrr]&&&&*+[o][F]{}\ar@{-}[rrrr]^{r-3}\ar@{=}[dd]_{r-1}^{r-2}&&&&*+[o][F]{a}\ar@{-}[dd]_{r-1}\ar@{-}[rr]^{r-2}&&*+[o][F]{b}\ar@{-}[rr]^{r-1}&&*+[o][F]{}\\
&&&&&&&&&&&&&&&&&&&&&&\\
*+[o][F]{}\ar@{-}[rrrr]_0&&&&*+[o][F]{}\ar@{-}[rrrr]_1&&&&*+[o][F]{}\ar@{.}[rrrr]&&&&*+[o][F]{}\ar@{-}[rrrr]_{r-3}&&&&*+[o][F]{}&&&&&&}
\]
Then either $\rho_{r-2}$ or $\rho_{r-1}$ is odd, a contradiction. A similar conclusion is reached if we consider the other possible permutation representation graph for $A$.
\end{proof}
\begin{hypothesis}\label{hypNP2}{\rm
Observe that, if $A$ and $B$ are both imprimitive, then Theorem~\ref{main} follows by Proposition~\ref{bothimp}. In particular, from Proposition~\ref{CCD}, we may suppose that either   $J_A$ or $J_B$ is an interval. Without loss of generality, we may assume that $J_B$ is an interval.

Up to duality and applying Proposition~\ref{path} with a path from an $(r-1)$-edge in $O_b$ to the $i$-split $\{a,b\}$, we may suppose that
$$ J_B=\{i+1,\ldots, r-1\}.$$
Now, applying Proposition~\ref{path} to the indices in $J_A$, we deduce that
$$J_A=\{0,\ldots, h\}\setminus\{i\},$$
 for some $h\in\{0,\ldots,r-1\}$. When $i\ne 0$, from Proposition~\ref{allinone}, we have either $r\le (n-1)/2$ or $h\ne r-1$. Similarly, when $i=0$, from~\cite[Proposition~5.19]{2017CFLM}, we have either $r\le (n-1)/2$ or $h\ne r-1$. In either case, for the rest of the argument, we may suppose that $h\ne r-1$.
 Thus
 $$i<h<r-1.$$
}
\end{hypothesis}
With this setting, Propositions 5.20–5.23 of \cite{2017CFLM} establish that the permutation representation graph of $\Gamma$ satisfies the following properties:
\begin{itemize}
    \item If $i = 0$, there exists a path fixed by $G_{>h}$ that contains the $0$-split and that has at least two edges for each label $j \in \{1, \ldots, h\}$.
        
    \item If $i \neq 0$ and $h \neq i+1$, the group $G_{<i}$ acts intransitively on $O_a$, and $A$ embeds into a wreath product where the blocks correspond to the $G_{<i}$-orbits. Moreover, if $r > (n-1)/2$, then no generator $\rho_j$ with $j > i$ fixes the $\Gamma_{<i}$-orbits. In this case, the permutation representation graph of $G_{<i}$ contains, in each connected component, a path with all labels in the set $\{0, \ldots, i-1\}$, thus there are at least two such paths, $\mathcal{P}_1$ and $\mathcal{P}_2$. Additionally, there exist at least two disjoint paths, $\mathcal{P}_3$ and $\mathcal{P}_4$, each containing all labels in the set $\{i+1, \ldots, h\}$, and each of these paths intersects $\mathcal{P}_1 \cup \mathcal{P}_2$ in at most two points.
\end{itemize}
Further details about these paths can be found in the cited propositions, which together yield the following result.

\begin{prop}\label{h}
If $r> (n-1)/2$, then, for each $\ell\in \{i+1,\ldots,h\}$, we have  $\ell\leq \frac{n-|X_\ell|-1}{2}$, where $X_\ell= \{1,\ldots,n\}\setminus \mathrm{Fix}(G_{>\ell})$. 
Moreover, if  $\ell\neq h$ and $\ell=\frac{n-|X_\ell|-1}{2}$, then $i=0$ and $\Gamma_{<\ell}$ has a permutation representation graph containing the $0$-split, that is a path, fixed by $G_{>ell}$, with the sequence of labels $[\ell-1,\ell-2,\ldots, 1,0,1, \ldots, \ell-2, \ell-1]$.
If $h=\frac{n-|X_h|-1}{2}$, then $\Gamma_{<h+1}$ has permutation representation graph isomorphic to one of the graphs in  Table~$\ref{Th}$, with eventually some additional edges with label $i$ (the unique $i$-edge represented is the split $\{a,b\}$).
\end{prop}
Together with the notation in Hypotheses~\ref{hypNP} and~\ref{hypNP2}, for the rest of this section, we reserve the letters $h$ and $X=X_h$ for the meanings we have just established.

\begin{table}[htbp]
\[\begin{array}{|l|}
\hline
(A)\; \xymatrix@-1.3pc{*+[o][F]{}\ar@{-}[rr]^h &&*+[o][F]{} \ar@{-}[rr]^{h-1}&&*+[o][F]{} \ar@{.}[rr] &&*+[o][F]{}\ar@{-}[rr]^1 &&*+[o][F]{a}\ar@{-}[rr]^{0} &&*+[o][F]{b}\ar@{-}[rr]^1&&*+[o][F]{} \ar@{.}[rr]&&*+[o][F]{} \ar@{-}[rr]^{h-1}&& *+[o][F]{} \ar@{-}[rr]^{h}&&*+[o][F]{}  }

\\[5pt]

(B)\;
\xymatrix@-1.3pc{*+[o][F]{}\ar@{-}[rr]^0 \ar@{=}[d]_{i+1}^k &&*+[o][F]{} \ar@{-}[rr]^1\ar@{=}[d]_{i+1}^k&&*+[o][F]{}\ar@{=}[d]_{i+1}^k\ar@{.}[rr] &&*+[o][F]{}\ar@{-}[rr]^{i-1}\ar@{-}[d]_{i+1} &&*+[o][F]{a}\ar@{-}[rr]^i \ar@{-}[d]_{i+1} &&*+[o][F]{b} \ar@{-}[rr]^{i+1}&&*+[o][F]{} \\
*+[o][F]{}\ar@{-}[rr]_0 &&*+[o][F]{} \ar@{-}[rr]_1&&*+[o][F]{}\ar@{.}[rr] &&*+[o][F]{}\ar@{-}[rr]_{i-1}&&*+[o][F]{}&&&&
 }
\qquad h=i+1, \textrm{for some }1<k<i \\[5pt]
  
(C)\; \xymatrix@-1.3pc{*+[o][F]{}\ar@{-}[rr]^0 \ar@{-}[d]_{i+1} &&*+[o][F]{} \ar@{-}[rr]^1\ar@{-}[d]_{i+1}&&*+[o][F]{}\ar@{-}[d]_{i+1} \ar@{.}[rr] &&*+[o][F]{}\ar@{-}[rr]^{i-1}\ar@{-}[d]_{i+1} &&*+[o][F]{a}\ar@{-}[rr]^i \ar@{-}[d]_{i+1} &&*+[o][F]{b} \ar@{-}[rr]^{i+1}&& *+[o][F]{} \\
*+[o][F]{}\ar@{-}[rr]_0 &&*+[o][F]{} \ar@{-}[rr]_1&&*+[o][F]{} \ar@{.}[rr] &&*+[o][F]{}\ar@{-}[rr]_{i-1}&&*+[o][F]{}&&\\ 
 }  
 \qquad h=i+1
  \\[5pt]
\hline
\end{array}\]
\caption{Permutation representation graphs  for Proposition~\ref{h}.} \label{Th}
\end{table}

Notice that $X$ is a subset of $O_b$, as illustrated in the following figure.
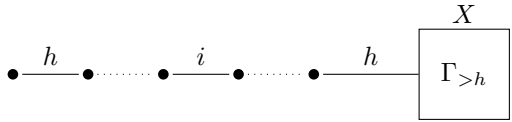
\begin{center}
\begin{tikzpicture}
    \node (A) at (0,0) {};
    \node (B) at (1,0) {};
    \node (C) at (2,0) {};
    \node (D) at (3,0) {};
    \node (E) at (4,0) {};
    \node[draw, minimum size=1.2cm] (F) at (6,0) {$\Gamma_{>h}$};
    \node (X) at (2.5,.9) {};
    \node (X) at (6,.8) {$X$};
    
    \draw[-] (A) -- node[above] {$h$} (B);
    \draw[dotted] (B) -- (C);
    \draw[-] (C) -- node[above] {$i$} (D);
    \draw[dotted] (D) -- (E);
    \draw[-] (E) -- node[above] {$h$} (F);
    
    \fill (A) circle (2pt);
    \fill (B) circle (2pt);
    \fill (C) circle (2pt);
    \fill (D) circle (2pt);
    \fill (E) circle (2pt);
\end{tikzpicture}
\end{center}

\begin{lemma}\label{pe}
The permutation representation graph of $\Gamma_{>h}$ has a pendent edge labeled $h+1$ which is incident, in $\mathcal{G}$, to an edge labeled $h$.
\end{lemma}

\begin{proof}
Let $h+1$ be the label of an edge $e=\{u,v\}$ in $\mathcal{G}$ such that the distance between $e$ and the split $\{a, b\}$ is minimal. By Proposition~\ref{path}, there exists a path connecting $e$ to the split $\{a,b\}$ that is fixed by $G_{>h}$. Say that this path is $b=u_0,u_1,\ldots,u_\ell=u$, for some $\ell$.   By minimality and by the fact that edges whose labels differ by two (or more) form an alternating square, $e$ is incident in this path to an edge labeled $h$, that is, $\{u_{\ell-1},u_\ell\}$ has label $h$.

If $e$ is not a pendent edge in the permutation representation graph of $\Gamma_{>h}$, then $u_\ell$ is on an edge having label $\kappa\ge h+2$. Since $G_{>h}$ fixes the above path, we have $$u_\ell^{\rho_h}=u_{\ell-1}=u_{\ell-1}^{\rho_\kappa}=u_\ell^{\rho_{h}\rho_{\kappa}}=(u_\ell^{\rho_\kappa})^{\rho_h}.$$
Thus $u_\ell^{\rho_\kappa}=u_\ell$, contradicting the fact that $u_\ell$ is on an edge having label $\kappa.$ 
\end{proof}

Since Proposition~\ref{h} (applied with $\ell=h$) gives an upper bound on $h$, it allows us to reduce the problem of determining an upper bound for the rank of $\Gamma=(G,S)$ to that of finding an upper bound for the rank of the SGGI \( \Gamma_{>h} \). 
However, $ G_{>h}$ is not necessarily transitive on $X=\{1,\ldots,n\}\setminus\mathrm{Fix}(G_{>h})$. To address this, we require an extension of the definition of  fracture graph, as given in Section~\ref{sec:fracturegraphs}, to the case of intransitive SGGIs.

An intransitive SGGI \( \Phi = (H, \{ \alpha_0, \ldots, \alpha_{w-1} \}) \) of rank \( w \) is said to have a {\it{\bf fracture graph}} if, for each $l\in \{0,\ldots, w-1\}$, the number of orbits of \( H_l \) exceeds that of \( H \). An \( l \)-edge \( \{x, y\} \) of a fracture graph is defined as a pair $\{x,y\}$, where $y=x\rho_l$ and where \( x \) and \( y \) belong to the same \( H \)-orbit but not to the same \( H_l \)-orbit.

When, for each $l\in \{0,\ldots, w-1\}$, there are at least two possible choices for the $l$-edge of a fracture graph, then $\Phi$ admits a {\it{\bf 2-fracture graph}}. 
On the other hand, if $e$ is the unique choice for an $l$-edge of a fracture graph, then $e$ is called a {\it{\bf split}}.

\subsection{The SGGI $\Gamma_{>h}$ admits a $2$-fracture graph}\label{sec:Gammah2fracture}

\begin{prop}[Proposition 5.27, \cite{2017CFLM}]\label{t}
Let $\Gamma=(G,S)$ be a transitive SGGI of rank $r$.
If $t\in\{0,\ldots,r-2\}$ is such that
\begin{itemize}
\item $t\leq \frac{n-|U|-1}{2}$, with $U= \{1,\ldots,n\}\setminus \mathrm{Fix}(G_{>t})$;
\item $\Gamma_{>t}$ has a $2$-fracture graph;
\item $G_{>t}$ acts intransitively on $U$;
\end{itemize}
then $r\leq (n-1)/2$. 
\end{prop}

By combining the previous proposition with the results from Section~\ref{2f}, we complete the analysis of the case when $\Gamma_{>h}$ admits a 2-fracture graph.

\begin{prop}\label{1splitNP}
Assume Hypotheses~$\ref{hypNP}$ and~$\ref{hypNP2}$.  
If $\Gamma_{>h}$ has a $2$-fracture graph, then $r \leq (n-1)/2$.
\end{prop}

\begin{proof}
Suppose that $\Gamma_{>h}$ has a 2-fracture graph. If $G_{>h}$ is intransitive on $X$, then, by Propositions~\ref{h} and \ref{t}, we have $r \leq (n-1)/2$.
If $G_{>h}$ is transitive on $X$, then, by Propositions~\ref{frac2} and \ref{pe},\footnote{To exclude the first two graphs in Table~\ref{T2F} we are using the fact that $\Gamma_{>h}$ has a pendent edge labeled $h+1$.} we obtain $
r - h - 1 \leq (|X| - 1)/2.$
Now, Proposition~\ref{h} implies that $r \leq \frac{n}{2}$.
Moreover, when equality holds and $r = \frac{n}{2}$, the permutation representation graph of $\Gamma_{>h}$ must be one of the graphs from Table~\ref{T2F} that contains a pendent edge, and the permutation representation graph of $\Gamma_{<h+1}$ is one of the graphs from Table~\ref{Th}.

A case-by-case analysis combining all possible graphs for $\Gamma_{<h+1}$ and $\Gamma_{>h}$ shows that there exists a unique $i$-edge—namely, the split $\{a, b\}$. However, this contradicts the fact that $\rho_i$ is an even permutation.
\end{proof}

\subsection{The SGGI $\Gamma_{>h}$ does not admit a $2$-fracture graph: notation}\label{sec:Gammahnot2fracture}
Now assume $\Gamma_{>h}$ does not have a $2$-fracture graph. Suppose that $j$, for $j>h$, is the label of a split  $\{c,d\}$ for $\Gamma_{>h}$, which may or may not be perfect. In addition suppose that $i$ and $j$ are labels of consecutive splits.\footnote{That is, no label between $i$ and $j$ is the label of a split.} 

For the $j$-split, we follow the notation in Section~\ref{sec:generalnotation}. Let $\{c,d\}$ be the split with label $j$. Let $O_c$ and $O_d$ be the $G_j$-orbits containing $c$ and $d$, respectively, and let $n_C$ and $n_D$ denote the sizes of $O_c$ and $O_d$. Let $C$ and $D$ be the permutation groups induced by $G_j$ in its action on $O_c$ and $O_d$. For each $l\in\{0,\ldots, r-1\}\setminus \{j\}$, let  $\rho_l=\gamma_l\delta_l$, where $\gamma_l$ and $\delta_l$ are the restrictions of $\rho_l$ to $O_c$ and $O_d$, respectively. Then 
$$C = \langle \gamma_l\, |\, l\in \{0,\ldots, r-1\}\setminus\{j\}\rangle,$$
$$D = \langle \delta_l \,| \,l\in \{0,\ldots, r-1\}\setminus\{j\}\rangle.$$
 Let 
 $$J_C= \{  l \in \{0,\ldots, r-1\}\setminus\{j\}\,|\,\gamma_l \textrm{ is not the identity}\},$$
 $$J_D= \{  l \in \{0,\ldots, r-1\}\setminus\{j\}\,|\,\delta_l\textrm{ is not the identity}\}.$$
In the remainder of this section, we consider two cases, depending on whether the split $\{c,d\}$ is perfect or not.
\subsection{The split $\{c,d\}$ is not perfect} 

In this section, we make an additional assumption.

\begin{hypothesis}\label{hypNP3}{\rm We assume that $i$ is the maximal label of an $l$-split satisfying the following property:
\begin{center}
 There exists a permutation $\rho_x$, with $x > l$, that acts non-trivially on both $G_l$-orbits.
\end{center}
This assumption is compatible with the fact that $\{a, b\}$ is a non-perfect $i$-split, as well as with Hypotheses~\ref{hypNP} and~\ref{hypNP2}.

Note also that if $\rho_g$ is a permutation that acts non-trivially on both $G_j$-orbits, then $g < j$. Indeed, since $j$ is the label of a non-perfect split and $j > i$, the maximality of $i$ implies that $g < j$.

Assume further that $g$ is the minimal label of a permutation acting non-trivially on both $G_j$-orbits.
}
\end{hypothesis}
We may summarise some of the information in Hypotheses~\ref{hypNP},~\ref{hypNP2} and~\ref{hypNP3} in the following figure.

\begin{center}
\begin{tikzpicture}[scale=1]

  \draw (0,0) circle (1cm);
  \fill (1,0) circle (2pt) node[above] {$a$};

  \draw (4,0) circle (1cm);
  \fill (3,0) circle (2pt) node[above] {$b$};
  \fill (5,0) circle (2pt) node[above] {$c$};

  \draw (8,0) circle (1cm);
  \fill (7,0) circle (2pt) node[above] {$d$};

  \draw (1,0) -- (3,0) node[midway, above] {$i$};
  \draw (5,0) -- (7,0) node[midway, above] {$j$};
\end{tikzpicture}
\end{center}

\begin{lemma}\label{g1}
Assume Hypotheses~$\ref{hypNP}$,~$\ref{hypNP2}$ and~$\ref{hypNP3}$.
If $r>  (n-1)/2$, then 
\begin{enumerate}
\item\label{eqa} $J_C=\{0,\ldots,j-1\}$ and $J_D=\{g,\ldots,r-1\}\setminus\{j\}$;
\item\label{eqb} $g>0$;
\item\label{eqc} for each $\ell\in \{g,\ldots,j-1\}$, $r-\ell-1\leq \frac{n-|Y_\ell|-1}{2}$ where $Y_\ell= \{1,\ldots,n\}\setminus \mathrm{Fix}(G_{<\ell})$ and $Y_g\subseteq O_c$;
\item\label{eqd}if $r-g-1=(n-|Y_g|-1)/2$, then $\Gamma_{g+1}$ has permutation representation graph isomorphic to one of the graphs in Table~$\ref{Tgg}$, with eventually some additional edges with label $j$ (the unique $j$-edge represented is the split $\{c,d\}$);
\item\label{eqe}if  $\ell\neq g$ and $r-\ell-1=\frac{n-|Y_\ell|-1}{2}$, then $j=r-1$ and $\Gamma_{>\ell}$ has a permutation representation graph containing the $(r-1)$-split, that is a path, fixed by \( G_{<\ell} \), with the sequence of labels $[\ell+1,\ell+2,\ldots, r-2,r-1,r-2, \ldots, \ell+2, \ell+1]$.
\end{enumerate}
\end{lemma}
\begin{proof} 
We start by proving~\eqref{eqa}. As $J_C$ contains $J_A\cup\{i\}$, $0\in J_C$, see the picture above.

If $j=r-1$  then, by Proposition~\ref{path}, $J_C=\{0,\ldots, r-2\}$, thus $J_C$ is an interval. Assume that $j\neq r-1$.

Suppose that $J_C$ is not an  interval. Hence, by Propositions~\ref{CCD} and~\ref{bothimp}, $J_D$ is an interval.
Hence, as $g$ is the minimal label of a permutation acting nontrivially on $O_d$,
$$J_D=\{g,\ldots, j-1\}.$$
Then, $\{0,\,r-1\}\subseteq  J_C$, which gives a contradiction  by Proposition~\ref{allinone}, because $r>(n-1)/2$. 
 Hence $J_C$ is an interval and applying Proposition~\ref{path} we get $J_C=\{0,\ldots,j-1\}$. Now, Proposition~\ref{path} implies that $J_D=\{g,\ldots,r-1\}\setminus\{j\}$. This shows part~\eqref{eqa}.

Next, we prove part~\eqref{eqb}. Arguing by contradiction, suppose $g=0$. In particular, $0\in J_C$ and $0\in J_D$. When $j\ne r-1$, we must have $r-1\in J_D$. Hence we obtain a contradiction from Proposition~\ref{allinone}, because $r>(n-1)/2$. When $j=r-1$, by~\cite[Proposition~5.19]{2017CFLM} (applied to the dual SGGI), we get $r\leq  (n-1)/2$, which is again a  contradiction.  

The first statement in~\eqref{eqc}, ~\eqref{eqd} and~\eqref{eqe} are the dual of Proposition~\ref{h}. The inclusion $Y_g\subseteq O_c$ follows from the minimality of $g$.

\end{proof}

\begin{table}[htbp]
\[\begin{array}{|l|}
\hline
(A')\; \xymatrix@-1.3pc{*+[o][F]{}\ar@{-}[rr]^g &&*+[o][F]{} \ar@{-}[rr]^{g+1}&&*+[o][F]{} \ar@{.}[rr] &&*+[o][F]{}\ar@{-}[rr]^{r-2} &&*+[o][F]{c}\ar@{-}[rr]^{r-1} &&*+[o][F]{d}\ar@{-}[rr]^{r-2}&&*+[o][F]{} \ar@{.}[rr]&&*+[o][F]{} \ar@{-}[rr]^{g+1}&& *+[o][F]{} \ar@{-}[rr]^{g}&&*+[o][F]{}  }

\\[5pt]

(B')\;
\xymatrix@-1.3pc{*+[o][F]{}\ar@{-}[rr]^{j-1}  &&*+[o][F]{c} \ar@{-}[rr]^j&&*+[o][F]{d}\ar@{-}[rr]^{j+1} \ar@{-}[d]_{j-1} &&*+[o][F]{} \ar@{.}[rr]\ar@{-}[d]_{j-1}&&*+[o][F]{}\ar@{=}[d]_{j-1}^k\ar@{-}[rr]^{r-2} &&*+[o][F]{}\ar@{-}[rr]^{r-1}\ar@{=}[d]_{j-1}^k &&*+[o][F]{} \ar@{=}[d]_{j-1}^k\\
&&&&*+[o][F]{}\ar@{-}[rr]_{j+1} &&*+[o][F]{} \ar@{.}[rr]&&*+[o][F]{}\ar@{-}[rr]_{r-2} &&*+[o][F]{}\ar@{-}[rr]_{r-1}&&*+[o][F]{}
 }
\qquad g=j-1, \textrm{for some }j<k<r-2 \\[5pt]
  
(C')\; \xymatrix@-1.3pc{*+[o][F]{}\ar@{-}[rr]^{j-1}  &&*+[o][F]{c} \ar@{-}[rr]^j&&*+[o][F]{d}\ar@{-}[rr]^{j+1} \ar@{-}[d]_{j-1} &&*+[o][F]{} \ar@{.}[rr]\ar@{-}[d]_{j-1}&&*+[o][F]{}\ar@{-}[d]_{j-1}\ar@{-}[rr]^{r-2} &&*+[o][F]{}\ar@{-}[rr]^{r-1}\ar@{-}[d]_{j-1}&&*+[o][F]{} \ar@{-}[d]_{j-1}\\
&&&&*+[o][F]{}\ar@{-}[rr]_{j+1} &&*+[o][F]{} \ar@{.}[rr]&&*+[o][F]{}\ar@{-}[rr]_{r-2} &&*+[o][F]{}\ar@{-}[rr]_{r-1}&&*+[o][F]{}
 }

 \qquad g=j-1
  \\[5pt]
\hline
\end{array}\]
\caption{Permutation representation graphs  for Proposition~\ref{g1}~\eqref{eqd}}\label{Tgg}
\end{table}

\begin{lemma}\label{g2}
Assume Hypotheses~$\ref{hypNP}$,~$\ref{hypNP2}$ and~$\ref{hypNP3}$.
We have $n\ge 6$ and, if  $r>  n/2$, then $g>h$.
\end{lemma}
\begin{proof}
The existence of two non-perfect splits implies $n\ge 6$.

Suppose that $r> n/2$ and $g\leq h$. By Hypothesis~\ref{hypNP2}, $\rho_h$ is defined as the maximal permutation acting non-trivially on both $G_i$-orbits. In particular, $h\in J_A\cap J_B$. By Hypothesis~\ref{hypNP2} $\rho_g$ is defined as the minimal permutation acting non-trivially on both $G_j$-orbits.
Moreover, by Lemma~\ref{g1}~\eqref{eqa}, $J_C=\{0,\ldots, j-1\}$ and $J_D=\{g,\ldots,  r-1\}\setminus\{j\}$. 
In particular, we also have $h\in J_C\cap J_D$.  
Therefore, $\rho_h$ acts nontrivially in each of the three $G_{i,j}$-orbits.
By Proposition~\ref{h} (applied with $\ell=h$), we have 
\begin{equation}\label{monday1}h\leq \frac{n-|X|-1}{2},\end{equation}  where $X=\{1,\ldots,n\}\setminus\mathrm{Fix}(G_{>h})$. We consider separately the cases $h+1\ne j$ and $h+1=j$.

\smallskip

\noindent\textsc{Case $h+1\neq j$.}

\smallskip

\noindent As $g\leq h< h+1<j$,  from Lemma~\ref{g1}~\eqref{eqc}, we have  
\begin{equation}\label{monday2}
r-(h+1)-1\leq \frac{n-|L|-1}{2},
\end{equation} where $L=\{1,\ldots,n\}\setminus\mathrm{Fix}(G_{<h+1})$. 
Let $\bar{L}=\{1,\ldots,n\}\setminus L=\mathrm{Fix}(G_{<h+1})$.
We claim that $\bar{L}\subseteq X$. Indeed, if $p\in \bar{L}$, then  $p\in {\rm Fix}(G_{<h+1})$. So, $p\notin {\rm Fix}(G_{\geq h+1})={\rm Fix}(G_{> h})=\{1,\ldots,n\}\setminus X$ and hence $p\in X$.
Combining~\eqref{monday1} and~\eqref{monday2}, we get 
\begin{equation}\label{monday3}r\leq \frac{n-|X\setminus\bar{L}|+2}{2}.\end{equation}
 Let us now prove that $X\setminus\bar{L}=\overline{ \mathrm{Fix}(G_{>h}) \cup \mathrm{Fix}(G_{<h+1})}$ is nonempty.
Observe that the permutation representation graph $\mathcal{G}$ of $\Gamma$ contains at least one pair of adjacent edges with labels $h$ and $h+1$, for otherwise $G$ would be a direct product. Let $p$ be the meeting point of such a pair of edges. Then $p \notin \mathrm{Fix}(G_{>h}) \cup \mathrm{Fix}(G_{<h+1})$. This implies that
$\{1, \ldots, n\} \ne \mathrm{Fix}(G_{>h}) \cup \mathrm{Fix}(G_{<h+1})$, and hence $X \ne \bar{L}$. In particular, $|X \setminus \bar{L}| \ge 1$. 

Since we are assuming $r > n/2$, from~\eqref{monday1} and~\eqref{monday3}, we deduce that $|X \setminus \bar{L}| = 1$, $r-(h+1)-1= \frac{n-|L|-1}{2}$ and $h = (n - |X| - 1)/2$. 


From the equality $r - (h+1) - 1 = \frac{n - |L| - 1}{2}$,
it follows from Lemma~\ref{g1}~\eqref{eqe} that \( j = r - 1 \), and the permutation representation of $\Gamma_{>h}$ is a path \( \mathcal{P} \), fixed by \( G_{<h+1} \), with the sequence of labels
\[
[h+1, h+2, \ldots, r-2, r-1, r-2, \ldots, h+2, h+1].
\]
Since \( \rho_{r-1} \) must be even, it must act nontrivially on \( L \). This is only possible if \( h+1 = r - 2 \). 
Hence, \( G_{>h}=\langle \rho_{r-2}, \rho_{r-3} \rangle\). As $G$ is even, this implies that \( |X| \geq 5 \). From the equality \( h = \frac{n - |X| - 1}{2} \), we then obtain \( r \leq \frac{n}{2} \), which is a contradiction.


\smallskip

\noindent\textsc{Case $h+1=j$. }

\smallskip

\noindent From the previous case applied dually, we may assume  $g-1= i$. This implies $g=h=j-1=i+1$.

There are two possibilities for $J_D$.  Either $J_D= \{r-2\}$ (when $h=r-2$) or $J_D=\{h\}\cup \{h+2,\ldots, r-1\}$ (when $h<r-2$). 

If $J_D=\{r-2\}$,  then $G_{>h}=\langle \rho_{r-1}\rangle$. Since $G$ consists of even permutations, we deduce $|X|\geq 4$. 
Then,~\eqref{monday1} gives $r-2\leq \frac{n-4-1}{2}$, which is a contradiction. 
A dual argument gives $J_A\neq \{1\}$.

It remains to consider the case $J_A=\{0,\ldots, h-2\}\cup\{h\}$ and $J_D=\{h\}\cup \{h+2,\ldots, r-1\}$. Then, there are two disjoint paths $\mathcal{P}_1$ and $\mathcal{P}_2$, in $O_a$, with all labels from $0$ to $h-2$. In $O_d$ there are also two disjoint paths, $\mathcal{P}_3$ and $\mathcal{P}_4$, containing all labels from $h+2$ to $r-1$. We deduce that
$$2(r-3)\leq (|\mathcal{P}_1|-1)+(|\mathcal{P}_2|-1)+(|\mathcal{P}_3|-1)+(|\mathcal{P}_4|-1).$$
Since $\mathcal{P}_1\cup\mathcal{P}_2\cup\mathcal{P}_3\cup\mathcal{P}_4\subseteq \{1,\ldots,n\}\setminus (O_b\cap O_c)$, we get
$2(r-3)\le n-4-|O_b\cap O_c|$.
Since $|O_b\cap O_c|\geq 2$ (as $\rho_h$ permutes a pair of vertices in this set), we obtain 
$$2(r-3)\leq (n-2)-4.$$
Hence, $r\leq n/2$, contradicting the assumption that $r>n/2$.
\end{proof}

In light of Lemma~\ref{g2}, we may assume $h< g$ and hence the following figure illustrates the situation we need to address, where $G_{\{h+1,\ldots, g-1\}}$ acts on $X\cap Y$, where $X= \{1,\ldots,n\}\setminus \mathrm{Fix}(G_{>h})$  and $Y= \{1,\ldots,n\}\setminus \mathrm{Fix}(G_{<g})$. In addition, as $i$ and $j$ are labels of consecutive splits, $\Gamma_{\{h+1\ldots, g-1\}}$ has a 2-fracture graph.
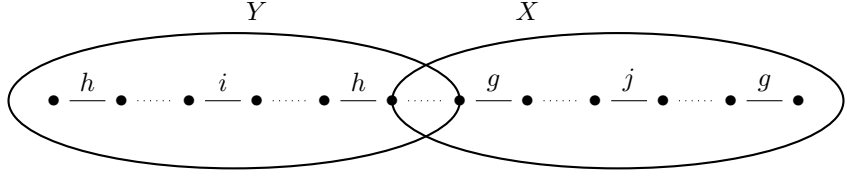
\begin{center}
\begin{tikzpicture}[scale=0.6]
    \node (v1) at (0,0) {$\bullet$};
    \node (v2) at (1.5,0) {$\bullet$};
    \node (v3) at (3,0) {$\bullet$};
    \node (v4) at (4.5,0) {$\bullet$};
    \node (v5) at (6,0) {$\bullet$};
    \node (v6) at (7.5,0) {$\bullet$};
    \node (v7) at (9,0) {$\bullet$};
    \node (v8) at (10.5,0) {$\bullet$};
    \node (v9) at (12,0) {$\bullet$};
    \node (v10) at (13.5,0) {$\bullet$};
    \node (v11) at (15,0) {$\bullet$};
    \node (v12) at (16.5,0) {$\bullet$};
    
    \draw (v1) -- node[above]{$h$} (v2);
    \draw[dotted] (v2) -- (v3);
    \draw (v3) -- node[above]{$i$} (v4);
    \draw[dotted] (v4) -- (v5);
    \draw (v5) -- node[above]{$h$} (v6);
    \draw[white] (v6) -- (v6); 
    \draw[dotted] (v6) -- (v7);
    \draw (v7) -- node[above]{$g$} (v8);
    \draw[dotted] (v8) -- (v9);
    \draw (v9) -- node[above]{$j$} (v10);
    \draw[dotted] (v10) -- (v11);
    \draw (v11) -- node[above]{$g$} (v12);
    
    \draw[thick] (4,0) ellipse (5 and 1.5);
    \node at (4.5,2) {$Y$};
    
    \draw[thick] (12.5,0) ellipse (5 and 1.5);
    \node at (10.5,2) {$X$};
    
\end{tikzpicture}
\end{center}

Similarly to the proof of Proposition~\ref{t} given in \cite[Proposition~5.27]{2017CFLM}, let us prove the following.

\begin{prop}\label{uv}
Let $\Gamma=(G,S)$ be a transitive SGGI.
Suppose that there exist $u,\,v\in\{0,\ldots, r-1\}$ such that
\begin{enumerate}
\item\label{enumerate1} $0<u<v<r-1$;
\item\label{enumerate2} $u\leq \frac{n-|U|-1}{2}$ with $U= \{1,\ldots,n\}\setminus \mathrm{Fix}(G_{>u})$;
\item\label{enumerate3} $r-v-1\leq  \frac{n-|V|-1}{2}$ with $V= \{1,\ldots,n\}\setminus \mathrm{Fix}(G_{<v})$;
\item\label{enumerate4} $\Gamma_{\{u+1,\ldots,v-1\} }$ has a $2$-fracture graph and 
\item\label{enumerate5} $G_{\{u+1,\ldots,v-1 \}}$  acts intransitively on $U\cap V$.
\end{enumerate}
Then $r\leq \frac{n-2}{2}$. 
\end{prop}
\begin{proof}
Suppose that \( u \) and \( v \) satisfy the hypotheses of this proposition and that \( v - u \) is minimal.  
Define  
$
H = G_{\{u+1, \ldots, v-1\}}$, $ \Phi = \Gamma_{\{u+1, \ldots, v-1\}},
$
and let \( \mathcal{G} \) be the permutation representation graph of \( \Gamma \).  

In what follows, the non-trivial \( H \)-orbits are denoted by \( X_1, \ldots, X_c \), while the corresponding group actions of \( H \) on each of these sets are denoted by \( G^{(1)}, \ldots, G^{(c)} \).  
Furthermore, let \( \mathcal{G}^{(1)}, \ldots, \mathcal{G}^{(c)} \) be the corresponding permutation representation graphs.  

Consider a (simple) fracture graph \( \mathcal{F} \) of \( \Phi \), i.e., a graph with \( |U \cap V| \) vertices, where the number of edges is equal to the rank of \( \Phi \).
Each \( l \)-edge of \( \mathcal{F} \) connects vertices in different \( G_l \)-orbits within the same $H$-orbit.  
For each set $X_s$, with $s\in\{1,\ldots, c\}$, denote by $I_s$ the set of labels of edges in $X_s$, and by $F_s$ the set of labels of edges of $\mathcal{F}$ within $X_s$.
Clearly, we have \( F_s \subseteq I_s \).  Choose \( \mathcal{F} \) such that it satisfies the following property:  
\begin{enumerate}
    \item[\textbf{P}] If \( l \in F_s \) is the label of the unique \( l \)-edge in one component swapping vertices in different \( G_l \)-orbits, then no other component has more than one pair of vertices in different \( G_l \)-orbits.\footnote{When \(\mathbf{P}\) holds, another component must contain exactly one \(l\)-edge connecting different \(G_l\)-orbits. These \(l\)-edges are splits within each \(G_l\)-orbit, ensuring the existence of paths with consecutive labels, which are crucial to the proof.
}
\end{enumerate}

We have that \( G^{(s)} \) is generated by a set of involutions (not necessarily independent) with labels in \( I_s \).  
The subset of involutions with labels in \( F_s \) is independent, since \( F_s \) corresponds to a subset of labels of edges in \( \mathcal{F} \).  

Let us bound the set of labels of each component of the fracture graph \( \mathcal{F} \).  
If \( \mathcal{G}^{(s)} \) admits a 2-fracture graph with set of labels \( F_s \), then,  by Proposition~\ref{frac2}, we have
\[
|F_s| \leq \frac{|X_s|}{2}.
\]  
When two components have exactly the same (labeled) permutation representation graph, since $I_s=I_{s'}$ and $|X_s|=|X_{s'}|$, we have
\[
|F_s \cup F_{s'}| = |F_s| \leq |X_s|-1=\frac{|X_s| + |X_{s'}|}{2}-1.
\]

Now, suppose that a component \( \mathcal{G}^{(s)} \) is not a copy of any other component and does not admit a 2-fracture graph.  
Then, there exists an edge \( e \) with label \( l \in F_s \) such that \( \rho_l \) swaps only one pair of vertices of \( X_s \), in different \( G_l \)-orbits.  
Let \( x \) and \( y \) denote, respectively, the minimal and maximal label of an edge in \( \mathcal{G}^{(s)} \).  

By Proposition~\ref{path}, there exist two paths \( \mathcal{P}_1 \) and \( \mathcal{P}_2 \) in $X_s$ such that  \( \mathcal{P}_1 \) contains all labels from \( l-1 \) to \( x \), and \( \mathcal{P}_2 \) contains all labels from \( l+1 \) to \( y \).  In particular, we have $I_s = \{x, x+1, \ldots, y-1, y\}$. 
Define \( P \) as the set of vertices of \( \mathcal{P} = \mathcal{P}_1 \cup \{e\} \cup \mathcal{P}_2 \).

Now, let \( \mathcal{G}^{(s')} \) be a component adjacent to \( \mathcal{G}^{(s)} \) in \( \mathcal{G} \).
Since \( \mathcal{G}^{(s)} \) cannot be a copy of \( \mathcal{G}^{(s')} \), using \textbf{P}, it follows that  
$x \in \{u+1, u+2\}$ or $ y \in \{v-1, v-2\}$.
  
In what follows, we show that when  \( x \in \{u+1, u+2\} \) we get a contradiction.  
Moreover a dual argument also leads to a contradiction when \( y \in \{v-1, v-2\} \).  
We analyze the cases \( x = u+2 \) and \( x = u+1 \) separately.

\smallskip

\noindent\textsc{Case  $x = u+2$.}

\smallskip
  
\noindent In this case, \( \mathcal{G}^{(s')} \) contains a path \( \mathcal{P}' \) that is a copy of \( \mathcal{P} \).

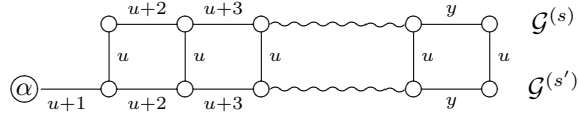
\[
\xymatrix@-1.3pc{
  &&*+[o][F]{} \ar@{-}[rr]^{u+2}\ar@{-}[d]^u&&*+[o][F]{} \ar@{-}[rr]^{u+3}\ar@{-}[d]^u&&*+[o][F]{} \ar@{~}[rrrr]\ar@{-}[d]^u&&&&*+[o][F]{}\ar@{-}[rr]^y\ar@{-}[d]^u&&*+[o][F]{}\ar@{-}[d]^u& \mathcal{G}^{(s)}\\
 *+[o][F]{\alpha}\ar@{-}[rr]_{u+1}&&*+[o][F]{} \ar@{-}[rr]_{u+2}&&*+[o][F]{} \ar@{-}[rr]_{u+3}&&*+[o][F]{} \ar@{~}[rrrr]&&&&*+[o][F]{}\ar@{-}[rr]_y&&*+[o][F]{}&\mathcal{G}^{(s')}
}
\]

Let \( C \) be the set of vertices of \( \mathcal{P} \cup \mathcal{P}' \).  
We have \( 2(y - (u+1)) \leq |C| - 2 \), hence, by hypothesis~\eqref{enumerate2}, we have \( y \leq u + \frac{|C|}{2} \leq \frac{n - |U\setminus C| - 1}{2} \).

The vertex \( \alpha \) is fixed by \( G_{>y} \) and $\{1,\ldots,n\}\setminus\mathrm{Fix}(G_{>y})\subseteq U\setminus C$ 
(recall that \( y \) is the maximal label of  \( \mathcal{G}^{(s)} \)).  This shows that $y$ and $v$ satisfy all the hypotheses of this proposition.
Thus, we obtain a contradiction with the minimality of \( v - u\).

\smallskip

\noindent\textsc{Case  $x = u+1$.}

\smallskip

\noindent  Assume that any component containing a unique \( l \)-edge between vertices in different \( G_l \)-orbits has minimal label \( u+1 \). Let \( y \) be the maximal label in  \( \mathcal{G}^{(s)} \) and let \( \mathcal{P} \) be as previously defined.

Since \( \Phi \) has a 2-fracture graph, there exists another component \( \mathcal{G}^{(s')} \) (which might not be adjacent to \( \mathcal{G}^{(s)} \)) containing an \( l \)-edge \( e' \) between vertices in different $G_l$-orbits. By \textbf{P}, this \( l \)-edge cannot be in an alternating square inside \( X_{s'} \). By assumption, the minimal label in \( X_{s'} \) is also \( u+1 \). Let \( y' \) be the maximal label in \( X_{s'} \). Without loss of generality, assume \( y' \geq y \).
Then, by Proposition~\ref{path}, there exists a path \( \mathcal{P}'_1 \) containing all labels from \( l-1 \) to \( u+1 \) and another path \( \mathcal{P}'_2 \) containing all labels from \( l+1 \) to \( y \), both in \( \mathcal{G}^{(s')} \). Define \( \mathcal{P}' = \mathcal{P}'_1 \cup \{e'\} \cup \mathcal{P}'_2 \).

Since an $l$-edge in $X_{s'}$ is not in an alternating square, $\beta$ (see the picture below) is the unique vertex in \( \mathcal{P}' \) that is not fixed by \( G_{>y} \). Let \( C' \) be the set of vertices of \( \mathcal{P} \cup \mathcal{P}' \).
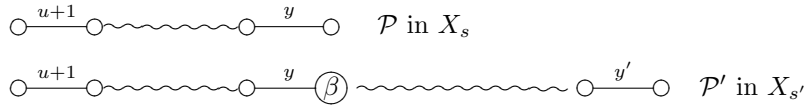
\[
\xymatrix@-1.3pc{
  *+[o][F]{} \ar@{-}[rr]^{u+1}&&*+[o][F]{} \ar@{~}[rrrr]&&&&*+[o][F]{}\ar@{-}[rr]^y&&*+[o][F]{}&\mathcal{P} \textrm{ in }X_s&&&&\\
  *+[o][F]{} \ar@{-}[rr]^{u+1}&&*+[o][F]{} \ar@{~}[rrrr]&&&&*+[o][F]{}\ar@{-}[rr]^y&&*+[o][F]{\beta} \ar@{~}[rrrr]&&&&*+[o][F]{}\ar@{-}[rr]^{y'}&&*+[o][F]{}&\mathcal{P'}\textrm{ in }X_{s'}
}
\]
Therefore, \( G_{>y} \) fixes the set \( C' \setminus \{\beta\} \). Hence, \( 2(y-u) \leq |C'| - 2 \), leading to \( y \leq \frac{n - |U\setminus (C'\setminus\{\beta\})| - 1}{2} \). Consequently, \(\{1,\ldots,n\}\setminus \mathrm{Fix}(G_{>y})\subseteq  U\setminus (C'\setminus\{\beta\})  \), contradicting as in the previous case the of minimality of \( v-u\).

\smallskip

By the previous two cases, the minimal (respectively maximal) label of a component that does not have a $2$-fracture graph is neither $u+1$ nor $u+2$ (respectively, neither $v-1$ nor $v-2$).

 Since \( G_{>u} \) fixes \( \{1, \ldots, n\} \setminus U \), there exists a component \( \mathcal{G}^{(s_1)} \) that has a pendent edge with label \( u+1 \). Hence, this component must have a 2-fracture graph; then, by Proposition~\ref{frac2}, \( |F_{s_1}| \leq \frac{|X_{s_1}| - 1}{2} \). Similarly, there is another component \( \mathcal{G}^{(s_2)} \) that has a pendent edge with label \( v-1 \), and therefore \( |F_{s_2}| \leq \frac{|X_{s_2}| - 1}{2} \). Consequently,
$v - u = \sum_{s=1}^{c} |F_s| \leq \frac{|U \cap V| - 2}{2}$.
Therefore, from this inequality and from~\eqref{enumerate2} and~\eqref{enumerate3}, we have
$r \leq \frac{2n - (|U| + |V| - |U \cap V|) - 2}{2} = \frac{n - 2}{2}$.
 \end{proof}
 
\begin{prop}\label{2splitNP} 
Assume Hypotheses~$\ref{hypNP}$,~$\ref{hypNP2}$ and~$\ref{hypNP3}$. If $\{c,d\}$ is not perfect, then $r\leq n/2$ and $n\ge 6$.
\end{prop}
\begin{proof}
The existence of two non-perfect splits implies $n\ge 6$.

We proceed by contradiction and assume that $r > n/2$.  
Recall from Hypothesis~\ref{hypNP3} that $\rho_g$ is the permutation with the minimal label acting non-trivially on both $G_j$-orbits. By Lemmas~\ref{g1} and~\ref{g2}, we have
\[
r - g-1 \leq \frac{n - |Y| - 1}{2},
\]
where $Y = \{1, \ldots, n\} \setminus \mathrm{Fix}(G_{<g})$,  $Y \subseteq O_c$ and $h < g$. 

If $G_{\{h+1, \ldots, g-1\}}$ acts intransitively on $X \cap Y$, then by Proposition~\ref{uv}, we obtain $r \leq (n-2)/2$, contradicting the assumption $r>n/2$. Therefore, $G_{\{h+1, \ldots, g-1\}}$ must act transitively on $X \cap Y$.

Since $i$ and $j$ are consecutive splits, the group $\Gamma_{\{h+1, \ldots, g-1\}}$ has a fracture graph. Then, by Proposition~\ref{frac2}, either $g-h-1< (|X\cap Y|-1)/2$ or the permutation representation graph of $\Gamma_{\{h+1, \ldots, g-1\}}$ corresponds to one of the graphs in Table~\ref{T2F}. This graph must contain two pendent edges labeled with the minimal and maximal values, namely $h+1$ and $g-1$, respectively. The only graph in Table~\ref{T2F} satisfying this condition is graph~(4). Hence, 
\[
g - h - 1 \le \frac{|X \cap Y| - 1}{2}.
\]

Furthermore, since $h \leq (n - |X| - 1)/2$ (from Proposition~\ref{h}) and $n = |X| + |Y| - |X \cap Y|$, we deduce that $
r \leq (n + 1)/2.$ Hence, $r=(n+1)/2$. This implies $h=(n-|X|-1)/2$, $r-g-1=(n-|Y|-1)/2$ and $g-h-1=(|X\cap Y|-1)/2$.

However, the permutation representation graph of $\Gamma$ can be constructed as follows:
\begin{itemize}
\item connecting a pendent $h$-edge from the graph in Table~\ref{Th} with the pendent $(h+1)$-edge of graph (4) in Table~\ref{T2F}, whose labels are ${h+1, \ldots, g-1}$ and  
\item connecting the pendent $(g-1)$-edge of graph (4)  to a pendent edge from one of the graphs in Table~\ref{Tgg}.
\end{itemize}
This construction leads to a contradiction with the evenness of $G$. Indeed, only graphs (B) and (C) permit the addition of $i$-edges distinct from ${a,b}$, and only graphs ($B'$) and ($C'$) permit the addition of $j$-edges distinct from ${c,d}$. Nevertheless, it is impossible for both $\rho_i$ and $\rho_{i+1}$ to be even, as well as for both $\rho_j$ and $\rho_{j+1}$ to be even, a contradiction.
\end{proof}

\begin{coro}\label{p=0}
Assume Hypotheses~$\ref{hypNP}$,~$\ref{hypNP2}$ and~$\ref{hypNP3}$. If none of the splits is perfect, then  either $r\leq (n-1)/2$, or $r\le n/2$ and $n\ge 6$. In both cases, $r\leq \lfloor 3(n-1)/5\rfloor$.
\end{coro}

\subsection{The split $\{c,d\}$ is  perfect} 
We proceed by induction on the number $p$ of perfect splits of $\Gamma$. The base case (that is, $p=0$) is given by Corollary~\ref{p=0}.

\begin{coro}\label{splitNPP} 
Assume Hypotheses~$\ref{hypNP}$,~$\ref{hypNP2}$ and~$\ref{hypNP3}$. If  at least one of the splits is perfect, then $r\leq\lfloor\frac{3n-3}{5}\rfloor$.
\end{coro}
\begin{proof}
From Proposition~\ref{2splitNP}, we may assume that $\{c,d\}$ is a perfect split. Thus, the number of perfect splits of $\Gamma_{>j}$ is smaller than $p$. Hence, by induction,
\[
r - j - 1 \leq \frac{3n_D - 3}{5}.
\]

By construction, $\Gamma_{\{h+1,\ldots,j-1\}}$ has a 2-fracture graph, which must have two pendent edges with labels $h+1$ and $j-1$. Hence, by Proposition~\ref{1splitNP},
\[
j \leq \frac{n_C - 1}{2}.
\]
Therefore,
\[
r \leq \frac{6n_D + 5n_C - 1}{10}.
\]
Now, note that $r \leq \frac{3n - 3}{5}$ whenever $n \geq n_D + 5$, which always holds. Indeed, if $n_C \le 4$, then $\rho_i$ must be an odd permutation, a contradiction.
\end{proof}


\section{The group $G_i$ is transitive for some $i\in\{0,\ldots,r-1\}$}

In the previous sections, we proved that if $\Gamma$ has a fracture graph, then 
$
r \leq \frac{3(n-1)}{5}.
$
We now turn to the remaining case, where $\Gamma$ does not admit a fracture graph. This implies that there exists an index $i \in \{0, \ldots, r-1\}$ such that the group $G_i$ is transitive.

We first deal with the case that $G_i$ is primitive. Here, we actually prove something slightly stronger.
\begin{prop}\label{prop:primitive}
Let $G$ be a finite primitive group of degree $n$ and let $r$ be the maximum rank of an SGGI for $G$. Then either  $r< \lfloor 3(n-1)/5\rfloor-1$ or one of the following holds:
\begin{enumerate}
\item\label{eq:0} $G$ contains the alternating group $\mathrm{Alt}(n)$,
\item\label{eq:5} $n=5$, $G=D_{5}$ and $r=2$,
\item\label{eq:6} $n=6$, $G=\mathrm{PSL}_2(5)$ and $r=3$, or $G=\mathrm{PGL}_2(5)$ and $r=4$,
\item\label{eq:7} $n=7$, $G=D_7$ and $r=2$,
\item\label{eq:8} $n=8$, $G=\mathrm{PGL}_2(7)$ and $r=3$,
\item\label{eq:9} $n=9$, $G=\mathrm{Sym}(3)\mathrm{wr}\mathrm{Sym}(2)$ or $G=\mathrm{PSL}_2(8)$, and $r=3$,
\item\label{eq:10} $n=10$, $G=\mathrm{Sym}(5)$ and $r=4$, or 
$G=\mathrm{P}\Sigma\mathrm{L}_2(9)$ and $r=5$.
\end{enumerate}
\end{prop}
\begin{proof}
Our proof is based on the following result of Mar\'oti, see~\cite{2002maroti}. Let $G$ be a primitive permutation group of degree $n$. Then one of
the following holds:
\begin{enumerate}
\item\label{eq:i}$G$ is a subgroup of $\mathrm{Sym}(m)\mathrm{wr}\mathrm{Sym}(t)$ containing the socle $(\mathrm{Alt}(m) )^t$, where the action of $\mathrm{Sym}(m)$ is
on $k$-element subsets of $\{1, \ldots , m\}$ and the wreath product has the product
action of degree $n = {m\choose k}^t$;
\item\label{eq:ii}$G = M_{11}$ , $M_{12}$, $M_{23}$, or $M_{24}$ with their $4$-transitive action;
\item\label{eq:iii}$|G|\le n\cdot \prod_{i=0}^{\lfloor\log_2(n)\rfloor-1}(n-2^i)$.
\end{enumerate}

We deal with each of these cases in turn. Assume that part~\eqref{eq:ii} holds. We have verified, with the auxiliary help of a computer, that $M_{11}$ and $M_{23}$ do not admit SGGIs, that the maximal rank of an SGGI for $M_{12}$ is $4<5=\lfloor 3(12-1)/5\rfloor-1$,  and that the maximal rank of an SGGI for $M_{24}$ is $5<12=\lfloor 3(24-1)/5 \rfloor-1$.\footnote{The Mathieu group $M_{24}$ admits only one SGGI of rank 5 up to duality, whereas $M_{12}$ contains many SGGIs of rank 4. These can be constructed using the \texttt{Magma} functions \texttt{Sggi4} and \texttt{Sggi5}, which are available in our companion paper on arXiv~\cite{arxiv}.}

Assume that part~\eqref{eq:iii} holds. Since $G$ admits an independent generating set of cardinality $r$, we deduce that
$$r\le \log_2|G|\le \log_2n+\sum_{i=0}^{\lfloor \log_2n\rfloor-1}\log_2(n-2^{i}) .$$
An elementary computation shows that the right hand side of this inequality is less than $\lfloor 3(n-1)/5\rfloor-1$, except when $n\le 72$. The rest of the proof is computational: for $n\le 72$, we have selected all the primitive groups $G$ of degree $n$ not containing $\mathrm{Alt}(n)$ and with $\log_2(|G|)\ge\lfloor 3(n-1)/5\rfloor-1$, and we have computed the maximal rank of an SGGI for $G$. The only examples where the rank is at least $\lfloor 3(n-1)/5\rfloor-1$ are reported in the statement of the result.

Finally assume that part~\eqref{eq:i} holds. As above, 
\begin{align*}
r&\le \log_2|G|\le\log_2(m!^tt!)= \log_2(m!)t+\log_2(t!)\\
&\le \log_2(m)mt+\log_2(t)t=(\log_2(m)m+\log_2t)t.
\end{align*}
If $r\ge \lfloor 3(n-1)/5\rfloor-1$, then
$$(\log_2(m)m+\log_2t)t\ge \left\lfloor\frac{3({m\choose k}^t-1)}{5}\right\rfloor-1\ge \left\lfloor\frac{3(m^t-1)}{5}\right\rfloor-1.$$
A computation shows that either $t=1$, or $t=2$ with $m\leq 12$. Assume first that $t=2$ and $m\leq 12$. By implementing the exact value of $n={m\choose k}^2$ and refining the upper bound on $|G|$ using $|\mathrm{Sym}(m)\,\mathrm{wr}\,\mathrm{Sym}(2)|=m!^2 2$, we deduce that $r\ge \lfloor 3(n-1)/5\rfloor-1$ only when $k=1$ and $m\in \{5,6\}$. These two cases can be resolved computationally to determine the exact value of $r$, and no exceptions arise in this case.  
Now, assume $t=1$. In this case, we have $\mathrm{Alt}(m)\le G\le \mathrm{Sym}(m)$ and $n={m\choose k}$. Moreover, from~\cite{2000W},  we have $r\leq m-1$. The case $t=k=1$ is the main exception listed in~\eqref{eq:0}; therefore, we may suppose $k\ge 2$. Recall that the binomial coefficient ${m\choose k}$ increases with $k$ when $1\leq k\leq m/2$. The inequality  
\[
m-1\ge \left\lfloor\frac{3({m\choose 2}-1)}{5}\right\rfloor-1
\]  
holds only when $m= 5$. When $G=\mathrm{Sym}(5)$, we obtain one of the exceptions listed in~\eqref{eq:10}. When $G=\mathrm{Alt}(5)$, the rank is $3<4=\lfloor 3(10-1)/5\rfloor-1$. 
\end{proof}

\begin{prop}\label{prop:Giprimitive}
Let $G=\mathrm{Alt}(n)$ be the alternating group of degree $n$ and let $\Gamma=(G,S)$ be an SGGI of rank $r$.
If $G_i$ is primitive for some $i\in\{0,\ldots,r-1\}$, then either $r\leq \lfloor 3(n-1)/5\rfloor$, or $n=5$ and $r=3$.
\end{prop}
\begin{proof}
Suppose $r> \lfloor 3(n-1)/5\rfloor$. As $G_i$ is a proper subgroup of $G=\mathrm{Alt}(n)$, by Proposition~\ref{prop:primitive}, we deduce that $G_i$ is one of the groups listed in parts~\eqref{eq:5}--\eqref{eq:10}.  Now, we have verified the correctness of the result using a computer. 
\end{proof}

Using Proposition~\ref{prop:Giprimitive}, we are now ready to deal with the case that $G_i$ is imprimitive.
\begin{prop}\label{prop:Giimprimitive}
Let $G=\mathrm{Alt}(n)$ be the alternating group of degree $n$ and let $\Gamma=(G,S)$ be an SGGI of rank $r$.
If $G_i$ is  transitive and imprimitive for some $i\in\{0,\ldots,r-1\}$, then $r\leq \lfloor 3(n-1)/5\rfloor$.
\end{prop}
\begin{proof}
Since $G_i$ is imprimitive, it admits a non-trivial system of imprimitivity. We choose a non-trivial system of imprimitivity such that the action of $G_i$ on the blocks is primitive. Let $m$ be the number of blocks and let $k$ be the cardinality of a block. In particular, $G_i$ is a transitive subgroup of the wreath product $\mathrm{Sym}(k)\mathrm{wr}\mathrm{Sym}(m)$ endowed of its imprimitive action of degree $km=n$. Let $H=G_i$, let $\pi:H\to \mathrm{Sym}(m)$ be the projection given by the action of $H$ on the system of imprimitivity and let $J$ be the image of $\pi$. Let $K$ be the permutation group induced by the action of the stabilizer of a block on the block. Thus, $K$ is a transitive subgroup of $\mathrm{Sym}(k)$ and $H\le K\mathrm{wr} J$.  
We now follow closely the argument of Whiston\footnote{The notation here matches that of Whiston.} in~\cite[Section~5]{2000W}, refining the argument by using the fact that $S \setminus \{\rho_i\}$ is an independent generating set consisting of involutions. The same idea can also be found in~\cite[page~471]{2016CFLM}.

Let $L$ be a subset of $S\setminus\{\rho_i\}$ forming an independent generating set for the permutation group $J$, that is, $L^\pi$ is an independent generating set for $J$. Let $C$ be the subset of $(S\setminus\{\rho_i\})\setminus L$ consisting of the elements commuting with $L$, and let $R$ be $(S\setminus\{\rho_i\})\setminus (C\cup L)$. Clearly, 
\begin{equation}\label{r=l+c+r}r=|L|+|C|+|R|+1.\end{equation}

Suppose first that  $J$ admits two distinct commuting proper normal subgroups $A$ and $B$ with $J=AB$. Since $J$ is primitive, $A$ and $B$ are both transitive. Since $A$ and $B$ centralize each other, we deduce that $A$ and $B$ are both regular. Thus $J=A\times B$ and $m=|A|.$ The primitivity of $J$ implies that $A$ and $B$ are non-abelian simple groups.  Since $L$ is independent, we get $2^{|L|}\le |J|=m^2$ and hence $|L|\le 2\log_2(m)$, where $m\ge 60$. As the centre of $J$ is the identity, ${\bf C}_{K\mathrm{wr}J}(\langle J\rangle)$ is contained in the base group $K^m$ of $K\mathrm{wr} J$. Since $\langle L\rangle$ acts transitively on the blocks, 
$${\bf C}_{K^m}(\langle L\rangle)$$
is a diagonal subgroup of the base group $K^m$ of $K\mathrm{wr}J$, that is, there exist group automorphisms $\varphi_2,\ldots,\varphi_m:K\to K$ such that
${\bf C}_{K^m}(\langle L\rangle)$
is a subgroup of $$\{(k,k^{\varphi_2},\ldots,k^{\varphi_m})\mid k\in K\}.$$
As $C\subseteq {\bf C}_{K^m}(\langle L\rangle)$, the main result of Whiston~\cite[Theorem~1]{2000W} implies 
\begin{equation}\label{eq:c}
 |C|\le k-1.
 \end{equation} Finally, from the structure of $J$ we have described above, $R$ has cardinality at most $4$. From~\eqref{r=l+c+r}, we get
\begin{equation*}
r\le 2\log_2m+k-1+4+1=2\log_2m+k+4.
\end{equation*}
Using $m\ge 60$, if follows from a computation that  $2\log_2m+k+4\le \lfloor 3(n-1)/5\rfloor$. For the rest of the argument, we may suppose that $J$ does not admit a decomposition $J=AB$, with $A$ and $B$ as above. In particular, this implies that the set of labels of the elements in $L$ forms an interval, say $L=\{\rho_{a},\rho_{a+1},\ldots,\rho_{a+\ell-1}\}$ where $\ell=|L|$ and $(\rho_{a+i}\rho_{a+i+1})^2\ne 1$ $\forall i\in \{0,\ldots,\ell-2\}$. From this, it also follows that 
\begin{equation}\label{eq:r}
|R|\le 2.
\end{equation}
 
 Applying~\cite[Theorem~1]{2000W} to $J$, we get
 \begin{equation}\label{eq:l}
 |L|\le m-1,
 \end{equation}
where the equality can be attained if and only if $J=\mathrm{Sym}(m)$. Arguing exactly as above with ${\bf C}_{K\mathrm{wr}J}(\langle L\rangle)$, we deduce that~\eqref{eq:c} holds true also in this case; moreover, if the equality is attained, then $K=\mathrm{Sym}(k)$ and $\langle C\rangle$ projects surjectively to $K$.

From~\eqref{r=l+c+r},~\eqref{eq:c},~\eqref{eq:r}, and~\eqref{eq:l}, we obtain  
\begin{equation}\label{eq:sunday}  
r\leq m+k+1.  
\end{equation}  
The inequality $m+k+1\leq \lfloor3(n-1)/5\rfloor$ holds, except in the following cases:  
\begin{itemize}  
\item $m=2$ and $2\leq k\leq 17$, or $2\leq m\leq 17$ and $k=2$, or  
\item $m=3$ and $3\leq k\leq 5$, or $3\leq m\leq 5$ and $k=3$.  
\end{itemize}  
With the aid of a computer, we have computed the maximum rank of an SGGI for $\mathrm{Alt}(n)$ for every $n\leq 18$. In particular, for the remainder of the argument, we may assume that $m=2$ and $10\leq k\leq 17$, or $10\leq m\leq 17$ and $k=2$.  

\smallskip  
\noindent\textsc{Case $k=2$.}  
\smallskip  

\noindent Here, $G_i$ is a subgroup of $\mathrm{Sym}(2)^m\rtimes \mathrm{Sym}(m)$. Again, using a computer computation, we have verified that, except for the symmetric group $\mathrm{Sym}(m)$, the maximum rank of an SGGI for a primitive group of degree $m$ is at most $m-3$. In particular, by refining~\eqref{eq:sunday}, we deduce that $r\leq m+1=n/2+1<\lfloor3(n-1)/5\rfloor$, where the last inequality follows from the fact that $m\geq 10$. Therefore, we may assume that $J=\mathrm{Sym}(m)$.  

Let $V=\mathrm{Sym}(2)^m$ and regard it as a module for $J$ over the field with two elements. The space $V$ has two natural $J$-submodules: the $1$-dimensional module $V_1$, consisting of the constant vectors, and the $(m-1)$-dimensional module $V_2$, consisting of the vectors whose entries sum to $0$. Clearly, $V=V_1\oplus V_2$ when $m$ is odd, and $0<V_1<V_2<V$ when $m$ is even. Since $G_i\leq \mathrm{Alt}(2m)$, we deduce that $G_i\leq V_2\rtimes J$ in both cases. Mortimer~\cite{mortimer} shows that $V_2$ is an irreducible $J$-module when $m$ is odd, and that $V_1$ is the unique non-trivial $J$-submodule of $V_2$ when $m$ is even.
After this digression, we are ready to conclude the proof of the case $k=2$.  

Recall that $L=\{\rho_{a},\rho_{a+1},\ldots,\rho_{a+\ell-1}\}$ and $\ell=m-1$. Let $\rho\in R\cup C$. Since $\langle L\rangle$ projects to $J$, there exists $\rho'\in \langle L\rangle$ such that $f=\rho\rho'\in V_2$. The independence condition implies that $f\neq 1$.  
When $m$ is odd, or when $m$ is even and $f\in V_2\setminus V_1$, the paragraph above implies that $\langle f,L\rangle=G_i$ and hence $r=m<\lfloor3(n-1)/5\rfloor$. When $m$ is even and $f\in V_1$, we deduce that $R\cup C$ can contain at most one more element besides $\rho$. Thus, $r=m+1<\lfloor3(n-1)/5\rfloor$.

\smallskip
\noindent\textsc{Case $m=2$.} 
\smallskip

\noindent Here $\ell=1$, $L=\{\rho_a\}$ and $G_i\le \mathrm{Sym}(k)\mathrm{wr}\mathrm{Sym}(2)$. We write the elements of $\mathrm{Sym}(k)\mathrm{wr}\mathrm{Sym}(2)$ as $(x,y)(1\,2)^\varepsilon$, where $x,y\in\mathrm{Sym}(k)$, $\varepsilon\in \{0,1\}$ and $(1\,2)$ is the generator of $J$.  Replacing $\rho_a$ and $G_i$ by a suitable conjugate via an element of $\mathrm{Sym}(k)\mathrm{wr}\mathrm{Sym}(2)$, we may suppose that $\rho_a=(1\,2)$.\footnote{Indeed, in general $\rho_a=(x,y)(1\,2)$, for some $x,y\in\mathrm{Sym}(k)$. Since $\rho_a$ has order $2$, we have $1=\rho_a^2=(xy,yx)$ and hence $y=x^{-1}$. Thus $\rho_a=(x,x^{-1})(1\,2)$. Now, $(1,x)\rho_a(1,x)^{-1}=(1,x)(x,x^{-1})(1\,2)(1,x^{-1})=(x,1)(1\,2)(1,x^{-1})=(x,1)(x^{-1},1)(1\,2)=(1\,2)$.} Now, $\rho_a=(1\,2)$ as a permutation of $\mathrm{Sym}(2k)$ consists of $k$ cycles of length $2$, because $\rho_a$ is an involution swapping the two blocks of imprimivity of $\mathrm{Sym}(k)\mathrm{wr}\mathrm{Sym}(2)$. In particular, $\rho_a$ is an even permutation if and only if $k$ is even.  As $\rho_a\in G_i\le\mathrm{Alt}(n)$, $k$ is even and hence $k\in \{10,12,14,16\}$. The set $C$ contains $\{\rho_0,\ldots,\rho_{a-2},\rho_{a+2},\ldots,\rho_{r-1}\}\setminus\{\rho_i\}$. For each $j\in \{0,\ldots,a-2,a+2,\ldots,r-1\}\setminus\{i\}$, we may write $\rho_j=(x_j,x_j)(1\,2)^{\varepsilon_j}$, where $x_j\in \mathrm{Sym}(k)$ and $\varepsilon_j\in \{0,1\}$. Set $V=\langle x_j\mid j\le a-2 \hbox{ or }j\ge a+2, j\ne i\rangle$.

If $R=\emptyset$, then by refining~\eqref{eq:sunday}, we deduce that $r\leq k+1=n/2+1<\lfloor3(n-1)/5\rfloor$. Therefore, for the rest of the argument, we may suppose $R\ne \emptyset$.

 Assume $V$ transitive and $V\ne \mathrm{Sym}(k)$. We have determined with \texttt{magma} the maximum rank of every transitive proper subgroup of $\mathrm{Sym}(k)$ and we have verified that it is at most $k-3$. In particular, if $V$ is transitive, then by refining~\eqref{eq:sunday}, we deduce that $r\leq k+1=n/2+1<\lfloor3(n-1)/5\rfloor$. 

Assume $V=\mathrm{Sym}(k)$. Let $\rho_{i'}\in R$. Clearly, $G_{i'}$ is transitive because it contains $L\cup C$ and $\langle L,C\rangle$ is transitive. If $G_{i'}$ is primitive, then $r\le\lfloor3(n-1)/5\rfloor$ by Proposition~\ref{prop:Giprimitive}. If $G_{i'}$ is imprimitive, then by the first part of the proof of this proposition, we deduce that either $r\le \lfloor 3(n-1)/5\rfloor$ or $G_{i'}$ admits a system of imprimitivity consisting of 2 blocks of size $k$. This system of imprimitivity must be the same system of imprimitivity preserved by $G_i$, because $V$ induces the whole symmetric group on the two blocks preserved by $G_i$. This implies that $G=\langle G_i,G_{i'}\rangle$ preserves a non-trivial system of imprimitivity, which is a contradiction.

Assume $V$ intransitive. If $V$ has at least three orbits or $V$ does not induce the whole symmetric group in its action on one of its orbits, then~\cite[Proposition~1]{2000W} implies that $|C|\le k-3$ and hence we reach the same conclusion as in the previous paragraphs. Therefore, $V$ has two orbits and $V$ induces the whole symmetric group on both of its orbits. In particular, $V$ is a subdirect subgroup of $\mathrm{Sym}(\kappa)\times\mathrm{Sym}(k-\kappa)$, for some $1\le \kappa\le k/2$. Suppose $V\ne \mathrm{Sym}(\kappa)\times\mathrm{Sym}(k-\kappa)$. With the aid of a computer we have determined the maximum rank of an SGGI for every proper subdirect subgroup of $\mathrm{Sym}(\kappa)\times\mathrm{Sym}(k-\kappa)$, in all cases the rank is at most $k-3$. Therefore, in this case we can argue as above. Thus $V=\mathrm{Sym}(\kappa)\times\mathrm{Sym}(k-\kappa)$. In particular,  $|C|\le k-2$ and hence $r\le |C|+|L|+|R|+1\le k+2$.  Now, $k+2\le\lfloor3(n-1)/5\rfloor$, for $k\in \{14,16\}$. Therefore, we may suppose $k\in\{10,12\}$. Here the proof is again via the use of a computer. For every $1\le\kappa\le k/2$, we have determined all the SGGI over 
$$\mathrm{Diag}(
(\mathrm{Sym}(\kappa)\times\mathrm{Sym}(k-\kappa))\times 
(\mathrm{Sym}(\kappa)\times\mathrm{Sym}(k-\kappa)))$$
of rank $k-2$. Next, for each such an SGGI, we have determined all possible extensions of the generating set to an SGGI for $\mathrm{Alt}(2k)$. In each case, the rank is at most $k+1\le \lfloor 3(n-1)/5\rfloor$.
\end{proof}

\section{Proof of Theorem~\ref{main}}\label{sec:proof}
We are now ready to bring together the threads of our argument.
\begin{proof}[Proof of Theorem~$\ref{main}$]
Let $G=\mathrm{Alt}(n)$ and let $\Gamma=(G,S)$ be an SGGI for $G$ of maximum rank. When $n\le 9$, the proof follows with a computation on small alternating groups. Therefore, for the rest of the argument we may suppose $n\ge 10$.

Suppose first that $\Gamma$ has a $2$-fracture graph. Then, by Proposition~\ref{frac2}, we deduce $r\le n/2$. Using $n\ge 10$, it follows that $n/2\le \lfloor 3(n-1)/5\rfloor$.

Suppose next that $\Gamma$ satisfies Hypothesis~\ref{hyp} and hence we may apply the result from Section~\ref{sec:4}. In this case, the result follows from Proposition~\ref{casePS}.

Suppose next that $\Gamma$ satisfies Hypothesis~\ref{hypNP}. In this case, the result follows from Corollaries~\ref{p=0} and~\ref{splitNPP}.

Finally, when $G_i$ is transitive for some $i$, the result follows from Propositions~\ref{prop:Giprimitive} and~\ref{prop:Giimprimitive}, depending on whether $G_i$ is primitive or imprimitive.
\end{proof}
\section*{Acknowledgments}
This paper is funded by the European Union - Next Generation EU, Missione 4 Componente 1 CUP B53D23009410006, PRIN 2022- 2022PSTWLB - Group Theory and Applications.

\section{Some codes}
\begin{lstlisting}
/*the input is a finite group G and this function returns yes
 there exist two proper subgroups A and B of G with G=AB and [A,B]=1*/

isok:=function(G)
local Nor,e,a,b,A,B;
Nor:=[x`subgroup:x in NormalSubgroups(G)|
      #(x`subgroup) gt 1 and #(x`subgroup) lt #G];
e:=0;
for a in [1..#Nor] do
 A:=Nor[a];
 for b in [a+1..#Nor] do
   B:=Nor[b];
   if #CommutatorSubgroup(A,B) eq 1 and G eq sub<G|A,B> then 
      e:=1;break a;end if;
 end for;
end for; 
return(e eq 1);
end function;


/*this function given a group G and a subgroup H returns 
a family of representatives for the involutions in G up to H-conjugacy*/
InvolutionRepresentatives := function(G, H)
local involutions,reps,ind,cind,eind,gg,f,g;

involutions :=[x:x in Classes(G)|x[1] eq 2];
reps:={@@};
for ind in [1..#involutions] do
  cind:=involutions[ind][2];
  eind:=0;
  while eind lt cind do
    gg:=involutions[ind][3]^Random(G);
    f:=0;for g in reps do if IsConjugate(H,g,gg) then 
     f:=1;break g;end if;end for;
  if f eq 0 then Include(~reps,gg);eind:=eind+Index(H,Centralizer(H,gg));end if;
  end while;
end for;
return(reps);
end function;


/*this function checks if S is an irredundant set of generators for the group G*/
isirredundant:=function(G,S)
local s;
e:=0;
if not G eq sub<G|S> then e:=1;
  else
  for s in [1..#S] do 
    if G eq sub<G|[S[i]:i in [1..#S]|not i eq s]> then e:=1;break s;end if;
  end for;
end if;
return(e eq 0);
end function;  

/*the input of this function is a group G and a list 
[g1,...,g_\ell] of elements of G. The function returns 
all the lists [g1,...,g_{\ell},g_{\ell+1}] where
-)g_{\ell+1} is an involution commuting with [g1,...,g_{\ell-1}],
-) [g_1,\ldots,g_{\ell+1}] is an irredundant set of generators 
for the group generated by these elements,
-) G is generated by [g1,\ldots,g_{\ell+1}] together with the 
centralizer of [g1,...,g_\ell].

 All such lists are determined up to conjugacy 
 in the centralizer of the group generated by
[g1,...,g_\ell]*/

extendone:=function(G,lst)
local C,CG,I;
C:=Centralizer(G,sub<G|[lst[x]:x in [1..#lst]]>);
CG:=Centralizer(G,sub<G|[lst[x]:x in [1..#lst-1]]>);
I:=InvolutionRepresentatives(CG,C);
I:={lst cat [x]:x in I|isirredundant(sub<G|lst cat [x]>,lst cat [x]) 
     and G eq sub<G|lst cat [x], C meet G>};
return(I);
end function;

/*the input is a finite group G and 
a family of lists of elements from G, 
Then this function applies to each element 
of this family the function extendone and 
then returns the union*/

extend:=function(G,Lst)
return(&join{extendone(G,lst):lst in Lst});
end function;


/*this function does the same job as extendone, 
but only select the g_{\ell+1} such that the order 
of g_\ellg_{\ell+1} is greater than 2. 
From the point of view of the 
Coxeter diagram this means that 
the graph is a connected path.*/

extendonei:=function(G,lst)
local C,CG,I;
C:=Centralizer(G,sub<G|[lst[x]:x in [1..#lst]]>);
CG:=Centralizer(G,sub<G|[lst[x]:x in [1..#lst-1]]>);
I:=InvolutionRepresentatives(CG,C);
I:={lst cat [x]:x in I|Order(x*lst[#lst]) gt 2 
     and isirredundant(sub<G|lst cat [x]>,lst cat [x]) 
      and G eq sub<G|lst cat [x], C meet G>};
return(I);
end function;


/*the input is a finite group G and a family of 
lists of elements from G, Then this function 
applies to each element of this family the 
function extendonei and then returns the union*/



extendi:=function(G,Lst)
return(&join{extendonei(G,lst):lst in Lst});
end function;


/*the next to functions extend the 
list on the left and also on the right*/


extendonebilateral:=function(G,lst)
local L,I,J,C,CG, ind;

C:=Centralizer(G,sub<G|[lst[x]:x in [1..#lst]]>);
CG:=Centralizer(G,sub<G|[lst[x]:x in [1..#lst-1]]>);
J:=InvolutionRepresentatives(CG,C);
I:={lst cat [x]:x in J|Order(lst[#lst]*x) gt 2 
    and isirredundant(sub<G|lst cat [x]>,lst cat [x])};
C:=Centralizer(G,sub<G|[lst[x]:x in [1..#lst]]>);
CG:=Centralizer(G,sub<G|[lst[x]:x in [2..#lst]]>);
J:=InvolutionRepresentatives(CG,C);
I:=I join {[x] cat lst:x in J|Order(x*lst[#lst]) gt 2 
   and isirredundant(sub<G|[x] cat lst>,[x] cat lst)};
for ind in {i:i in [1..#lst-1]|Order(lst[i]*lst[i+1]) le 2} do
   C:=Centralizer(G,sub<G|[lst[x]:x in [1..#lst]]>);
   CG:=Centralizer(G,sub<G|[lst[x]:x in [1..#lst]|not x in [ind,ind+1]]>);
   J:=InvolutionRepresentatives(CG,C);
   I:=I join {[lst[v]:v in [1..ind]] cat [x] 
            cat [lst[v]:v in [ind+1..#lst]]:x in J|
        Order(lst[ind]*x) gt 2 and isirredundant(sub<G|[lst[v]:v in [1..ind]] 
        cat [x] cat [lst[v]:v in [ind+1..#lst]]>,[lst[v]:v in [1..ind]] 
        cat [x] cat [lst[v]:v in [ind+1..#lst]])};
end for;
return(I);
end function;


extendbilateral:=function(G,Lst)
return(&join{extendonebilateral(G,lst):lst in Lst});
end function;





/* the function returns true if G is a sggi of rank 0*/

sggi0:=function(G)
return(#G eq 1);
end function;



/*  the function returns true if G is a sggi of rank 0*/

sggi1:=function(G)
return(#G eq 2);end function;

/*  the function returns true if G is a sggi of rank 2*/

sggi2:=function(G)
local Cl,e,g1,g2,G2;
Cl:=[x[3]:x in Classes(G)|x[1] eq 2];
G2:=&join{Class(G,x):x in Cl};
e:=0;
for g1 in Cl do
  for g2 in G2 do    
  if isirredundant(G,[g1,g2]) then e:=1;break g1;end if; 
  end for;end for;  
return(e eq 1);
end function;



/*  the function returns true if G is a sggi of rank 3*/

sggi3:=function(G)
local N,Cl,g1,Cg1,I1,g2,I2,g3,Cg1g2,I3,g4,Cg1g2g3,Answers;
Answers:={@@};
N:=G;Cl:=InvolutionRepresentatives(G,N);
for g1 in Cl do
  Cg1:=Centralizer(N,g1);
  I1:=InvolutionRepresentatives(G,Cg1);
  I1:=[x:x in I1|isirredundant(sub<G|g1,x>,[g1,x]) 
       and G eq sub<G|g1,x,Cg1 meet G>];
  for g2 in I1 do
    Cg1g2:=Centralizer(Cg1,g2);
    I2:=InvolutionRepresentatives(Cg1 meet G,Cg1g2);
    I2:=[x:x in I2|isirredundant(sub<G|g1,g2,x>,[g1,g2,x])];
    for g3 in I2 do 
    if G eq sub<G|g1,g2,g3> then Include(~Answers,[g1,g2,g3]);break g1;end if;
    end for;
  end for;  
end for;  
return(Answers);
end function;





/*  the function returns true if G is a sggi 
of rank 3 and finds all strings of length 3 for G */

Sggi3:=function(G)
local N,Cl,g1,Cg1,I1,g2,I2,g3,Cg1g2,I3,g4,Cg1g2g3,Answers;
Answers:={@@};
N:=G;Cl:=InvolutionRepresentatives(G,N);
for g1 in Cl do
  Cg1:=Centralizer(N,g1);
  I1:=InvolutionRepresentatives(G,Cg1);
  I1:=[x:x in I1|isirredundant(sub<G|g1,x>,[g1,x]) 
       and G eq sub<G|g1,x,Cg1 meet G>];
  for g2 in I1 do
    Cg1g2:=Centralizer(Cg1,g2);
    I2:=InvolutionRepresentatives(Cg1 meet G,Cg1g2);
    I2:=[x:x in I2|isirredundant(sub<G|g1,g2,x>,[g1,g2,x])];
    for g3 in I2 do 
    if G eq sub<G|g1,g2,g3> then Include(~Answers,[g1,g2,g3]);end if;
    end for;
  end for;  
end for;  
return(Answers);
end function;



/*  the function returns true if G is a sggi of rank 4*/

sggi4:=function(G)
local Cl,g1,Cg1,I1,g2,I2,g3,
      Cg1g2,I3,g4,Cg1g2g3,Answers;
Answers:={@@};
N:=G;Cl:=InvolutionRepresentatives(G,N);
for g1 in Cl do
  Cg1:=Centralizer(G,g1);
  I1:=InvolutionRepresentatives(G,Cg1);
  I1:=[x:x in I1|isirredundant(sub<G|g1,x>,[g1,x]) 
        and G eq sub<G|g1,x,Cg1 meet G>];
  for g2 in I1 do
    Cg1g2:=Centralizer(Cg1,g2);
    I2:=InvolutionRepresentatives(Cg1 meet G,Cg1g2);
    I2:=[x:x in I2|isirredundant(sub<G|g1,g2,x>,[g1,g2,x]) 
         and G eq sub<G|g1,g2,x,Cg1g2 meet G>];
    for g3 in I2 do
      Cg1g2g3:=Centralizer(Cg1g2,g3);
      I3:=InvolutionRepresentatives(Cg1g2 meet G,Cg1g2g3);
      I3:=[x:x in I3|isirredundant(sub<G|g1,g2,g3,x>,[g1,g2,g3,x])];
      for g4 in I3 do 
      if G eq sub<G|g1,g2,g3,g4> then 
      Include(~Answers,[g1,g2,g3,g4]);break g1;end if;
      end for;
    end for;
  end for;  
end for;  
return(Answers);
end function;

/* the function returns true if G is a sggi of rank 4 and finds all strings of length 4 for G*/

Sggi4:=function(G)
local Cl,g1,Cg1,I1,g2,I2,g3,
        Cg1g2,I3,g4,Cg1g2g3,Answers;
Answers:={@@};
N:=G;Cl:=InvolutionRepresentatives(G,N);
for g1 in Cl do
  Cg1:=Centralizer(G,g1);
  I1:=InvolutionRepresentatives(G,Cg1);
  I1:=[x:x in I1|isirredundant(sub<G|g1,x>,[g1,x]) 
  and G eq sub<G|g1,x,Cg1 meet G>];
  for g2 in I1 do
    Cg1g2:=Centralizer(Cg1,g2);
    I2:=InvolutionRepresentatives(Cg1 meet G,Cg1g2);
    I2:=[x:x in I2|isirredundant(sub<G|g1,g2,x>,[g1,g2,x]) 
    and G eq sub<G|g1,g2,x,Cg1g2 meet G>];
    for g3 in I2 do
      Cg1g2g3:=Centralizer(Cg1g2,g3);
      I3:=InvolutionRepresentatives(Cg1g2 meet G,Cg1g2g3);
      I3:=[x:x in I3|isirredundant(sub<G|g1,g2,g3,x>,[g1,g2,g3,x])];
      for g4 in I3 do 
      if G eq sub<G|g1,g2,g3,g4> then 
      Include(~Answers,[g1,g2,g3,g4]);end if;
      end for;
    end for;
  end for;  
end for;  
return(Answers);
end function;



/* the function returns true if G is a sggi of rank 5*/

sggi5:=function(G)
local Cl,e,g1,Cg1,I1,g2,I2,g3,Cg1g2,I3,g4,Cg1g2g3,Cg1g2g3g4,I4,g5,Answers;
Answers:={@@};N:=G;Cl:=InvolutionRepresentatives(G,N);
e:=0;
for g1 in Cl do
  Cg1:=Centralizer(G,g1);
  I1:=InvolutionRepresentatives(G,Cg1);I1:=[x:x in I1|isirredundant(sub<G|g1,x>,[g1,x]) and G eq sub<G|g1,x,Cg1 meet G>];
  for g2 in I1 do
    Cg1g2:=Centralizer(Cg1,g2);
    I2:=InvolutionRepresentatives(Cg1 meet G,Cg1g2);
    I2:=[x:x in I2|isirredundant(sub<G|g1,g2,x>,[g1,g2,x]) and G eq sub<G|g1,g2,x,Cg1g2 meet G>];
    for g3 in I2 do
      Cg1g2g3:=Centralizer(Cg1g2,g3);
      I3:=InvolutionRepresentatives(Cg1g2 meet G,Cg1g2g3);I3:=[x:x in I3|isirredundant(sub<G|g1,g2,g3,x>,[g1,g2,g3,x]) and G eq sub<G|g1,g2,g3,x,Cg1g2g3 meet G>];
      for g4 in I3 do
        Cg1g2g3g4:=Centralizer(Cg1g2g3,g4);
        I4:=InvolutionRepresentatives(Cg1g2g3 meet G,Cg1g2g3g4);
        I4:=[x:x in I4|isirredundant(sub<G|g1,g2,g3,g4,x>,[g1,g2,g3,g4,x])];
        for g5 in I4 do if G eq sub<G|g1,g2,g3,g4,g5> then Include(~Answers,[g1,g2,g3,g4,g5]);break g1;end if;end for;  
      end for;
    end for;
  end for;  
end for;  
return(Answers);
end function;


/* the function returns true if G is a sggi of rank 5 and finds all strings of length 5 for G*/

Sggi5:=function(G)
local Cl,e,g1,Cg1,I1,g2,I2,g3,Cg1g2,I3,g4,Cg1g2g3,Cg1g2g3g4,I4,g5,Answers;
Answers:={@@};N:=G;Cl:=InvolutionRepresentatives(G,N);
e:=0;
for g1 in Cl do
  Cg1:=Centralizer(G,g1);
  I1:=InvolutionRepresentatives(G,Cg1);I1:=[x:x in I1|isirredundant(sub<G|g1,x>,[g1,x]) and G eq sub<G|g1,x,Cg1 meet G>];
  for g2 in I1 do
    Cg1g2:=Centralizer(Cg1,g2);
    I2:=InvolutionRepresentatives(Cg1 meet G,Cg1g2);
    I2:=[x:x in I2|isirredundant(sub<G|g1,g2,x>,[g1,g2,x]) and G eq sub<G|g1,g2,x,Cg1g2 meet G>];
    for g3 in I2 do
      Cg1g2g3:=Centralizer(Cg1g2,g3);
      I3:=InvolutionRepresentatives(Cg1g2 meet G,Cg1g2g3);I3:=[x:x in I3|isirredundant(sub<G|g1,g2,g3,x>,[g1,g2,g3,x]) and G eq sub<G|g1,g2,g3,x,Cg1g2g3 meet G>];
      for g4 in I3 do
        Cg1g2g3g4:=Centralizer(Cg1g2g3,g4);
        I4:=InvolutionRepresentatives(Cg1g2g3 meet G,Cg1g2g3g4);
        I4:=[x:x in I4|isirredundant(sub<G|g1,g2,g3,g4,x>,[g1,g2,g3,g4,x])];
        for g5 in I4 do if G eq sub<G|g1,g2,g3,g4,g5> then Include(~Answers,[g1,g2,g3,g4,g5]);break g1;end if;end for;  
      end for;
    end for;
  end for;  
end for;  
return(Answers);
end function;




/* the function returns true if G is a sggi of rank 6*/

sggi6:=function(G)
local Cl,g1,Cg1,I1,g2,I2,g3,Cg1g2,I3,g4,Cg1g2g3,Cg1g2g3g4,I4,g5,Cg1g2g3g4g5,I5,g6,Answers;
Answers:={@@};N:=G;Cl:=InvolutionRepresentatives(G,N);
for g1 in Cl do
  Cg1:=Centralizer(G,g1);
  I1:=InvolutionRepresentatives(G,Cg1);I1:=[x:x in I1|isirredundant(sub<G|g1,x>,[g1,x]) and G eq sub<G|g1,x,Cg1 meet G>];
  for g2 in I1 do
    Cg1g2:=Centralizer(Cg1,g2);
    I2:=InvolutionRepresentatives(Cg1 meet G,Cg1g2);I2:=[x:x in I2|isirredundant(sub<G|g1,g2,x>,[g1,g2,x]) and G eq sub<G|g1,g2,x,Cg1g2 meet G>];
    for g3 in I2 do
      Cg1g2g3:=Centralizer(Cg1g2,g3);
      I3:=InvolutionRepresentatives(Cg1g2 meet G,Cg1g2g3);I3:=[x:x in I3|isirredundant(sub<G|g1,g2,g3,x>,[g1,g2,g3,x]) and G eq sub<G|g1,g2,g3,x,Cg1g2g3 meet G>];
      for g4 in I3 do
        Cg1g2g3g4:=Centralizer(Cg1g2g3,g4);
        I4:=InvolutionRepresentatives(Cg1g2g3 meet G,Cg1g2g3g4);I4:=[x:x in I4|isirredundant(sub<G|g1,g2,g3,g4,x>,[g1,g2,g3,g4,x]) and G eq sub<G|g1,g2,g3,g4,x,Cg1g2g3g4 meet G>];
        for g5 in I4 do
          Cg1g2g3g4g5:=Centralizer(Cg1g2g3g4,g5);
          I5:=InvolutionRepresentatives(Cg1g2g3g4 meet G,Cg1g2g3g4g5);
          I5:=[x:x in I5|isirredundant(sub<G|g1,g2,g3,g4,g5,x>,[g1,g2,g3,g4,g5,x])];
          for g6 in I5 do if G eq sub<G|g1,g2,g3,g4,g5,g6> then Include(~Answers,[g1,g2,g3,g4,g5,g6]);break g1;end if;end for;
        end for;  
      end for;
    end for;
  end for;  
end for;  
return(Answers);
end function;




/* the function returns true if G is a sggi of rank 6 and finds all strings of length 6 for G*/

Sggi6:=function(G)
local Cl,g1,Cg1,I1,g2,I2,g3,Cg1g2,I3,g4,Cg1g2g3,Cg1g2g3g4,I4,g5,Cg1g2g3g4g5,I5,g6,Answers;
Answers:={@@};N:=G;Cl:=InvolutionRepresentatives(G,N);
for g1 in Cl do
  Cg1:=Centralizer(G,g1);
  I1:=InvolutionRepresentatives(G,Cg1);I1:=[x:x in I1|isirredundant(sub<G|g1,x>,[g1,x]) and G eq sub<G|g1,x,Cg1 meet G>];
  for g2 in I1 do
    Cg1g2:=Centralizer(Cg1,g2);
    I2:=InvolutionRepresentatives(Cg1 meet G,Cg1g2);I2:=[x:x in I2|isirredundant(sub<G|g1,g2,x>,[g1,g2,x]) and G eq sub<G|g1,g2,x,Cg1g2 meet G>];
    for g3 in I2 do
      Cg1g2g3:=Centralizer(Cg1g2,g3);
      I3:=InvolutionRepresentatives(Cg1g2 meet G,Cg1g2g3);I3:=[x:x in I3|isirredundant(sub<G|g1,g2,g3,x>,[g1,g2,g3,x]) and G eq sub<G|g1,g2,g3,x,Cg1g2g3 meet G>];
      for g4 in I3 do
        Cg1g2g3g4:=Centralizer(Cg1g2g3,g4);
        I4:=InvolutionRepresentatives(Cg1g2g3 meet G,Cg1g2g3g4);I4:=[x:x in I4|isirredundant(sub<G|g1,g2,g3,g4,x>,[g1,g2,g3,g4,x]) and G eq sub<G|g1,g2,g3,g4,x,Cg1g2g3g4 meet G>];
        for g5 in I4 do
          Cg1g2g3g4g5:=Centralizer(Cg1g2g3g4,g5);
          I5:=InvolutionRepresentatives(Cg1g2g3g4 meet G,Cg1g2g3g4g5);
          I5:=[x:x in I5|isirredundant(sub<G|g1,g2,g3,g4,g5,x>,[g1,g2,g3,g4,g5,x])];
          for g6 in I5 do if G eq sub<G|g1,g2,g3,g4,g5,g6> then Include(~Answers,[g1,g2,g3,g4,g5,g6]);end if;end for;
        end for;  
      end for;
    end for;
  end for;  
end for;  
return(Answers);
end function;


/* the function returns true if G is a sggi of rank 7*/

sggi7:=function(G)
local Cl,e,g1,Cg1,I1,g2,I2,g3,Cg1g2,I3,g4,Cg1g2g3,Cg1g2g3g4,I4,g5,Cg1g2g3g4g5,I5,g6,Cg1g2g3g4g5g6,I6,g7,Answers;
Answers:={@@};N:=G;Cl:=InvolutionRepresentatives(G,N);
for g1 in Cl do
  Cg1:=Centralizer(G,g1);
  I1:=InvolutionRepresentatives(G,Cg1);I1:=[x:x in I1|isirredundant(sub<G|g1,x>,[g1,x]) and G eq sub<G|g1,x,Cg1 meet G>];
  for g2 in I1 do
    Cg1g2:=Centralizer(Cg1,g2);
    I2:=InvolutionRepresentatives(Cg1 meet G,Cg1g2);I2:=[x:x in I2|isirredundant(sub<G|g1,g2,x>,[g1,g2,x]) and G eq sub<G|g1,g2,x,Cg1g2 meet G>];
    for g3 in I2 do
      Cg1g2g3:=Centralizer(Cg1g2,g3);
      I3:=InvolutionRepresentatives(Cg1g2 meet G,Cg1g2g3);I3:=[x:x in I3|isirredundant(sub<G|g1,g2,g3,x>,[g1,g2,g3,x]) and G eq sub<G|g1,g2,g3,x,Cg1g2g3 meet G>];
      for g4 in I3 do
        Cg1g2g3g4:=Centralizer(Cg1g2g3,g4);
        I4:=InvolutionRepresentatives(Cg1g2g3 meet G,Cg1g2g3g4);I4:=[x:x in I4|isirredundant(sub<G|g1,g2,g3,g4,x>,[g1,g2,g3,g4,x]) and G eq sub<G|g1,g2,g3,g4,x,Cg1g2g3g4 meet G>];
        for g5 in I4 do
          Cg1g2g3g4g5:=Centralizer(Cg1g2g3g4,g5);
          I5:=InvolutionRepresentatives(Cg1g2g3g4 meet G,Cg1g2g3g4g5);I5:=[x:x in I5|isirredundant(sub<G|g1,g2,g3,g4,g5,x>,[g1,g2,g3,g4,g5,x]) and G eq sub<G|g1,g2,g3,g4,g5,x,Cg1g2g3g4g5 meet G>];
          for g6 in I5 do
            Cg1g2g3g4g5g6:=Centralizer(Cg1g2g3g4g5,g6);
            I6:=InvolutionRepresentatives(Cg1g2g3g4g5 meet G,Cg1g2g3g4g5g6);I6:=[x:x in I6|isirredundant(sub<G|g1,g2,g3,g4,g5,g6,x>,[g1,g2,g3,g4,g5,g6,x])];
            for g7 in I6 do if G eq sub<G|g1,g2,g3,g4,g5,g6,g7> then Include(~Answers,[g1,g2,g3,g4,g5,g6,g7]);break g1;end if;end for;  
          end for;
        end for;  
      end for;
    end for;
  end for;  
end for;  
return(Answers);
end function;


/* the function returns true if G is a sggi of rank 7 and finds all strings of length 7*/

Sggi7:=function(G)
local Cl,e,g1,Cg1,I1,g2,I2,g3,Cg1g2,I3,g4,Cg1g2g3,Cg1g2g3g4,I4,g5,Cg1g2g3g4g5,I5,g6,Cg1g2g3g4g5g6,I6,g7,Answers;
Answers:={@@};N:=G;Cl:=InvolutionRepresentatives(G,N);
for g1 in Cl do
  Cg1:=Centralizer(G,g1);
  I1:=InvolutionRepresentatives(G,Cg1);I1:=[x:x in I1|isirredundant(sub<G|g1,x>,[g1,x]) and G eq sub<G|g1,x,Cg1 meet G>];
  for g2 in I1 do
    Cg1g2:=Centralizer(Cg1,g2);
    I2:=InvolutionRepresentatives(Cg1 meet G,Cg1g2);I2:=[x:x in I2|isirredundant(sub<G|g1,g2,x>,[g1,g2,x]) and G eq sub<G|g1,g2,x,Cg1g2 meet G>];
    for g3 in I2 do
      Cg1g2g3:=Centralizer(Cg1g2,g3);
      I3:=InvolutionRepresentatives(Cg1g2 meet G,Cg1g2g3);I3:=[x:x in I3|isirredundant(sub<G|g1,g2,g3,x>,[g1,g2,g3,x]) and G eq sub<G|g1,g2,g3,x,Cg1g2g3 meet G>];
      for g4 in I3 do
        Cg1g2g3g4:=Centralizer(Cg1g2g3,g4);
        I4:=InvolutionRepresentatives(Cg1g2g3 meet G,Cg1g2g3g4);I4:=[x:x in I4|isirredundant(sub<G|g1,g2,g3,g4,x>,[g1,g2,g3,g4,x]) and G eq sub<G|g1,g2,g3,g4,x,Cg1g2g3g4 meet G>];
        for g5 in I4 do
          Cg1g2g3g4g5:=Centralizer(Cg1g2g3g4,g5);
          I5:=InvolutionRepresentatives(Cg1g2g3g4 meet G,Cg1g2g3g4g5);I5:=[x:x in I5|isirredundant(sub<G|g1,g2,g3,g4,g5,x>,[g1,g2,g3,g4,g5,x]) and G eq sub<G|g1,g2,g3,g4,g5,x,Cg1g2g3g4g5 meet G>];
          for g6 in I5 do
            Cg1g2g3g4g5g6:=Centralizer(Cg1g2g3g4g5,g6);
            I6:=InvolutionRepresentatives(Cg1g2g3g4g5 meet G,Cg1g2g3g4g5g6);I6:=[x:x in I6|isirredundant(sub<G|g1,g2,g3,g4,g5,g6,x>,[g1,g2,g3,g4,g5,g6,x])];
            for g7 in I6 do if G eq sub<G|g1,g2,g3,g4,g5,g6,g7> then Include(~Answers,[g1,g2,g3,g4,g5,g6,g7]);end if;end for;  
          end for;
        end for;  
      end for;
    end for;
  end for;  
end for;  
return(Answers);
end function;



/* the function returns true if G is a sggi of rank 8*/

sggi8:=function(G)
local Cl,e,g1,Cg1,I1,g2,I2,g3,Cg1g2,I3,g4,Cg1g2g3,Cg1g2g3g4,I4,g5,Cg1g2g3g4g5,I5,g6,Cg1g2g3g4g5g6,I6,g7,Cg1g2g3g4g5g6g7,g8,I7,Answers;
Answers:={@@};N:=G;Cl:=InvolutionRepresentatives(G,N);
for g1 in Cl do
  Cg1:=Centralizer(G,g1);
  I1:=InvolutionRepresentatives(G,Cg1);I1:=[x:x in I1|isirredundant(sub<G|g1,x>,[g1,x]) and G eq sub<G|g1,x,Cg1 meet G>];
  for g2 in I1 do
    Cg1g2:=Centralizer(Cg1,g2);
    I2:=InvolutionRepresentatives(Cg1 meet G,Cg1g2);I2:=[x:x in I2|isirredundant(sub<G|g1,g2,x>,[g1,g2,x]) and G eq sub<G|g1,g2,x,Cg1g2 meet G>];
    for g3 in I2 do
      Cg1g2g3:=Centralizer(Cg1g2,g3);
      I3:=InvolutionRepresentatives(Cg1g2 meet G,Cg1g2g3);I3:=[x:x in I3|isirredundant(sub<G|g1,g2,g3,x>,[g1,g2,g3,x]) and G eq sub<G|g1,g2,g3,x,Cg1g2g3 meet G>];
      for g4 in I3 do
        Cg1g2g3g4:=Centralizer(Cg1g2g3,g4);
        I4:=InvolutionRepresentatives(Cg1g2g3 meet G,Cg1g2g3g4);I4:=[x:x in I4|isirredundant(sub<G|g1,g2,g3,g4,x>,[g1,g2,g3,g4,x]) and G eq sub<G|g1,g2,g3,g4,x,Cg1g2g3g4 meet G>];
        for g5 in I4 do
          Cg1g2g3g4g5:=Centralizer(Cg1g2g3g4,g5);
          I5:=InvolutionRepresentatives(Cg1g2g3g4 meet G,Cg1g2g3g4g5);I5:=[x:x in I5|isirredundant(sub<G|g1,g2,g3,g4,g5,x>,[g1,g2,g3,g4,g5,x]) and G eq sub<G|g1,g2,g3,g4,g5,x,Cg1g2g3g4g5 meet G>];
          for g6 in I5 do
            Cg1g2g3g4g5g6:=Centralizer(Cg1g2g3g4g5,g6);
            I6:=InvolutionRepresentatives(Cg1g2g3g4g5 meet G,Cg1g2g3g4g5g6);I6:=[x:x in I6|isirredundant(sub<G|g1,g2,g3,g4,g5,g6,x>,[g1,g2,g3,g4,g5,g6,x]) and G eq sub<G|g1,g2,g3,g4,g5,g6,x,Cg1g2g3g4g5g6 meet G>];
            for g7 in I6 do
              Cg1g2g3g4g5g6g7:=Centralizer(Cg1g2g3g4g5g6,g7);
              I7:=InvolutionRepresentatives(Cg1g2g3g4g5g6 meet G,Cg1g2g3g4g5g6g7);I7:=[x:x in I7|isirredundant(sub<G|g1,g2,g3,g4,g5,g6,g7,x>,[g1,g2,g3,g4,g5,g6,g7,x])];
              for g8 in I7 do if G eq sub<G|g1,g2,g3,g4,g5,g6,g7,g8> then Include(~Answers,[g1,g2,g3,g4,g5,g6,g7,g8]);break g1;end if;end for;
            end for;  
          end for;
        end for;  
      end for;
    end for;
  end for;  
end for;  
return(Answers);
end function;



/* the function returns true if G is a sggi of rank 8 and finds all strings of length 8 for G*/

Sggi8:=function(G)
local Cl,e,g1,Cg1,I1,g2,I2,g3,Cg1g2,I3,g4,Cg1g2g3,Cg1g2g3g4,I4,g5,Cg1g2g3g4g5,I5,g6,Cg1g2g3g4g5g6,I6,g7,Cg1g2g3g4g5g6g7,g8,I7,Answers;
Answers:={@@};N:=G;Cl:=InvolutionRepresentatives(G,N);
for g1 in Cl do
  Cg1:=Centralizer(G,g1);
  I1:=InvolutionRepresentatives(G,Cg1);I1:=[x:x in I1|isirredundant(sub<G|g1,x>,[g1,x]) and G eq sub<G|g1,x,Cg1 meet G>];
  for g2 in I1 do
    Cg1g2:=Centralizer(Cg1,g2);
    I2:=InvolutionRepresentatives(Cg1 meet G,Cg1g2);I2:=[x:x in I2|isirredundant(sub<G|g1,g2,x>,[g1,g2,x]) and G eq sub<G|g1,g2,x,Cg1g2 meet G>];
    for g3 in I2 do
      Cg1g2g3:=Centralizer(Cg1g2,g3);
      I3:=InvolutionRepresentatives(Cg1g2 meet G,Cg1g2g3);I3:=[x:x in I3|isirredundant(sub<G|g1,g2,g3,x>,[g1,g2,g3,x]) and G eq sub<G|g1,g2,g3,x,Cg1g2g3 meet G>];
      for g4 in I3 do
        Cg1g2g3g4:=Centralizer(Cg1g2g3,g4);
        I4:=InvolutionRepresentatives(Cg1g2g3 meet G,Cg1g2g3g4);I4:=[x:x in I4|isirredundant(sub<G|g1,g2,g3,g4,x>,[g1,g2,g3,g4,x]) and G eq sub<G|g1,g2,g3,g4,x,Cg1g2g3g4 meet G>];
        for g5 in I4 do
          Cg1g2g3g4g5:=Centralizer(Cg1g2g3g4,g5);
          I5:=InvolutionRepresentatives(Cg1g2g3g4 meet G,Cg1g2g3g4g5);I5:=[x:x in I5|isirredundant(sub<G|g1,g2,g3,g4,g5,x>,[g1,g2,g3,g4,g5,x]) and G eq sub<G|g1,g2,g3,g4,g5,x,Cg1g2g3g4g5 meet G>];
          for g6 in I5 do
            Cg1g2g3g4g5g6:=Centralizer(Cg1g2g3g4g5,g6);
            I6:=InvolutionRepresentatives(Cg1g2g3g4g5 meet G,Cg1g2g3g4g5g6);I6:=[x:x in I6|isirredundant(sub<G|g1,g2,g3,g4,g5,g6,x>,[g1,g2,g3,g4,g5,g6,x]) and G eq sub<G|g1,g2,g3,g4,g5,g6,x,Cg1g2g3g4g5g6 meet G>];
            for g7 in I6 do
              Cg1g2g3g4g5g6g7:=Centralizer(Cg1g2g3g4g5g6,g7);
              I7:=InvolutionRepresentatives(Cg1g2g3g4g5g6 meet G,Cg1g2g3g4g5g6g7);I7:=[x:x in I7|isirredundant(sub<G|g1,g2,g3,g4,g5,g6,g7,x>,[g1,g2,g3,g4,g5,g6,g7,x])];
              for g8 in I7 do if G eq sub<G|g1,g2,g3,g4,g5,g6,g7,g8> then Include(~Answers,[g1,g2,g3,g4,g5,g6,g7,g8]);end if;end for;
            end for;  
          end for;
        end for;  
      end for;
    end for;
  end for;  
end for;  
return(Answers);
end function;



/* the function returns true if G is a sggi of rank 9*/

sggi9:=function(G)
local Cl,g1,Cg1,I1,g2,I2,g3,Cg1g2,I3,g4,Cg1g2g3,Cg1g2g3g4,I4,g5,Cg1g2g3g4g5,I5,g6,Cg1g2g3g4g5g6,I6,g7,Cg1g2g3g4g5g6g7,g8,I7,Cg1g2g3g4g5g6g7g8,I8,g9,Answers;
Answers:={@@};N:=G;Cl:=InvolutionRepresentatives(G,N);
for g1 in Cl do
  Cg1:=Centralizer(G,g1);
  I1:=InvolutionRepresentatives(G,Cg1);I1:=[x:x in I1|isirredundant(sub<G|g1,x>,[g1,x]) and G eq sub<G|g1,x,Cg1 meet G>];
  for g2 in I1 do
    Cg1g2:=Centralizer(Cg1,g2);
    I2:=InvolutionRepresentatives(Cg1 meet G,Cg1g2);I2:=[x:x in I2|isirredundant(sub<G|g1,g2,x>,[g1,g2,x]) and G eq sub<G|g1,g2,x,Cg1g2 meet G>];
    for g3 in I2 do
      Cg1g2g3:=Centralizer(Cg1g2,g3);
      I3:=InvolutionRepresentatives(Cg1g2 meet G,Cg1g2g3);I3:=[x:x in I3|isirredundant(sub<G|g1,g2,g3,x>,[g1,g2,g3,x]) and G eq sub<G|g1,g2,g3,x,Cg1g2g3 meet G>];
      for g4 in I3 do
        Cg1g2g3g4:=Centralizer(Cg1g2g3,g4);
        I4:=InvolutionRepresentatives(Cg1g2g3 meet G,Cg1g2g3g4);I4:=[x:x in I4|isirredundant(sub<G|g1,g2,g3,g4,x>,[g1,g2,g3,g4,x]) and G eq sub<G|g1,g2,g3,g4,x,Cg1g2g3g4 meet G>];
        for g5 in I4 do
          Cg1g2g3g4g5:=Centralizer(Cg1g2g3g4,g5);
          I5:=InvolutionRepresentatives(Cg1g2g3g4 meet G,Cg1g2g3g4g5);I5:=[x:x in I5|isirredundant(sub<G|g1,g2,g3,g4,g5,x>,[g1,g2,g3,g4,g5,x]) and G eq sub<G|g1,g2,g3,g4,g5,x,Cg1g2g3g4g5 meet G>];
          for g6 in I5 do
            Cg1g2g3g4g5g6:=Centralizer(Cg1g2g3g4g5,g6);
            I6:=InvolutionRepresentatives(Cg1g2g3g4g5 meet G,Cg1g2g3g4g5g6);I6:=[x:x in I6|isirredundant(sub<G|g1,g2,g3,g4,g5,g6,x>,[g1,g2,g3,g4,g5,g6,x]) and G eq sub<G|g1,g2,g3,g4,g5,g6,x,Cg1g2g3g4g5g6 meet G>];
            for g7 in I6 do
              Cg1g2g3g4g5g6g7:=Centralizer(Cg1g2g3g4g5g6,g7);
              I7:=InvolutionRepresentatives(Cg1g2g3g4g5g6 meet G,Cg1g2g3g4g5g6g7);I7:=[x:x in I7|isirredundant(sub<G|g1,g2,g3,g4,g5,g6,g7,x>,[g1,g2,g3,g4,g5,g6,g7,x]) and G eq sub<G|g1,g2,g3,g4,g5,g6,g7,x,Cg1g2g3g4g5g6g7 meet G>];
              for g8 in I7 do
                Cg1g2g3g4g5g6g7g8:=Centralizer(Cg1g2g3g4g5g6g7,g8);
                I8:=InvolutionRepresentatives(Cg1g2g3g4g5g6g7 meet G,Cg1g2g3g4g5g6g7g8);I8:=[x:x in I8|isirredundant(sub<G|g1,g2,g3,g4,g5,g6,g7,g8,x>,[g1,g2,g3,g4,g5,g6,g7,g8,x])];
                for g9 in I8 do if G eq sub<G|g1,g2,g3,g4,g5,g6,g7,g8,g9> then Include(~Answers,[g1,g2,g3,g4,g5,g6,g7,g8,g9]);break g1;end if;end for;
              end for;
            end for;  
          end for;
        end for;  
      end for;
    end for;
  end for;  
end for;  
return(Answers);
end function;



/* the function returns true if G is a sggi of rank 10*/
sggi10:=function(G)
local Cl,g1,Cg1,I1,g2,I2,g3,Cg1g2,I3,g4,Cg1g2g3,Cg1g2g3g4,I4,g5,Cg1g2g3g4g5,I5,g6,Cg1g2g3g4g5g6,I6,g7,Cg1g2g3g4g5g6g7,g8,I7,Cg1g2g3g4g5g6g7g8,I8,g9,
   Cg1g2g3g4g5g6g7g8g9,I9,g10,Answers;
Answers:={@@};N:=G;Cl:=InvolutionRepresentatives(G,N);
for g1 in Cl do
  Cg1:=Centralizer(G,g1);
  I1:=InvolutionRepresentatives(G,Cg1);I1:=[x:x in I1|isirredundant(sub<G|g1,x>,[g1,x]) and G eq sub<G|g1,x,Cg1 meet G>];
  for g2 in I1 do
    Cg1g2:=Centralizer(Cg1,g2);
    I2:=InvolutionRepresentatives(Cg1 meet G,Cg1g2);I2:=[x:x in I2|isirredundant(sub<G|g1,g2,x>,[g1,g2,x]) and G eq sub<G|g1,g2,x,Cg1g2 meet G>];
    for g3 in I2 do
      Cg1g2g3:=Centralizer(Cg1g2,g3);
      I3:=InvolutionRepresentatives(Cg1g2 meet G,Cg1g2g3);I3:=[x:x in I3|isirredundant(sub<G|g1,g2,g3,x>,[g1,g2,g3,x]) and G eq sub<G|g1,g2,g3,x,Cg1g2g3 meet G>];
      for g4 in I3 do
        Cg1g2g3g4:=Centralizer(Cg1g2g3,g4);
        I4:=InvolutionRepresentatives(Cg1g2g3 meet G,Cg1g2g3g4);I4:=[x:x in I4|isirredundant(sub<G|g1,g2,g3,g4,x>,[g1,g2,g3,g4,x]) and G eq sub<G|g1,g2,g3,g4,x,Cg1g2g3g4 meet G>];
        for g5 in I4 do
          Cg1g2g3g4g5:=Centralizer(Cg1g2g3g4,g5);
          I5:=InvolutionRepresentatives(Cg1g2g3g4 meet G,Cg1g2g3g4g5);I5:=[x:x in I5|isirredundant(sub<G|g1,g2,g3,g4,g5,x>,[g1,g2,g3,g4,g5,x]) and G eq sub<G|g1,g2,g3,g4,g5,x,Cg1g2g3g4g5 meet G>];
          for g6 in I5 do
            Cg1g2g3g4g5g6:=Centralizer(Cg1g2g3g4g5,g6);
            I6:=InvolutionRepresentatives(Cg1g2g3g4g5 meet G,Cg1g2g3g4g5g6);I6:=[x:x in I6|isirredundant(sub<G|g1,g2,g3,g4,g5,g6,x>,[g1,g2,g3,g4,g5,g6,x]) and G eq sub<G|g1,g2,g3,g4,g5,g6,x,Cg1g2g3g4g5g6 meet G>];
            for g7 in I6 do
              Cg1g2g3g4g5g6g7:=Centralizer(Cg1g2g3g4g5g6,g7);
              I7:=InvolutionRepresentatives(Cg1g2g3g4g5g6 meet G,Cg1g2g3g4g5g6g7);I7:=[x:x in I7|isirredundant(sub<G|g1,g2,g3,g4,g5,g6,g7,x>,[g1,g2,g3,g4,g5,g6,g7,x]) and G eq sub<G|g1,g2,g3,g4,g5,g6,g7,x,Cg1g2g3g4g5g6g7 meet G>];
              for g8 in I7 do
                Cg1g2g3g4g5g6g7g8:=Centralizer(Cg1g2g3g4g5g6g7,g8);
                I8:=InvolutionRepresentatives(Cg1g2g3g4g5g6g7 meet G,Cg1g2g3g4g5g6g7g8);I8:=[x:x in I8|isirredundant(sub<G|g1,g2,g3,g4,g5,g6,g7,g8,x>,[g1,g2,g3,g4,g5,g6,g7,g8,x]) and G eq sub<G|g1,g2,g3,g4,g5,g6,g7,g8,x,Cg1g2g3g4g5g6g7g8 meet G>];
                  for g9 in I8 do
                Cg1g2g3g4g5g6g7g8g9:=Centralizer(Cg1g2g3g4g5g6g7g8,g9);
                I9:=InvolutionRepresentatives(Cg1g2g3g4g5g6g7g8 meet G,Cg1g2g3g4g5g6g7g8g9);I9:=[x:x in I9|isirredundant(sub<G|g1,g2,g3,g4,g5,g6,g7,g8,g9,x>,[g1,g2,g3,g4,g5,g6,g7,g8,g9,x])];
                for g10 in I9 do if G eq sub<G|g1,g2,g3,g4,g5,g6,g7,g8,g9,g10> then Include(~Answers,[g1,g2,g3,g4,g5,g6,g7,g8,g9,g10]);break g1;end if;end for;
               end for;     
              end for;
            end for;  
          end for;
        end for;  
      end for;
    end for;
  end for;  
end for;  
return(Answers);
end function;

/*given a finite permutation group G which is a SGGI with generators [g1,...,g_\ell], the function checks if the corresponding Schreier graph admits a 2-fracture graph*/
has2fracturegraph:=function(G,lst)
local d,V,E,i,Gi,Oi,Ci,e,O;
d:=#lst;
e:=0;
for i in [1..d] do
 Gi:=sub<G|[lst[j]:j in [1..d]|not j eq i]>;
 Oi:=Orbits(Gi);
 Ci:={o:o in Orbits(sub<G|lst[i]>)|#o eq 2};
 if #{o:o in Ci|1 in {#(o meet O):O in Oi}} ge 2 then e:=e+1;else break i;end if;
end for;
return(e eq d);
end function;



/*given a finite permutation group G which is a SGGI with generators [g1,...,g_\ell], the function returns all 2-fracture graphs for the given set of generators*/
Twofracturegraphs:=function(G,lst)
local d,V,E,i,Gi,Oi,Ci,e,O,CCi,Ans,GGraphs;
d:=#lst;
GGraphs:={};
Ans:=[];
e:=0;
for i in [1..d] do
 Gi:=sub<G|[lst[j]:j in [1..d]|not j eq i]>;
 Oi:=Orbits(Gi);
 Ci:={Set(o):o in Orbits(sub<G|lst[i]>)|#o eq 2};
 CCi:={o:o in Ci|1 in {#(o meet O):O in Oi}}; if #CCi ge 2 then e:=e+1;Append(~Ans,CCi);else break i;end if;
end for;
if e eq d then
GGraphs:=CartesianProduct([Subsets(Ans[i],2):i in [1..d]]);
GGraphs:={Graph<{1..Degree(G)}|&join{c[i]:i in [1..d]}>: c in GGraphs};
else 
GGraphs:={};
end if;
return(GGraphs);
end function;


/*code for Corollary 2.2*/
for n in [1..8] do
for g in [1..#TransitiveGroups(n)] do
  G:=TransitiveGroup(n,g);
  if sggi2(G) then
    S:=[[x,y]:x,y in G|sub<G|x,y> eq G and Order(x) eq 2 and Order(y) eq 2];
    U:={};
    for x in S do 
     UU:=Twofracturegraphs(G,x);
     UU:={Gr:Gr in UU|IsConnected(Gr) and (not IsTree(Gr))};
     U:=U join UU;
    end for;
    if #U ge 1 then 2,n,g;end if; 
  end if;   
  if #G ge 2^3 and #sggi3(G) ge 1then
    S:=Sggi3(G);
    U:={};
    for x in S do 
     UU:=Twofracturegraphs(G,x);
     UU:={Gr:Gr in UU|IsConnected(Gr) and (not IsTree(Gr))};
     U:=U join UU;
    end for;
    if #U ge 1 then 3,n,g;end if; 
  end if;   
  if #G ge 2^4 and #sggi4(G) ge 1then
    S:=Sggi4(G);
    U:={};
    for x in S do 
     UU:=Twofracturegraphs(G,x);
     UU:={Gr:Gr in UU|IsConnected(Gr) and (not IsTree(Gr))};
     U:=U join UU;
    end for;
    if #U ge 1 then 4,n,g;end if; 
  end if;   
  if #G ge 2^5 and #sggi5(G) ge 1then
    S:=Sggi5(G);
    U:={};
    for x in S do 
     UU:=Twofracturegraphs(G,x);
     UU:={Gr:Gr in UU|IsConnected(Gr) and (not IsTree(Gr))};
     U:=U join UU;
    end for;
    if #U ge 1 then 5,n,g;end if; 
  end if;   
  if #G ge 2^6 and #sggi6(G) ge 1then
    S:=Sggi6(G);
    U:={};
    for x in S do 
     UU:=Twofracturegraphs(G,x);
     UU:={Gr:Gr in UU|IsConnected(Gr) and (not IsTree(Gr))};
     U:=U join UU;
    end for;
    if #U ge 1 then 6,n,g;end if; 
  end if; 
end for;
end for;

/*code for Proposition 3.1*/
for n in [1..8] do
if IsOdd(n) then
for g in [1..#TransitiveGroups(n)] do
  G:=TransitiveGroup(n,g);
  r:=(n-1)/2;
  if r eq 0 and sggi0(G) then n,g;end if;
  if r eq 1 and sggi1(G) then n,g;end if;
  if r eq 2 and sggi2(G) then n,g;end if;
  if r eq 3 then
    V:=[x:x in Sggi3(G)|has2fracturegraph(G,x)];n,g,V;end if;
end for;end if;end for;

\end{lstlisting}


\bibliographystyle{amsplain}

\end{document}